\documentclass[a4paper,11pt]{amsart}
\usepackage[utf8x]{inputenc}
\usepackage[margin=1.25in]{geometry}
\usepackage[T1]{fontenc}
\usepackage{ae,aecompl}
\usepackage{tikz-cd}
\usepackage{mathrsfs} 
\usepackage{verbatim,amsmath,amsthm,amsfonts,amssymb,latexsym,graphicx,mathtools,extpfeil,color,mathabx}
\usepackage{tikz}
\usetikzlibrary{cd}
\usepackage[all]{xy}
\usepackage{graphicx}
\usepackage{booktabs}
\usepackage{caption}
\usepackage{tabularx}
\usepackage{subcaption}
\usepackage{float}
\usepackage{pdflscape}
\usepackage{cancel}
\usepackage{slashed}
\DeclareMathAlphabet{\mathpzc}{OT1}{pzc}{m}{it}

\usepackage[colorlinks,pagebackref,hypertexnames=false]{hyperref} \usepackage[alphabetic,backrefs,msc-links]{amsrefs}

\usepackage{aliascnt}
\numberwithin{equation}{section}

\renewcommand{\epsilon}{\varepsilon}

\def\({\mathopen{}\left(}
\def\){\right)\mathclose{}}
\def\<{\mathopen{}\left<}
\def\>{\right>\mathclose{}}

\usepackage{multicol, color}

\definecolor{gold}{rgb}{0.85,.66,0}
\definecolor{cherry}{rgb}{0.9,.1,.2}
\definecolor{burgundy}{rgb}{0.8,.2,.2}
\definecolor{orangered}{rgb}{0.85,.3,0}
\definecolor{orange}{rgb}{0.85,.4,0}
\definecolor{olive}{rgb}{.45,.4,0}
\definecolor{lime}{rgb}{.6,.9,0}
\definecolor{green}{rgb}{.2,.7,0}
\definecolor{grey}{rgb}{.4,.4,.2}
\definecolor{brown}{rgb}{.4,.3,.1}

\def\makeautorefname#1#2{\AtBeginDocument{\expandafter\def\csname#1autorefname\endcsname{#2}}}

\newcommand{\mynewtheorem}[2]{
  \newaliascnt{#1}{equation}          
  \newtheorem{#1}[#1]{#2}
  \aliascntresetthe{#1}
  \makeautorefname{#1}{#2}
}
\mynewtheorem{theorem}{Theorem}
\mynewtheorem{prop}{Proposition}
\mynewtheorem{cor}{Corollary}
\mynewtheorem{construction}{Construction}
\mynewtheorem{lemma}{Lemma}
\mynewtheorem{conjecture}{Conjecture}
\mynewtheorem{shitu}{Situation}
\mynewtheorem{conds}{Condition}

\numberwithin{substep}{step}
\makeautorefname{step}{Step}
\makeautorefname{substep}{Step}

\numberwithin{subcase}{case}
\makeautorefname{case}{Case}
\makeautorefname{subcase}{case}

\theoremstyle{remark}
\mynewtheorem{remark}{Remark}

\theoremstyle{definition}
\mynewtheorem{definition}{Definition}
\mynewtheorem{example}{Example}
\mynewtheorem{exercise}{Exercise}
\mynewtheorem{convention}{Convention}
\newtheorem*{convention*}{Convention}
\newtheorem*{conventions*}{Conventions}
\mynewtheorem{question}{Question}
        
\makeautorefname{chapter}{Chapter}
\makeautorefname{section}{Section}
\makeautorefname{subsection}{Section}
\makeautorefname{subsubsection}{Section}

\theoremstyle{introthm}
\newtheorem{introthm}{Theorem}

\newcommand\bbZ{{\mathbb Z}}
\newcommand\bbQ{{\mathbb Q}}
\newcommand\bbR{{\mathbb R}}
\newcommand\bbC{{\mathbb C}}

\title[Lagrangian Floer theory in divisor complement]{Monotone Lagrangian Floer theory in smooth divisor complements: III}
\author{Aliakbar Daemi, Kenji Fukaya}
\thanks{The work of the first author was supported by NSF Grant DMS-1812033. The work of the second author was supported by NSF Grant DMS-1406423 and the Simons Foundation through its Homological Mirror Symmetry Collaboration grant.}

\date{}

\begin{document}
\address{Department of Mathematics and Statistics, Washington University in St. Louis, MO 63130}\email{adaemi@wustl.edu}
\address{Simons Center for Geometry and Physics, State University of New York, Stony Brook, NY 11794-3636, USA} \email{kfukaya@scgp.stonytbrook.edu}

\begin{abstract}
This is the third paper in a series of papers studying intersection Floer theory of Lagrangians in the complement of a smooth divisor.
We complete the construction of Floer homology for such Lagrangians.
\end{abstract}
\maketitle
{
  \hypersetup{linkcolor=black}
  \tableofcontents
}

\section{Introduction}
Let $(X,\omega)$ be a compact symplectic manifold and $\mathcal D \subset X$ be a smooth divisor. That is to say, $\mathcal D$ is a codimension $2$  closed submanifold of $X$ such that the restriction of $\omega$ to $\mathcal D$ is non-degenerate.
Before stating our main theorem, which is a slightly stronger version of \cite[Theorem 1]{part1:top}, we recall the definition of some basic ingredients of the theorem.
\begin{definition}
	Let $L_0, L_1$ be compact subspaces of $X\setminus \mathcal D$.
	We say $L_0$ is {\it Hamiltonian isotopic to $L_1$ in $X\setminus \mathcal D$} if there exists a 
	compactly supported time dependent Hamiltonian 
	$H : (X \setminus \mathcal D) \times [0,1] \to \bbR$ such that the Hamiltonian 
	diffeomorphism $\psi_H : X \setminus \mathcal D \to X \setminus \mathcal D$
	maps $L_0$ to $L_1$.
	Here $\psi_H$ is defined as follows.
	Let $H_t(x) = H(x,t)$ and $X_{H_t}$ be the Hamiltonian vector field associated to $H_t$. 
	We define $\psi^H_t$ by 
	\[
	  \psi^H_0(x) = x, \qquad \frac{d}{dt}\psi^H_t = X_{H_t} \circ \psi^H_t.
	\]
	Then $\psi_H := \psi^H_1$. We say that $\psi_H$ is the {\it Hamiltonian diffeomorphism associated to the 
	(non-autonomous) Hamiltonian $H$}.
\end{definition}
\begin{definition}
	Let $\Lambda^{\bbQ}$ denote the {\it Novikov field} of all formal sums
	\[
	  \sum_{i=1}^{\infty} a_i T^{\lambda_i},
	\]
	where $a_i \in \bbQ$, $\lambda_i \in \bbR$, $\lambda_i < \lambda_{i+1}$ and
	$\lim_{i\to \infty} \lambda_i = +\infty.$ 
\end{definition}
Now we state:

\begin{introthm}\label{mainthm-part3}
	Let $L_0,L_1\subset X\setminus \mathcal D$ be compact, oriented and spin Lagrangian submanifolds 
	such that they are monotone in $X\setminus \mathcal D$.
	Suppose one of the following conditions holds:
	\begin{enumerate}
		\item[(a)] The minimal Maslov numbers of $L_0$ and of $L_1$ are both strictly greater than 2.
		\item[(b)] $L_1$  is Hamiltonian isotopic to $L_0$. 
	\end{enumerate}
	Then we can define Floer homology group $HF(L_1,L_0;X\setminus \mathcal D)$,
	which is a $\Lambda^{\bbQ}$-vector space,
	and satisfies the following properties.
\begin{enumerate}
	\item If $L_0$ is transversal to $L_1$ then we have
		\[
		  {\rm rank}_{\Lambda^{\bbQ}} HF(L_1,L_0;X\setminus D) \le \# (L_0 \cap L_1).
		\]
	\item If $L_i'$ is Hamiltonian isotopic to $L_i$ in $X \setminus \mathcal D$ for $i=0,1$, then
		\[
		  HF(L_1,L_0;X\setminus \mathcal D) \cong HF(L'_1,L'_0;X\setminus \mathcal D).
		\]
	\item If we assume either {\rm (b)} or $\pi_1(L_0) = \pi_1(L_1) = 0$, then we can replace $\Lambda^{\bbQ}$
	with $\bbQ$, the field of rational numbers.
	\item If $L_0 = L_1 =L$, then there exists a spectral sequence 
		whose $E^2$ page is the singular homology group $H_*(L;\bbQ)$ of $L$ and which converges to 
		$HF(L,L;X\setminus \mathcal D)$.
\end{enumerate}
\end{introthm}
See Condition \ref{cond420} (2) for the definition of the monotonicity of Lagrangians in $X\setminus \mathcal D$.
\begin{remark}
	The above theorem can be strengthened in various directions. 
	In Theorem \ref{mainthm-part3}, we assumed $L_i$ is spin. We can relax this condition to the condition that 
	$L_i \subset X \setminus \mathcal D$ is relatively spin and $(L_1,L_0)$ forms a relatively spin pair in $X\setminus \mathcal D$.
	(See \cite[Definition 3.1.1]{fooobook} for the definition of relatively spin Lagrangians and 
	relatively spin pairs.)
		
	Let $L$ be a monotone Lagrangian in $X\setminus \mathcal D$. Associated to any such Lagrangian,
	there is an element $\frak{PO}_{L}(1)\in \Lambda^{\bbQ}$, which vanishes if minimal Maslov number of $L$ is not equal to $2$. 
	The assumption in Item (a) can be also weakened to the assumption that $\frak{PO}_{L_0}(1)=\frak{PO}_{L_1}(1)$.
	We can also replace the assumption in Item (3), with the weaker assumption stated in Condition \ref{cond420} (2).
\end{remark}

Using the construction of Kuranishi structures in \cite{part2:kura} 
on the moduli spaces 
we introduced in \cite{part1:top}, the proof of Theorem \ref{mainthm-part3}
is  in principle 
a modification of the arguments of the construction of Lagrangian Floer theory 
based on virtual fundamental chain techniques. (See, for example, \cite{fooobook2,fooonewbook}.)
It is also a variant of Oh's construction of Floer homology of monotone Lagrangian submanifolds \cite{oh}.
\par
There are, however, a few subtle technical points that need to be sorted out for the construction of our version of Floer homology. We highlight two of these points, which are more important than the other ones. The first one is related to the description of the boundary 
of the moduli spaces that are used in the definition of Lagrangian Floer homology. We will discuss this point in Sections \ref{sec:minimas3} and \ref{sub:systemconst}. In our situation, the boundaries of the moduli spaces involved are {\it not} given 
as fiber products of the moduli spaces with smaller energy, but they only accept such descriptions outside codimension two strata.  This is a new feature which does not seem to appear previously 
in the literature.

The second point is related to compatibility of our Kuranishi structures with certain forgetful maps. This point, which we will study in Section \ref{sec:forget}, needs to be addressed only when there are Maslov index 2 pseudo-holomorphic disks. Oh studied 
a similar problem in \cite{oh}, where he could use perturbation of almost complex structures to achieve transversality for the moduli spaces of pseudo-holomorphic curves. Therefore, compatibility with forgetful maps is immediate in his setup. Even though we are studying monotone Lagrangian submanifolds, we have to use virtual fundamental chain techniques and abstract perturbations.
For this reason, compatibility with forgetful maps becomes a much more nontrivial problem. 
\par
The proof of Theorem \ref{mainthm-part3} is completed in Section \ref{sec:mainthmproof}.
In Subsection \ref{subsec:welldef} the 
independence of Floer homology from auxiliary choices is proved. In Sections \ref{sec:minimas3}, \ref{sub:systemconst}, \ref{sec:forget}
we assume that not only $L_0$ and $L_1$ but also the pair $(L_0,L_1)$ satisfies 
a certain monotonicity condition (Condition \ref{cond420} (2)).
In Subsection \ref{subsec:novikv} we remove this condition
by using Novikov ring.
Finally a spectal sequence calculating the Floer homology in case $L_1=L_0$ is defied in Subsection \ref{thecaseofLandL}.
Since the arguments of Section \ref{sec:mainthmproof} are modifications of standard proofs, we sketch the proofs without going into details.

The final section of the paper is devoted to various possible extensions of our main result. In particular, we explain how we expect that one can associate an $A_\infty$-category to a family of Lagrangians in the complement of a smooth divisor, generalize the construction of this paper to the case of normal crossing divisors and define an equivariant version of Lagrangian Floer homology of Theorem \ref{mainthm-part3}. Perhaps the most interesting direction among the proposals in Section \ref{secgeneralization} is the extension of the present theory to the case of non-compact Lagrangians. In this part, we present some formal similarities between the proposed theory for non-compact Lagrangians and monopole Floer homology in \cite{km:monopole}.

\begin{remark}
We remark that the proofs of various claims of this paper require a study of orientations of moduli spaces and the resulting signs in the formulas in several places. However, there is nothing novel about signs happening in our case. The key point here is that we only need to orient the part of the moduli space that is obtained from the zero set of perturbed multi-sections. Even though we use the existence of Kuranishi structures on the strata with arbitrary high codimension in our construction, we may arrange that the zero set of perturbed multi-sections consists of pseudo-holomorphic maps that do not intersect divisor. This allows us to adapt standard arguments to our setup without any new issue. Because of this, we skip the discussion of orientations and signs, and refer the reader to \cite[Section 8]{fooobook2} for more details.
\end{remark}


\section{Floer homology when minimal Maslov number 
is greater than 2.}\label{sec:minimas3}

\subsection{The Definition of Floer Homology}
\label{boundaryope}

Let $L_0,L_1$ be Lagrangian submanifolds in $X \setminus \mathcal D$.
We assume they are monotone in $X \setminus \mathcal D$ and 
are relatively spin there.
We assume that $L_0$ is transversal to $L_1$.
In Sections \ref{sec:minimas3} and \ref{sub:systemconst}, we firstly work under the assumption of Condition \ref{cond420} below, where we can use $\bbQ$ as the coefficient ring rather than the Novikov ring.
We need to prepare some notations to state this condition.
\begin{definition}
	Let $\Omega(L_0,L_1)$ denote the space of 
        all continuous maps $\gamma : [0,1] \to X$
        with $\gamma(0) \in L_0$, $\gamma(1) \in L_1$. 
        For $o \in \pi_0(\Omega(L_0,L_1))$, we denote 
        the fundamental group of the corresponding 
        connected component by $\pi_1(\Omega(L_0,L_1);o)$.
        We also write $L_0 \cap_o L_1$ for the subset of $L_0 \cap L_1$ consisting 
        of $p$ such that the constant map $\gamma(t) \equiv p$ belongs to the 
        connected component $o$.
\end{definition}
\begin{definition}      
        An element of $\pi_1(\Omega(L_0,L_1);o)$ determines a relative second homology class
        $H_2(X,L_0\cup L_1;\bbZ)$. Since $(L_0 \cup L_1) \cap \mathcal D = \emptyset$, we obtain
        a map:
        $$
        \cap \mathcal D : \pi_1(\Omega(L_0,L_1);o) \to \bbZ.
        $$
	We may also integrate $\omega$ on homology classes represented by elements of $\pi_1(\Omega(L_0,L_1);o)$
	to obtain another homomorphism
	\[
	  \omega : \pi_1(\Omega(L_0,L_1);o) \to \bbR.
	\]
	Finally, we have the Maslov index homomorphism
	\[
	  \mu : \pi_1(\Omega(L_0,L_1);o) \to \bbZ.
	\]
	(See, for example, \cite[p50]{fooobook}, where it is denoted by $I_{\mu}$.) 
\end{definition}

\begin{conds}\label{cond420}
	Let $o \in \pi_0(\Omega(L_0,L_1))$ be given. We require the following two properties:
	\begin{enumerate}
		\item The minimum Maslov numbers of $L_0$ and $L_1$ are not smaller than 4.
		\item There is a constant $c$ such that for any 
			$\alpha \in \pi_1(\Omega(L_0,L_1),o)$ with $\alpha \cap \mathcal D = 0$, we have
			\begin{equation}
				 \omega(\alpha)=c\mu(\alpha).
			\end{equation}
	\end{enumerate}
\end{conds}
\begin{lemma} \label{index-energy-strip}
	We assume Condition \ref{cond420} {\rm (2)} for $o \in \pi_0(\Omega(L_0,L_1))$ holds. 
	Then for each $p,q \in L_0 \cap_o L_1$, there exists $c(p,q)$ such that: 
	\begin{equation}\label{monotonicityforpair}
		 \omega(\beta)=c\mu(\beta) - c(p,q),
	\end{equation}
	for any $\beta \in \Pi_2(X;L_0,L_1;p,q)$ with $\beta\cap \mathcal D=0$.
\end{lemma}
\begin{proof}
	Let $\beta_1,\beta_2 \in \Pi_2(X;L_0,L_1;p,q)$. Since $\beta_2$ is represented by a map $\bbR \times [0,1] \to X$,
	we can compose it with $(\tau,t) \mapsto (-\tau,t)$ and obtain $-\beta_2 \in \Pi_2(X;L_0,L_1;q,p)$.
	We concatenate $\beta_1$ and $-\beta_2$ in an obvious way 
	to obtain $\beta_1 \# -\beta_2 \in  \pi_1(\Omega(L_0,L_1);o)$.
	It follows easily from definition that
	$$
	\mu(\beta_1) - \mu(\beta_2) = \mu(\beta_1 \# -\beta_2),\quad\omega(\beta_1) - \omega(\beta_2) = 
	\omega(\beta_1 \# -	\beta_2).
	$$
	The lemma follows easily from these identities.
\end{proof} 

\begin{lemma}
	For $i=0$ or $1$, let $\pi_1(L_i)$ be trivial. 
	Then Condition \ref{cond420} {\rm (2)} holds for any $o$.
\end{lemma}
\begin{proof}
        Let $\alpha \in \pi_1(\Omega(L_0,L_1);o)$.
        It is represented by a map $u : S^1 \times [0,1] \to X$
        with $u(S^1 \times\{i\}) \subset L_i$
        for $i=0,1$.
        We assume $\pi_1(L_0) = 0$. Then 
        there exists a map $v : D^2 \to L_0$
        such that $v\vert_{S^1} = u_{S^1 \times \{0\}}$.
        We glue $v$ and $u$ to obtain $v'$. We have:
        $$
        \mu([v]) + \mu(\alpha) = \mu([v']) \quad
        \omega([v]) + \omega(\alpha) = \omega([v']).
        $$ 
	Since $ \mu([v])=\omega([v])=0$, Condition \ref{cond420} (2) follows from the monotonicity of $L_i$.
\end{proof}

To prove Theorem \ref{mainthm-part3}, we use moduli spaces $\mathcal M^{\rm RGW}_{k_1,k_0}(L_1,L_0;p,q;\beta)$ and $\mathcal M^{\rm RGW}_k(L;\alpha)$ introduced in \cite{part1:top} and the construction of Kuranishi structures on these spaces obtained in \cite{part2:kura}.  These moduli spaces are defined as the union of spaces which are parametrized by combinatorial objects called {\it SD-ribbon trees} and {\it DD-ribbon trees}. We may evaluate a pseudo-holomorphic curve representing an element of the moduli space $\mathcal M^{\rm RGW}_{k_1,k_0}(L_1,L_0;p,q;\beta)$ at the boundary marked points to obtain the evaluation maps
\[
  {\rm ev}^{\partial}_{i,j}:\mathcal M^{\rm RGW}_{k_1,k_0}(L_1,L_0;p,q;\beta) \to L_i,
\]
where $i\in \{0,1\}$ and $j\in \{1,2,\dots,k_i\}$. Similarly, we have evaluation maps 
\[
  {\rm ev}^{\partial}_j : \mathcal M^{\rm RGW}_k(L;\alpha) \to L,
\]
for $j\in \{0,1,2,\dots,k\}$. We assume familiarity with the definitions of these notions and refer the reader to \cite{part1:top,part2:kura} for more details. We also need the following result. 
\begin{theorem}\label{theorem30}
	For any $\beta \in \Pi_{2}(X;L_1,L_0;p,q)$,
	then there is a Kuranishi structure on $\mathcal M^{\rm RGW}_{k_1,k_2}(L_1,L_0;p,q;\beta)$ such that the normalized boundary\footnote{
	See 	\cite[Definition 8.4]{fooo:tech2-1}for the definition.} of this space
	is the union of the following three types of `fiber or direct products' (as spaces with Kuranishi structure).
	\begin{enumerate}
		\item
			$$\mathcal M^{\rm RGW}_{k'_1,k'_2}(L_1,L_0;p,r;\beta_1)\, \hat\times\, \mathcal M^{\rm RGW}_{k''_1,k''_2}(L_1,L_0;r,q;\beta_2).
			$$
			Here $r \in L_0\cap L_1$, $k'_1+k''_1 = k_1$,
			$k'_2+k''_2 = k_2$, $\beta_1 \in \Pi_2(X,L_1,L_0;p,r)$ and $\beta_2 \in \Pi_2(X,L_1,L_0;r,q)$
			such that $\beta_1 \# \beta_2 = \beta$ and $\beta_1 \cdot [\mathcal D] = \beta_2 \cdot [\mathcal D] = 0$.(
			The symbol $\hat\times$ will be discussed in Subsection \ref{c}.\footnote{Roughly speaking, it is the ordinary direct product outside 
			a union of codimension $2$ strata, and the obstruction bundle on the complement of these codimension $2$ strata is given by the 
			product of the pulled-back of obstruction bundles of product summands.})
		\item
			$$\mathcal M^{\rm RGW}_{k'_1,k_2}(L_1,L_0;p,q;\beta')\,\hat\times_{L_1}\, \mathcal M^{\rm RGW}_{k''_1+1}(L_1;\alpha).$$
			Here $k'_1+k''_1 = k_1+1$, $\beta' \in \Pi_2(X,L_1,L_0;p,q)$, $\alpha \in \Pi_2(X,L_1;\bbZ)$ form a pair  such that 
			$\beta' \# \alpha = \beta$ and $\beta' \cdot [\mathcal D] = \alpha \cdot [\mathcal D] = 0.$
			The  `fiber product' $\hat\times_{L_1}$ is defined  using ${\rm ev}^{\partial}_{1,i}$,
			for $i = 1,\dots,k'_1$, and ${\rm ev}^{\partial}_0$. 
			(Here $\hat\times_{L_1}$ is slightly different from the ordinary fiber product. See Subsection \ref{c}.)
		\item
			$$\mathcal M^{\rm RGW}_{k_1,k'_2}(L_1,L_0;p,q;\beta')\, \hat\times_{L_0}\, \mathcal M^{\rm RGW}_{k''_2}(L_0;\alpha).$$
			Here $k'_2+k''_2 = k_2+1$, $\beta' \in \Pi_2(X,L_1,L_0;p,q)$,
			$\alpha \in \Pi_2(X,L_0;\bbZ)$ form a pair such that $\beta' \# \alpha = \beta$ and $
			\beta' \cdot [\mathcal D] = \alpha \cdot [\mathcal D] = 0.$
			The `fiber product' $\hat\times_{L_0}$ is defined using ${\rm ev}^{\partial}_{0,i}$, for $i=1,\dots,k'_2$, and ${\rm ev}^{\partial}_0$.
	\end{enumerate}
	Moreover, these Kuranishi structures are orientable and 
	the above isomorphisms are orientation preserving if $L_0$ and $L_1$ are (relatively) spin.
\end{theorem}
We have a similar compatibility of the Kuranishi structures 
at the boundary of $\mathcal M^{\rm RGW}_k(L;\alpha)$. See Theorem \ref{lema362rev}.

We use the next result to find appropriate system of perturbations 
to define the boundary operators.
Note that we do {\it not} assume Condition \ref{cond420} in Theorem \ref{prop61111}.
\begin{theorem}\label{prop61111}
	Let $L_0,L_1 \subset X \setminus \mathcal D$ be a  pair of compact Lagrangian submanifolds.
	We assume that $L_0$ is transversal to $L_1$. Let $E$ be a positive number.
	Then there exists a system of multi-valued perturbations $\{\frak s_{n}\}$ \footnote{ 
	Compared to \cite[Definition 6.12]{fooonewbook}, we slightly modify the definition, 
	that is, we require only $C^1$ or $C^0$ smoothness of multi-section. } on
	the moduli spaces $\mathcal M^{\rm RGW}_{k_1,k_0}(L_1,L_0;p,q;\beta)$
	of virtual dimension at most $1$ and $\omega(\beta) \le E$ such that the following holds.
\begin{enumerate}
	\item  The multi-section $\frak s_{n}$ is $C^0$ and is $C^1$ in a neighborhood of
		$\frak s_{n}^{-1}(0)$. The multi-sections $\frak s_{n}$ are transversal to $0$.
		The sequence of multi-sections $\frak s_n$ converges to the Kuranishi map in $C^0$. 
		Moreover, this convergence is in $C^1$ in a neighborhood of the zero locus of the Kuranishi map.
	\item The multi-valued perturbations $\{\frak s_{n}\}$ are compatible with the description of the 
		boundary given by Theorem \ref{theorem30}. 
	\item Suppose that the (virtual) dimension of 
                $\mathcal M^{\rm RGW}_{k_1,k_0}(L_1,L_0;p,q;\beta)$
                is not greater than $1$. Then the multisection 
                $\frak s_n$ does not vanish on the codimension $2$ stratum
                $\mathcal M^{\rm RGW}_{k_1,k_0}(L_1,L_0;p,q;\beta)^{(1)}$
                 described by Proposition \ref{prop356}.
\end{enumerate}
\end{theorem}

The proof of Theorem \ref{prop61111} will be given in Subsection \ref{subsub:constmulti}. 
We  observe that the next statement follows immediately.
\begin{cor}\label{Cor27}
Let $L_i$  and $\{\frak s_{n}\}$ be as in Theorem \ref{prop61111}.
For given $E$ and sufficiently large $n$ (depending on $E$) the following 
holds.
\begin{enumerate}
	\item If the (virtual) dimension of 
                $\mathcal M^{\rm RGW}_{k_1,k_0}(L_1,L_0;p,q;\beta)$
                is negative then the multi-valued perturbation $\frak s_{n}$ has no zero on it.
	\item If the (virtual) dimension of 
                $\mathcal M^{\rm RGW}_{k_1,k_0}(L_1,L_0;p,q;\beta)$ 
                is $0$, then  the multi-valued perturbation $\frak s_{n}$ has only a finite number of zeros.
	\item If the (virtual) dimension of 
                $\mathcal M^{\rm RGW}_{k_1,k_0}(L_1,L_0;p,q;\beta)$ 
                is $1$, then the set of zeros of multi-valued perturbation $\frak s_{n}$ can be compactified
                to an orientated 1-dimensional chain. Its boundary is the 
                zero set of the multi-valued perturbation on the space described by 
                Theorem \ref{theorem30}
                and Item (2).
\end{enumerate}
\end{cor}

\begin{definition}
Let $o \in \pi_0(\Omega(L_0,L_1))$.
We define
$$
CF(L_1,L_0;\bbQ;o) = \bigcup_{p \in L_0 \cap_o L_1} \bbQ [p].
$$
We also define
$$
CF(L_1,L_0;\bbQ) = \bigoplus_{o \in \pi_0(\Omega(L_0,L_1))}  CF(L_1,L_0;\bbQ;o).
$$
\end{definition}
The following lemma can be easily verified using the dimension formulas.
\begin{lemma}\label{lem460}
        If Condition \ref{cond420} is satisfied for $o \in \pi_0(\Omega(L_0,L_1))$, then there exists $E$ such that if the moduli space $\mathcal M_{k+1}^{\rm RGW}(L_i;\beta)$ or 
        $\mathcal M^{\rm RGW}_{k_1,k_0}(L_1,L_0;p,q;\beta)$ ($p,q \in L_0 \cap_o L_1$)
        has virtual dimension not greater than $1$, then we have
        \[
          \omega(\beta) \le E.
        \]
\end{lemma}
We take $E$ as in Lemma \ref{lem460} and apply 
Corollary \ref{Cor27} to  obtain a 
system of multi-valued perturbations $\{\widehat{\frak s}^{n}\}$.
If $\dim \mathcal M^{\rm RGW}_{0,0}(L_1,L_0;p,q;\beta) = 0$,
then by Item (2) of Corollary \ref{Cor27}, we have a rational number
$$
[\mathcal M^{\rm RGW}_{0,0}(L_1,L_0;p,q;\beta),\widehat{\frak s}^n] \in \bbQ
$$
for sufficiently large values of  $n$. We can use compactness of stable map compactification to show
that there exists only a finite number of $\beta$ with 
$\omega(\beta) \le E$ such that the moduli space $\mathcal M^{\rm RGW}_{0,0}(L_1,L_0;p,q;\beta)$ is non-empty.

\begin{definition}\label{defn2.10}
	Assume Item(2) of Condition \ref{cond420} holds. We define
	\begin{equation}\label{differential-Floer}
		\langle \partial p,q\rangle= \sum_{\beta}[\mathcal M^{\rm RGW}_{0,0}(L_1,L_0;p,q;\beta),\widehat{\frak s}^n]  \in \bbQ.
	\end{equation}
	Here the sum is taken over all $\beta$ that $\dim (\mathcal M^{\rm RGW}_{0,0}(L_1,L_0;p,q;\beta)) = 0$.
	We  define the Floer's boundary operator $\partial : CF(L_1,L_0;\bbQ) \to CF(L_1,L_0;\bbQ)$ by
	\begin{equation}\label{form43}
		\partial [p] = \sum_q \langle \partial p,q\rangle [q].
	\end{equation}
\end{definition}
Note that \eqref{differential-Floer} depends on $n$.  However, we shall show in Theorem \ref{lem48} that $\partial$ is a differential, and then in Subsection \ref{subsec:welldef} we verify that the homology of this differential is independent of $n$.
\begin{theorem}\label{lem48}
If Condition \ref{cond420} is satisfied, then
$$
\partial \circ \partial = 0.
$$
\end{theorem}
We need a more detailed description of the boundary to prove this theorem. We provide this description in Subsection \ref{c} and Section \ref{sub:systemconst}. Using this description, we will prove Theorem \ref{lem48} in Subsection \ref{subsub:Floer2}.

\subsection{Stratifications and Description of the Boundary}
\label{c}

A subtle feature of the RGW compactification and its 
boundary structure is that the inclusion
\begin{equation}\label{eq355}
\aligned
&\mathcal M^{\rm RGW}_{k'_1,k'_0}(L_1,L_0;p,r;\beta_1)
\times
\mathcal M^{\rm RGW}_{k''_1,k''_0}(L_1,L_0;r,q;\beta_2)
\\
&\subset 
\partial \mathcal M^{\rm RGW}_{k_1,k_0}(L_1,L_0;p,q;\beta)
\endaligned
\end{equation}
does {\it not} hold. 
However, this does not cause any problem for our construction of Floer homology, because \eqref{eq355} holds outside a union of codimension 2 strata. This subtlety is also responsible for the notation $\hat\times$ appearing in 
Theorem \ref{theorem30}. The purpose of this subsection is to clarify this point.
We firstly introduce the notion of an even codimension stratification of a Kuranishi structure. We follow the terminology in \cite{fooonewbook} for the definition of Kuranishi structures.

\begin{definition}
	Let $X$ be a space with a Kuranishi structure 
	$\widehat{\mathcal U} = (\{{\mathcal U}_p\}$, $\{\Phi_{pq}\})$, where 
	${\mathcal U}_p = (U_p,E_p,s_p,\psi_p)$ and $\Phi_{pq} = (U_{pq},\hat{\varphi}_{pq},\varphi_{pq})$
	are respectively Kuranishi charts and coordinate changes.
	By shrinking $U_p$, we may assume $U_p = V_p/\Gamma_p$ where $V_p$ is a 
	manifold and $\Gamma_p$ is a finite group.
	
	An {\it even codimension stratification} of $\widehat{\mathcal U}$
	is a choice of a close subset $X^{(n)} \subset X$  for each $n \in \bbZ_{+}$ and $V_p^{(n)}
	\subset V_p$ 
	for each $p \in X^{(n)}$ with the following properties:
	\begin{enumerate}
	\item
	$\overset{\circ}{V}_p^{(n)}: = V_p^{(n)} \setminus V_p^{(n+1)}$ is a codimension $2n$ embedded submanifold of 
	$V_p$. $V_p^{(n)}$ is $\Gamma_p$-invariant. We define $U_p^{(n)} = V_p^{(n)}/\Gamma_p$, 
	$\overset{\circ}U_p^{(n)} = 	\overset{\circ}V_p^{(n)}/\Gamma_p$.
	\item
	$\psi_p(s_p^{-1}(0) \cap \overset{\circ}U_p^{(n)}) \subset X^{(n)}\setminus X^{(n+1)}$.
	\item
	$\varphi_{pq}(U_{pq} \cap U_q^{(n)}) \subset U_p^{(n)}$
	for any $p \in X^{(n)}$, $q \in \psi_p(s_p^{-1}(0) \cap U_p^{(n)})$.
	\item
	$\widehat{\mathcal U}^{(n)}_p = (\overset{\circ}U^{(n)}_p,
	E_p\vert_{\overset{\circ}U^{(n)}_p},s_p\vert_{\overset{\circ}U^{(n)}_p},
	\psi_p\vert_{\overset{\circ}U^{(n)}_p\cap s_p^{-1}(0)})$
	and a similar restriction of coordinate changes $\Phi_{pq}$ define a Kuranishi structure 
	for $X^{(n)} \setminus X^{(n+1)}$.
	In particular, the dimension of $X^{(n)} \setminus X^{(n+1)}$ with respect to this Kuranishi structure
	is $\dim (X,\widehat{\mathcal U}) - 2n$.
	\item
	If $(X,\widehat{\mathcal U})$ has boundary and corner (that is, 
	$U_p$ has boundary and corner) we require that the codimension $k$ corner $S_kU^{(n)}_p$ coincides with 
	$U^{(n)}_p \cap S_kU_p$ where $S_kU_p$ is the codimension $k$ corner of $U_p$.
\end{enumerate}
We call $\{X^{(n)}\}$ the {\it underlying topological stratification}.
\end{definition}

The Kuranishi structure of the RGW compactification has an even codimension stratification in the above sense. For the following definition, recall that $\mathcal M^{\rm RGW}_{k_1,k_0}(L_1,L_0;p,q;\beta)$ is a union of the spaces ${\mathcal M}^{0}(\mathcal R)$ for all RD-ribbon trees $\mathcal R$ of type 
	$(p,q;\beta;k_0,k_1)$ (see \cite[Subsection 3.5]{part1:top}).
\begin{definition}
	We define 
\begin{equation}
\mathcal M^{\rm RGW}_{k_1,k_0}(L_1,L_0;p,q;\beta)^{(n)}
\subset \mathcal M^{\rm RGW}_{k_1,k_0}(L_1,L_0;p,q;\beta)
\end{equation}
as the disjoint union of ${\mathcal M}^{0}(\mathcal R)$ 
such that the total number of positive levels of $\mathcal R$
is not smaller than $n$. Similarly, we can define 
\begin{equation}
\mathcal M^{\rm RGW}_{k+1}(L;\beta)^{(n)}
\subset \mathcal M^{\rm RGW}_{k+1}(L;\beta),
\end{equation}
\begin{equation}
\mathcal M(\mathcal D\subset X;\alpha;{\bf m})^{(n)}
\subset \mathcal M(\mathcal D\subset X;\alpha;{\bf m}),
\end{equation}
where $\mathcal M(\mathcal D\subset X;\alpha;{\bf m})$ is defined in \cite[Subsection 3.3]{part1:top}.
\end{definition}
\begin{prop}\label{prop356}
The Kuranishi structure of 
$\mathcal M^{\rm RGW}_{k_1,k_0}(L_1,L_0;p,q;\beta)$
has an even codimension stratification whose underlying 
topological stratification is given by 
$$
\{\mathcal M^{\rm RGW}_{k_1,k_0}(L_1,L_0;p,q;\beta)^{(n)}
\}.
$$
The same holds for $\mathcal M^{\rm RGW}_{k+1}(L;\beta)$ and 
$\mathcal M(\mathcal D\subset X;\alpha;{\bf m})$.
\end{prop}
\begin{proof}
	The Kuranishi structure on $\mathcal M^{\rm RGW}_{k_1,k_0}(L_1,L_0;p,q;\beta)$ in \cite{part2:kura} is constructed 
	 by relaxing the non-linear Cauchy-Riemann equation $\overline \partial u = 0$
	to $\overline \partial u \equiv 0 \mod E(u)$, where $E(u)$ is an appropriate finite dimensional subspace of $u^*TX$-valued 
	$(0,1)$-forms on the source curve of $u$. Namely, the Kuranishi neighborhood of $U_{\frak p}$
	is the set of the solutions of this relaxed equation.
	Our stratification is given by the combinatorial type 
	of the source curve. So we can stratify 
	$U_{\frak p}$ in the same way using the combinatorial type of the source curve. In order to show that this strafication
	has the required properties, we need to compute the dimension of different subspaces of 
	$\mathcal M^{\rm RGW}_{k_1,k_0}(L_1,L_0;p,q;\beta)$.


	Let $\mathcal R$ be an SD-ribbon tree of type $(p,q;\beta;k_0,k_1)$ with $n$ levels.
	Then the space $\widehat {\mathcal M}^0(\mathcal R) = \widetilde{\mathcal M}^0(\mathcal R)/\bbC_*^n$
	is contained in the codimension $\#\{ v \mid c(v) = {\rm d_0},{\rm d_1},{\rm str}\}-1$ 
	corner of $\mathcal M^{\rm RGW}_{k_1,k_0}(L_1,L_0;p,q;\beta)$. 
	We shall show that:
	\begin{equation}\label{dimension-M0R}
	  \dim(\widetilde{\mathcal M}^0(\mathcal R))=
	  \mu(\beta) + (k_1+k_0) - \#\{ v \mid c(v) = {\rm d_0},{\rm d_1}, {\rm str}\}.
	\end{equation}
	This implies that $\widehat{\mathcal M}^0(\mathcal R)$ has codimension $2n$ in the corner of codimension 
	$\#\{ v \mid c(v) = {\rm d_0},{\rm d_1}, {\rm str}\}-1$. 
	It is straightforward to use this observation to verify the proposition.
	
	If $v$ is a vertex of the detailed tree  $\hat R$ of $\mathcal R$ with color ${\rm D}$, then the 
	corresponding factor  $\widetilde {\mathcal M}^0(\mathcal R;v)$ is:
	\[
	  \widetilde {\mathcal M}^{0}(\mathcal D \subset X;\alpha(v);{\bf m}^{v}).
	\]
	Here ${\bf m}^{v} = (m(e_0(v)),\cdots,m(e_{\ell(v)}(v))$ with 
	$e_0(v),\dots,e_{\ell(v)}(v)$ being the edges incident to $v$.
	Using \cite[Remark 4.69]{part1:top}, we have
\begin{equation}\label{form359}
\aligned
&{\rm dim} \widetilde{\mathcal M}^{0}(\mathcal D \subset X;\alpha(v);{\bf m}^{v}) \\
&=
2(n-1) + 2 \alpha(v) \cdot c_1(\mathcal D) + 2(\ell(v)+1)- 6 + 2.
\endaligned
\end{equation}
	In fact, $2(n-1) + 2 \alpha(v) \cdot c_1(\mathcal D) - 6$ is the 
	dimension of the moduli space of holomorphic spheres in $\mathcal D$ of homology class $\alpha$. 
	Adding $\ell(v)+1$ marked points
	increases the dimension by $2(\ell(v)+1)$. 
	The last term $2$ appears in the formula because our moduli space  
	contains the data of a section $s$. 
	The difference between two choices of $s$ is given by a nonzero complex number.

	If $v$ is a level $0$ vertex with color ${\rm s}$, then the factor   $\widetilde {\mathcal M}^0(\mathcal R;v)$ is
	\[
	  \widetilde {\mathcal M}^{\rm reg, s}(\alpha(v);{\bf m}^{v}),
	\]
	where ${\bf m}^{v}=(-m(e_0(v)),m(e_1(v)),\cdots,m(e_{\ell(v)}(v))$ with $e_0(v),\dots,e_{\ell(v)}(v)$ being the edges in $\hat R$
	that connects $v$ to a vertex with color D. We have
	\begin{align}\label{form361new}
		\phantom{i + j + k}
		&\begin{aligned}
			&\dim \widetilde {\mathcal M}^{\rm reg, s}(\alpha(v);{\bf m}^{v}) \\
			&=2n + 2 \alpha(v) \cdot c_1(X) + 2\sum_{i=0}^{\ell(v)} (1 - |m(e_{i}(v))|)-6.
		\end{aligned}
	\end{align}
	This formula can be verified using the fact that the condition about tangency
	at the $i$-th marked point (\cite[Definition 3.57 (6)]{part1:top}) decreases the dimension 
	by $2(m(e_{i}(v))-1)$ or $-2(m(e_{0}(v))+1)$ depending on whether $i>0$ or $i=0$. 
	recall that $m(e_{0}(v)) < 0 < m(e_{i}(v))$ for $i>0$.

	If $v$ is a level $0$ vertex with color ${\rm d}_0$ or ${\rm d}_1$, then the corresponding factor is
	\[
	  \mathcal M_{k(v)+1}^{\rm reg, d}(\alpha(v);{\bf m}^{v})
	\]
	where ${\bf m}^{v}=(m(e_1(v)),\cdots,m(e_{\ell(v)}(v))$ 
	is defined as in the previous case, and $k(v) + 1$ is the number of level $0$ edges of $v$.
	We have
	\begin{align}\label{361form}
		\phantom{i + j + k}
		&\begin{aligned}
			&\dim \mathcal M_{k(v)+1}^{\rm reg, d}(\alpha(v);{\bf m}^{v}) \\
			&=n + \mu(\alpha(v)) + 2\sum_{i=1}^{\ell(v)} (1 - m(e_{i}(v)))+ k(v)- 2.
		\end{aligned}
	\end{align}
	Here $\mu$ is the Maslov index. (See, for example, \cite[Definition 2.1.15]{fooobook}.)
	
	If $v$ is a level $0$ vertex with color ${\rm str}$, then the corresponding factor is:
	\[
	  \mathcal M_{k_1(v),k_0(v)}^{\rm reg}(L_1,L_0;{\rm pt}(\frak e_\lambda(v)),{\rm pt}(\frak e_r(v));\alpha(v);{\bf m}^{v}).
	\]
	Here $k_i(v)$ is the number of level $0$ edges of $v$ contained in $R_i$, and $\frak e_l(v)$, $\frak e_r(v)$ 
	are the edges of $C$ as in \cite[Definition 3.94 (6)]{part1:top}. Furthermore, the tuple  
	${\bf m}^{v}=(m(e_1(v)),\cdots,m(e_{\ell(v)}(v))$ is defined as in the previous two cases.
	We have
	\begin{align}\label{form362}
		\phantom{i + j + k}
		&\begin{aligned}
			&\dim \mathcal M_{k_1(v),k_0(v)}^{\rm reg}(L_1,L_0;{\rm pt}(\frak e_l(v)),{\rm pt}(\frak e_r(v)); \alpha(v);{\bf m}^{v}) \\
			&=\mu(\alpha(v)) + 2\sum_{i=1}^{\ell(v)} (1- m(e_{i}(v)))+(k_1(v)+k_0(v)) - 1.
		\end{aligned}
	\end{align}
	where $\mu$ is the Maslov index.  (See, for example, \cite[Definition-Proposition 2.3.9]{fooobook}.)

	Now we can compute the dimension of $\widetilde{\mathcal M}^0(\mathcal R)$
	in \cite[(3.75)]{part2:kura} using (\ref{form359})-(\ref{form362}).
	This dimension is the sum of (\ref{form359})-(\ref{form362}) 
	minus $n\cdot \#(C_1^{\rm int}(R) \setminus C_1(C))$ minus $2(n-1)\cdot N_{>0}$.
	Here 
	$N_{>0}$ is the number of all interior edges that are not in $C_1(R)$.
	We denote by $C_1^{0<\lambda<1}(\hat R)$ the set of the edges of $\hat R$ 
	which join level $0$ vertices to positive level vertices.
	Then we have:
	\begin{align}\label{36333666}
		\phantom{i + j + k}
		&\begin{aligned}
		&\!\!\!\!\!\dim (\widetilde{\mathcal M}^0(\mathcal R))
                \\
                = &n\#\{v \mid c(v) = {\rm d_0, d_1}\} + 2(n-1)\#\{v \mid c(v) = {\rm D}\}
                \\
                &+2n\#\{v \mid c(v) = {\rm s}\} \\
                &+ \sum_{c(v) = {\rm d_0,d_1,str}}\mu(\alpha(v))
                + \sum_{c(v) = {\rm s}} 2\alpha(v)\cdot c_1(X) \\
                &+ \sum_{c(v) = {\rm D}} 2\alpha(v) \cdot c_1(\mathcal D)
                \\
                &+ 4 N_{>0} -\sum_{e \in C_1^{0<\lambda<1}(\hat R)} 2|m(e)|\\
                &- 4 \#\{v \mid c(v) = {\rm D}\} - 6 \#\{v \mid c(v) = {\rm s}\}\\
                &- 2 \#\{v \mid c(v) = {\rm d_0, d_1}\} 
                 - \#\{v \mid c(v) = {\rm str}\} \\
                &+ \sum_{i=0,1}\sum_{v \in C_0(R), c(v) = {\rm str}} k_i(v)
                + \sum_{v \in C_0(R), c(v) = {\rm d_0},{\rm d_1}} k(v)\\
                &- n \#(C_1^{\rm int}(R) \setminus C_1(C))- 2(n-1) N_{>0} 
		\end{aligned}
	\end{align}
	Recall that if $c(v) = {\rm d}_0$ or ${\rm d}_1$, then $k(v)+1$ is the number of level $0$ edges containing $v$.
	Note that the sixth line of \eqref{36333666} is obtained by summing up the terms 
	$2(\ell(v)+1)$ in \eqref{form359} and $2\sum (1 - |m(e_{i}(v))|)$ in 
	\eqref{form361new}, \eqref{361form} and \eqref{form362}.
	
	We have the following straightforward identities:
	\[
	\aligned
	&\#\{v \mid c(v) = {\rm d_0, d_1}\} =   \#(C_1^{\rm int}(R) \setminus C_1(C)), \\
	& \#\{v \mid c(v) = {\rm s},{\rm D}\} =  N_{>0},   \\
	&\sum_{\lambda(v)>0}\alpha(v) \cdot c_1(\mathcal D) = \sum_{\lambda(v)>0}\alpha(v) 
	\cdot c_1(X)+\sum_{e\in C_1^{0<\lambda<1}(\hat R)}m(e).
	\endaligned
	\]
	Moreover, $\beta = \sum_{v \in C^{\rm ins}_0(\hat R)}\alpha(v)$ and 
        $$
        \aligned
        &\sum_{i=0,1}\sum_{\{v \in C_0(R), c(v) = {\rm str}\}} k_i(v)
        + \sum_{\{v \in C_0(R) \mid c(v) = {\rm d_0},{\rm d_1}\}} k(v) \\
        &= 
        k_1 + k_0 + \#\{ v \in C_0(R) \mid c(v) = {\rm d_0},{\rm d_1}\}\endaligned
        $$
        Therefore, we can use the above equalities to simplify \eqref{36333666} and show that 
        the dimension of $\widetilde{\mathcal M}^0(\mathcal R)$ is equal to the expression given in \eqref{dimension-M0R}.
\end{proof}
Now we describe the boundary of our moduli spaces 
$\mathcal M^{\rm RGW}_{k_1,k_0}(L_1,L_0;p,q;\beta)$.
\begin{definition}
	A {\it boundary SD-ribbon tree of type} (1) for
	$\mathcal M^{\rm RGW}_{k_1,k_0}(L_1,L_0;p,q;\beta)$
	is an SD-ribbon tree 
	\[\mathcal R^{(1)}(p,r,q;\beta_1,\beta_2;k_{1,0},k_{1,1};k_{2,0},k_{2,1})=
	(R;\frak v_l,\frak v_r;\mathcal S,{\rm pt},\alpha,\le)\]
	where $\beta_1 \in \Pi_2(X;L_1,L_0;p,r)$, $\beta_2\in \Pi_2(X;L_1,L_0;r,q)$,
	$\beta=\beta_1 \# \beta_2$ and $k_{1,i} + k_{2,i} = k_i$ for $i=1,\,2$.
	This SD-ribbon tree is required to satisfy the following properties.
	\begin{itemize}
		\item[(1)] The path $C$ in $R$ contains only two interior vertices $\frak v_1$ and $\frak v_2$.
				If $e$ is the edge connecting $\frak v_1$ to $\frak v_2$, then
				$r\in L_0\cap L_1$ is associated to this edge.
		\item[(2)] The graphs $R_0$ and $R_1$ associated to $R$ do not have any interior vertex. 
			In $R_0$ (resp. $R_1$) there are $k_{i,0}$ (resp. $k_{i,1}$) exterior vertices connected to $\frak v_i$. 
		\item [(3)]	$\mathcal S(\frak v_1)$ (resp. $\mathcal S(\frak v_2)$)
			is the unique SD-tree of type $(p,r;\beta_1;k_{1,0},k_{1,1})$ (resp. SD-tree of type
			$(r,q;\beta_2;k_{2,0},k_{2,1})$) with only one vertex (which is necessarily labeled with d).
	\end{itemize}
\end{definition}

An example of a boundary SD-ribbon tree is given in Figure \ref{FIgureprqsplit}.
\begin{figure}[h]
\centering
\includegraphics[scale=0.3]{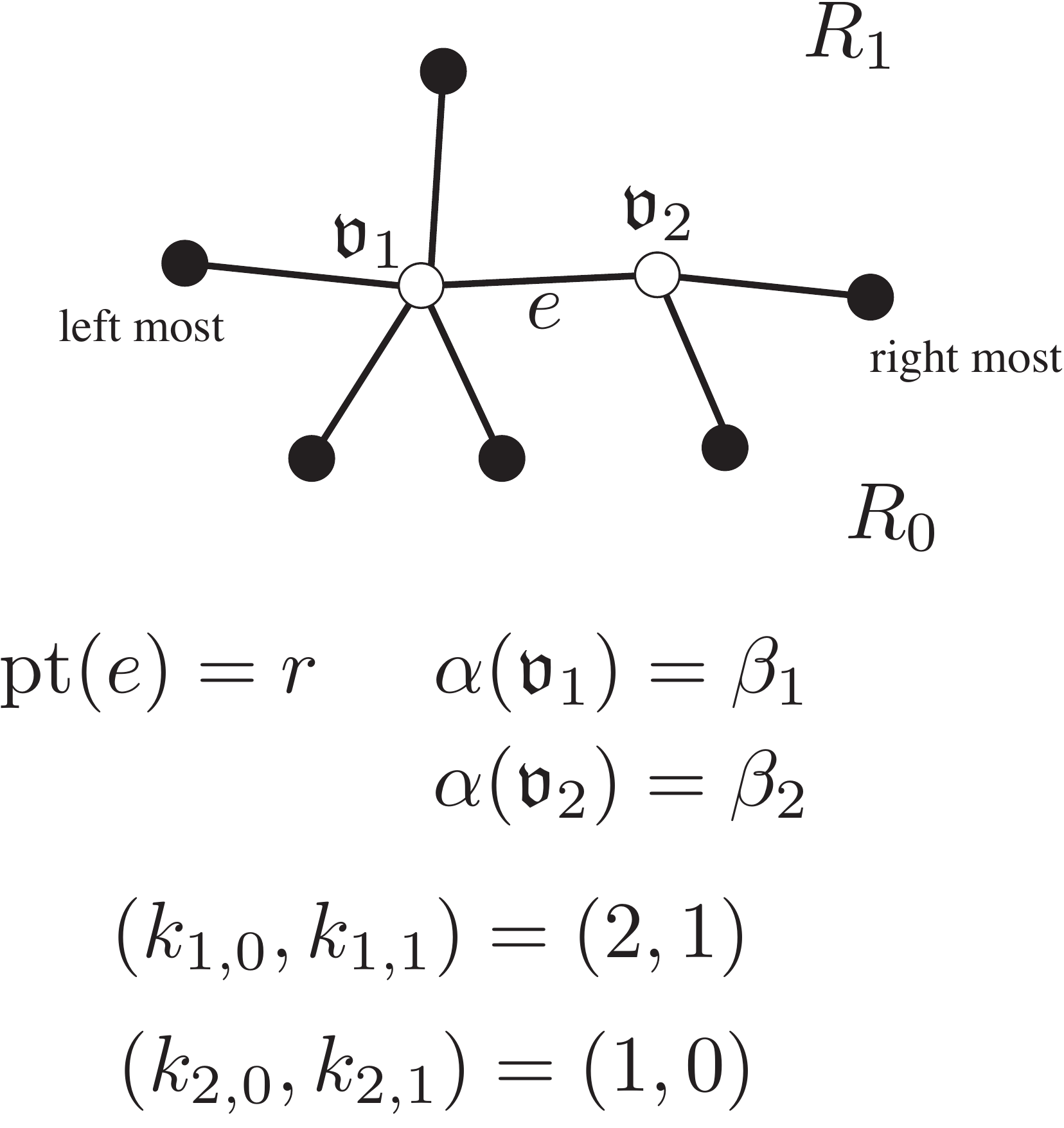}
\caption{A boundary SD-ribbon tree}
\label{FIgureprqsplit}
\end{figure}

\begin{definition}\label{def231}
	We denote by 
	\begin{equation}\label{form360}
		\mathcal M^{\rm RGW}_{k_{1,1},k_{1,0}}(L_1,L_0;p,r;\beta_1)
		\,\hat\times\,\mathcal M^{\rm RGW}_{k_{2,1},k_{2,0}}(L_1,L_0;r,q;\beta_2)
	\end{equation}
	the union of all moduli spaces ${\mathcal M}^0(\mathcal R')$ such that:
	\[
	  \mathcal R' \le \mathcal R^{(1)}(p,r,q;\beta_1,\beta_2;k_{1,0},k_{1,1};k_{2,0},k_{2,1})
	\]
	for some $r, \beta_1,\beta_2, k_{1,0}, k_{1,1}, k_{2,0}, k_{2,1}$ as above.
\end{definition}
By \cite[Theorem 4.61]{part1:top}, \eqref{form360} is closed in $\mathcal M^{\rm RGW}_{k_1,k_0}(L_1,L_0;p,q;\beta)$.
The arguments of \cite{part2:kura} show that \eqref{form360} has a Kuranishi structure and is a component of the normalized boundary of 
$\mathcal M^{\rm RGW}_{k_1,k_0}(L_1,L_0;p,q;\beta)$.

\begin{prop}\label{lema362}
	Kuranishi structures of our moduli spaces can be chosen such that the following holds.
	There exists a continuous map:
        $$
        \aligned
        \Pi : 
        &\mathcal M^{\rm RGW}_{k_{1,1},k_{1,0}}(L_1,L_0;p,r;\beta_1)\,\hat\times\,
        \mathcal M^{\rm RGW}_{k_{2,1},k_{2,0}}(L_1,L_0;r,q;\beta_2)\\
        &\to\mathcal M^{\rm RGW}_{k_{1,1},k_{1,0}}(L_1,L_0;p,r;\beta_1)\times 
        \mathcal M^{\rm RGW}_{k_{2,1},k_{2,0}}(L_1,L_0;r,q;\beta_2)
        \endaligned
        $$
        with the following properties.
        \begin{enumerate}
                \item On the inverse image of the complement of 
			\begin{equation}\label{co-dim}
                        \aligned
                        &\left(\mathcal M^{\rm RGW}_{k_{1,1},k_{1,0}}(L_1,L_0;p,r;\beta_1)^{(1)}
                        \times
                        \mathcal M^{\rm RGW}_{k_{2,1},k_{2,0}}(L_1,L_0;r,q;\beta_2)\right)
                        \\
                        &\cup
                        \left(\mathcal M^{\rm RGW}_{k_{1,1},k_{1,0}}(L_1,L_0;p,r;\beta_1)
                        \times
                        \mathcal M^{\rm RGW}_{k_{2,1},k_{2,0}}(L_1,L_0;r,q;\beta_2)^{(1)}\right)
                        \endaligned
                        \end{equation}
                        $\Pi$ is induced by an isomorphism of Kuranishi structures.
                \item Let $\frak p$ be an element of
                		\begin{equation}\label{form369696}
  		              \mathcal M^{\rm RGW}_{k_{1,1},k_{1,0}}(L_1,L_0;p,r;\beta_1)\,\hat\times\,
             			\mathcal M^{\rm RGW}_{k_{2,1},k_{2,0}}(L_1,L_0;r,q;\beta_2)
	                \end{equation}
        		        and 
                        $$
                        \overline{\frak p} := \Pi(\frak p)
                        \in
                        \mathcal M^{\rm RGW}_{k_{1,1},k_{1,0}}(L_1,L_0;p,r;\beta_1)
                        \times
                        \mathcal M^{\rm RGW}_{k_{2,1},k_{2,0}}(L_1,L_0;r,q;\beta_2).
                        $$
                		Then the Kuranishi neighborhoods 
			$\mathcal U_{\frak p} = (U_{\frak p},E_{\frak p},s_{\frak p},\psi_{\frak p})$ and
	                $\mathcal U_{\overline{\frak p}} = (U_{\overline{\frak p}},E_{\overline{\frak p}},
        		        s_{\overline{\frak p}},\psi_{\overline{\frak p}})$ of $\frak p$ and $\overline{\frak p}$
	                assigned by our Kuranishi structures have the following properties.
                Let $U_{\frak p} = V_{\frak p}/\Gamma_{\frak p}$ and 
                $U_{\overline{\frak p}} = V_{\overline{\frak p}}/\Gamma_{\overline{\frak p}}$.
                \begin{enumerate}
                \item
                There exists an injective group homomorphism
                $\phi_{\frak p} : \Gamma_{\frak p} \to \Gamma_{\overline{\frak p}}$.
                \item
                There exists a $\phi_{\frak p}$-equivariant map
                $$
                F_{\frak p} : V_{\frak p} \to V_{\overline{\frak p}}
                $$
                that is a strata-wise smooth submersion.
                \item
                $E_{\frak p}$ is isomorphic to the pullback of $E_{\overline{\frak p}}$
                by $F_{\frak p}$. In other words,
                there exists fiberwise isomorphic lift 
                $$
                \tilde F_{\frak p} : E_{\frak p} \to E_{\overline{\frak p}}
                $$
                of $F_{\frak p}$, which is $\phi_{\frak p}$ equivariant.
                \item
                $\tilde F_{\frak p} \circ s_{\frak p} = s_{\overline{\frak p}} \circ F_{\frak p}$.
                \item
                $\psi_{\overline{\frak p}} \circ F_{\frak p} = \Pi \circ \psi_{\frak p}$
                on $s_{\frak p}^{-1}(0)$.
                \end{enumerate}
                \item
                $\tilde F_{\frak p},F_{\frak p}$ are compatible with the coordinate 
                changes.
        \end{enumerate}
\end{prop}
Before turning to the proof of this proposition, we elaborate on Item (3). Recall that a Kuranishi structure for $X$ assigns a 
Kuranishi chart $\mathcal U_{\frak p}  = (U_{\frak p},E_{\frak p},s_{\frak p},\psi_{\frak p})$ to each  $\frak p \in X$.
If $\frak q \in {\rm Im}(\psi_{\frak p})$, then 
we require that there exists a coordinate change 
$\Phi_{\frak p\frak q} =(U_{\frak p\frak q},\widehat\varphi_{\frak p\frak q},\varphi_{\frak p\frak q})$
from $\mathcal U_{\frak q}$ to $\mathcal U_{\frak p}$ in the following sense:
\begin{enumerate}
\item[(CoCh.1)]
There exists an open neighborhood $U_{\frak p\frak q}$ of $o_{\frak q}$ in $U_{\frak q}$.
\item[(CoCh.2)]
There exists an embedding of orbifolds $\varphi_{\frak p\frak q} : 
U_{\frak p\frak q} \to U_{\frak p}$.
\item[(CoCh.3)]
There exists an embedding of orbibundles 
$\widehat\varphi_{\frak p\frak q} : 
E_{\frak q}\vert_{U_{\frak p\frak q}} \to E_{\frak p}$
which is a lift of $\varphi_{\frak p\frak q}$.
\item[(CoCh.4)]
$\widehat\varphi_{\frak p\frak q}$ is compatible with Kuranishi maps 
and parametrization maps in the following sense:
\begin{enumerate}
\item 
$s_{\frak p} \circ \varphi_{\frak p\frak q} = \widehat\varphi_{\frak p\frak q}
\circ s_{\frak q}$.
\item
$\psi_{\frak p} \circ \varphi_{\frak p\frak q} = \psi_{\frak q}$
on $U_{\frak p\frak q} \cap s_{\frak q}^{-1}(0)$.
\end{enumerate}
\item[(CoCh.5)]
On ${\rm Im}(\varphi_{\frak p\frak q}) \cap s_{\frak p}^{-1}(0)$,  
the differential $ds_{\frak p}$ induces a bundle isomorphism:
$$
\mathcal N_{U_{\frak p\frak q}}U_{\frak p} \cong \frac{E_{\frak p}}
{{\rm Im}(\widehat\varphi_{\frak p\frak q})}. 
$$
where the bundle on the left hand side denotes the normal bundle of $\varphi_{\frak p\frak q}(U_{\frak p\frak q})$ in $U_{\frak p}$.
\end{enumerate}
See \cite[Definition A.1.3 and Definition A1.14]{fooobook2} or \cite[Definition 3.6]{fooonewbook}.
The notion of embedding of orbifolds and orbibundle is 
defined in \cite[Chapter 23]{fooonewbook}.
The coordinate change maps $\Phi_{\frak p\frak q}$ are required to satisfy certain compatibility conditions and we refer the reader to \cite[(A1.6.2)]{fooobook2},
\cite[Definition 3.9]{fooonewbook}
for these conditions.
\par
In Proposition \ref{lema362}, let $\frak p$, $\frak q$ be elements of \eqref{form369696} such that $\frak q \in {\rm Im}(\psi_{\frak p})$.
We denote $\Pi(\frak p)$, $\Pi(\frak q)$ by $\overline{\frak p}$, $\overline{\frak q}$.
We then have $\varphi_{\frak p\frak q}$, $\widehat\varphi_{{\frak p}{\frak q}}$, 
$\varphi_{\overline{\frak p}\overline{\frak q}}$, 
$\widehat\varphi_{\overline{\frak p}\overline{\frak q}}$.
Item (3) of Proposition \ref{lema362} asserts that:
\[
  F_{\frak p}\circ \varphi_{\frak p\frak q} = \varphi_{\overline{\frak p}\overline{\frak q}}\circ F_{\frak q},\qquad 
  \widetilde F_{\frak p}\circ \widehat\varphi_{{\frak p}{\frak q}} = \widehat\varphi_{\overline{\frak p}\overline{\frak q}}\circ \widetilde
  F_{\frak q}
\]
on the domain where both sides are defined.

\begin{proof}[Proof of Proposition \ref{lema362}]
	We construct the continuous map $\Pi$ in this subsection.
	The study of its relation to Kuranishi structure is postponed to
	Section \ref{sub:systemconst}.
	
	Removing the edge $e$ of 
	$\mathcal R=\mathcal R^{(1)}(p,r,q;\beta_1,\beta_2;k_{1,0},k_{1,1};k_{2,0},k_{2,1})$ 
	and adding two exterior vertices produces
	two SD-ribbon graphs $\mathcal R_1$, $\mathcal R_2$ of types $(p,r;\beta_1;k_{1,0},k_{1,1})$, 
	$(r,q;\beta_2;k_{2,0},k_{2,1})$. (See Figure \ref{Figuresplitted}.)
	Let $\mathcal R'_1$ (resp. $\mathcal R'_2$) be an SD-ribbon tree of type $(p,r;\beta_1;k_{1,0},k_{1,1})$
	(resp. $(r,q;\beta_2;k_{2,0},k_{2,1})$) such that $\mathcal R'_1 \le \mathcal R_1$
	(resp. $\mathcal R'_2 \le \mathcal R_2$).
	We glue the edge incident to the right most vertex of $\mathcal R'_1$ 
	to the edge incident to the left most vertex of $\mathcal R'_2$.
	(The right most vertex of $\mathcal R'_1$ and 
	the left most vertex of $\mathcal R'_2$ are not 
	vertices of the resulting tree.) 
	\begin{figure}[h]
                \centering
                \includegraphics[scale=0.3]{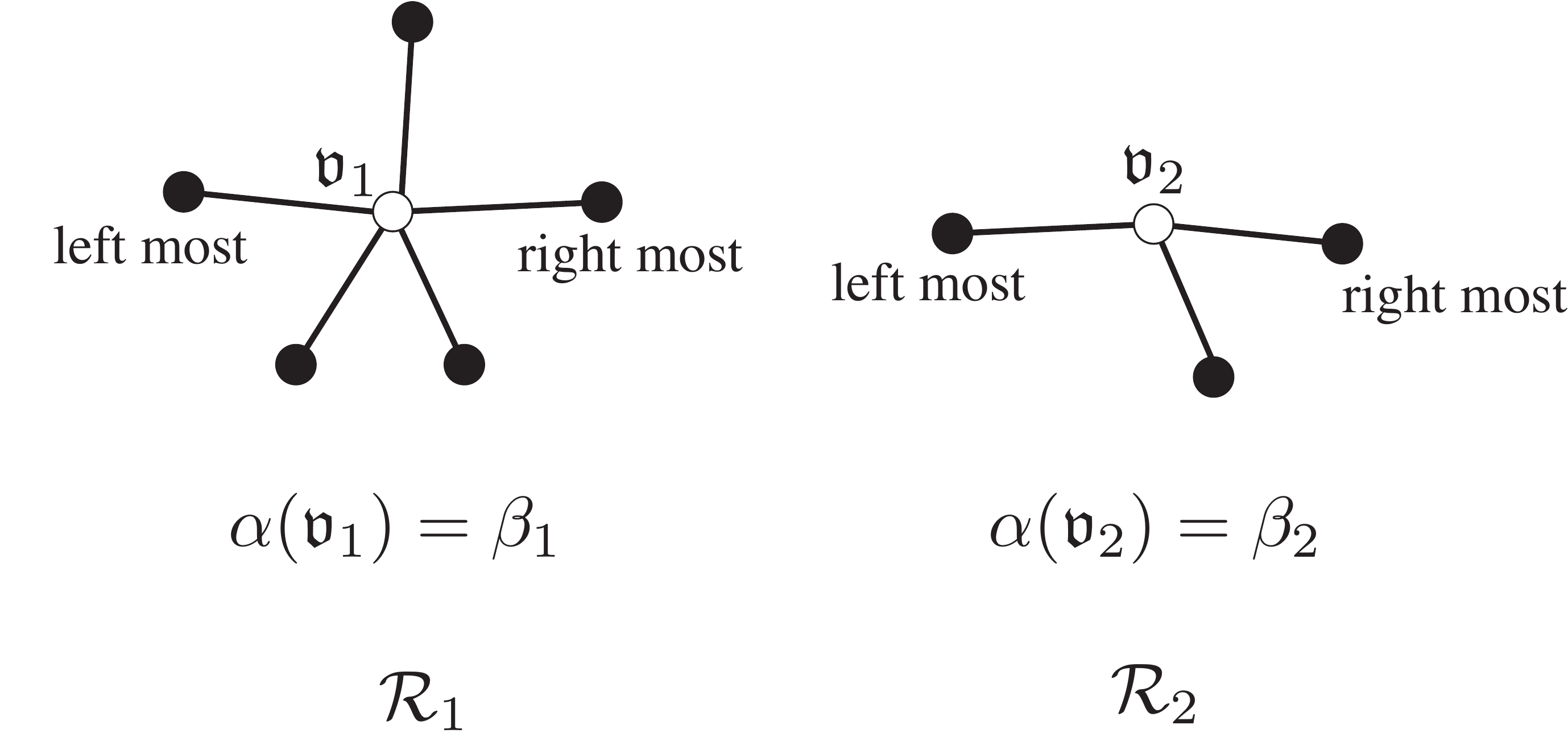}
                \caption{$\mathcal R_1$ and $\mathcal R_2$.}
                \label{Figuresplitted}
        \end{figure}
	
	This tree together with a level function determines an 
	SD-ribbon tree of type $(\beta,k_0,k_1)$.
	Let $\hat\lambda$ be a level function associated to a quasi order $\le$, which 
	restricts to the quasi order associated to $\mathcal R'_i$ on $C^{\rm ins}_0(\hat R'_i)$.
	Note that the choice of $\hat\lambda$ is {\it not} unique.
	(See Figure \ref{Section3newfigure}.)
	We denote the resulting SD-ribbon tree of type $(\beta;k_0,k_1)$ by:
	$$
	\mathcal R'_1 \#_{\hat\lambda} \mathcal R'_2
	$$
	For any SD-ribbon tree $\mathcal R'$ of type $(\beta;k_0,k_1)$ with $\mathcal R' \le \mathcal R$,
	there exist unique $\mathcal R'_1$, $\mathcal R'_2$ and $\hat\lambda$ such that:
        \begin{equation}\label{form36666}
                \mathcal R' = \mathcal R'_1 \#_{\hat\lambda} \mathcal R'_2.
                \end{equation}
                \begin{figure}[h]
                \centering
                \includegraphics[scale=0.3]{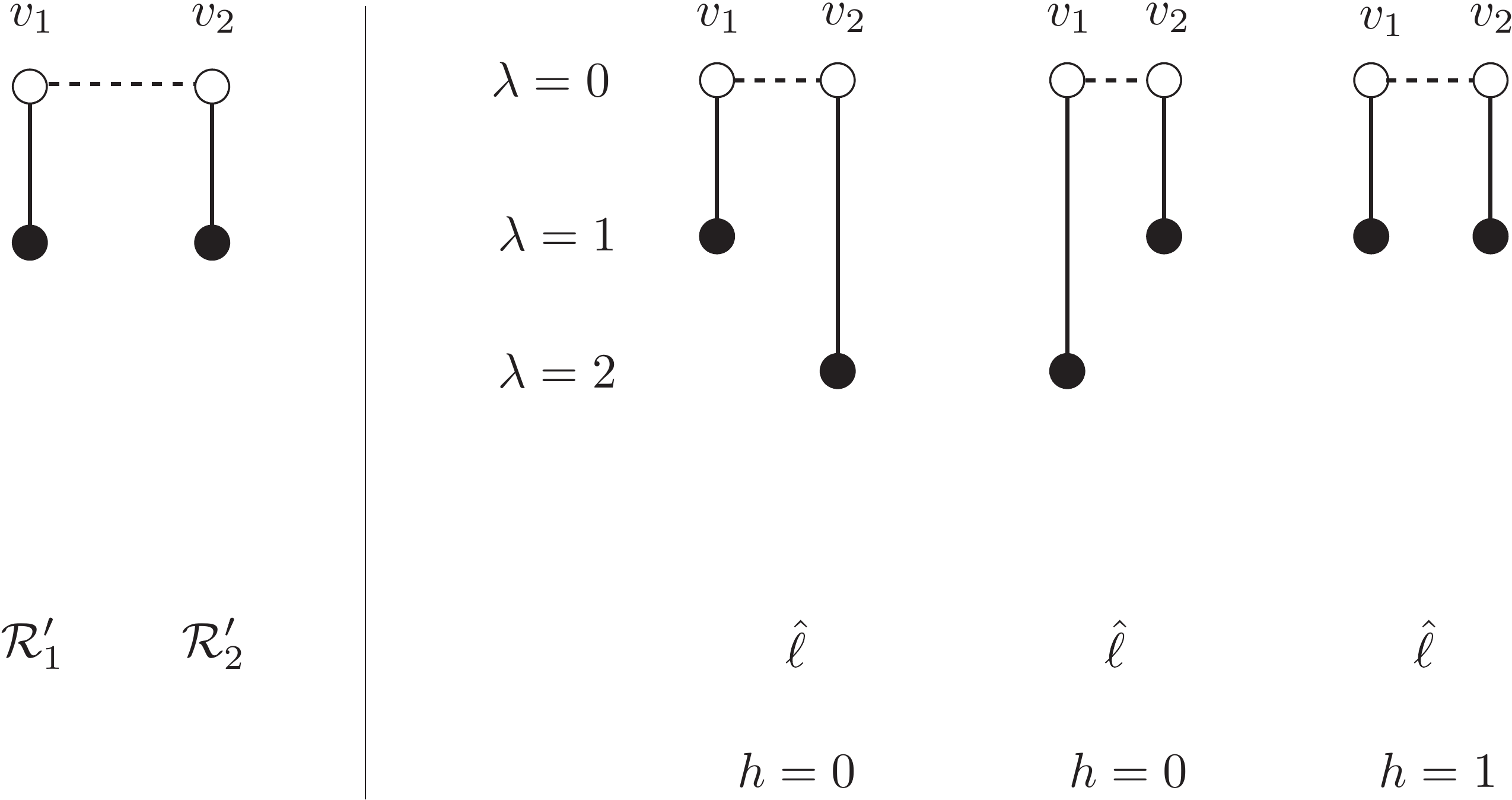}
                \caption{Three different choices of $\hat{\ell}$.}
                \label{Section3newfigure}
        \end{figure}



	For $\mathcal R'$ as in \eqref{form36666}, the map $\Pi$ restricts to a map:
	$$
	\Pi : {\mathcal M}^0(\mathcal R')\to  {\mathcal M}^0(\mathcal R'_1) \times  {\mathcal M}^0(\mathcal R'_2).
	$$
	which can be described as follows.
	Let $\vert \hat\lambda\vert$, $\vert\lambda_1\vert$ and $\vert\lambda_2\vert$  be the number of positive levels of 
	$\mathcal R'$, $\mathcal R_1'$ and $\mathcal R_2'$.
	The fact that the quasi order corresponding to $\hat\lambda$ 
	restricts to the quasi order corresponding $\lambda_i$
	on $C^{\rm ins}_0(\hat R_i)$ implies that
	$\vert \hat\lambda\vert \le \vert \lambda_1\vert + \vert\lambda_2\vert$.
	Let $h = \vert\lambda_1\vert + \vert\lambda_2\vert - \vert \hat\lambda\vert$.
	Using identifications:
	\[
	  \widetilde{\mathcal M}^0(\mathcal R'_1) \times  \widetilde{\mathcal M}^0(\mathcal R'_2)=
	  \widetilde{\mathcal M}^0(\mathcal R'),\hspace{1cm}
	  {\rm Aut}(\mathcal R')= {\rm Aut}(\mathcal R'_1) \times {\rm Aut}(\mathcal R'_2),
	\]
	it is easy to see that there exists a $\bbC_*^{h}$-action on ${\mathcal M}^0(\mathcal R')$ such that 
	\begin{equation}\label{firm37272}
		{\mathcal M}^0(\mathcal R'_1) \times  {\mathcal M}^0(\mathcal R'_2)=
		{\mathcal M}^0(\mathcal R')/\bbC_*^{h}.
	\end{equation}
	The map $\Pi$ is defined to be the projection map induced by this identification.
\end{proof}
We thus gave a detailed description of the part of the boundary of the moduli space $\mathcal M^{\rm RGW}_{k_1,k_0}(L_1,L_0;p,q;\beta)$ given in Definition \ref{def231}.
Next, we describe other boundary elements in a similar way.
\begin{definition}
	A {\it boundary SD-ribbon trees of type} (2) for 
	$\mathcal M^{\rm RGW}_{k_1,k_0}(L_1,L_0;p,q;\beta)$
	is an SD-ribbon tree 
	\[\mathcal R^{(2)}(p,q;\beta_1,\beta_2;k_{1,1}+1,k_{2,1},j;k_{0})
	=(R;\frak v_l,\frak v_r;\mathcal S,{\rm pt},\alpha,\le)\]
	where $\beta_1 \in \Pi_2(L_1,L_0;p,q)$, $\beta_2\in\Pi_2(X,L_1)$,
	$\beta=\beta_1 \# \beta_2$, $k_{1,1}+k_{2,1}=k_1$ and $j \in \{1,\dots,k_{1,1}+1\}$. 
	This SD-ribbon tree is required to satisfy the following properties: 
	\begin{itemize}
		\item[(1)] There are only two interior vertices $\frak v_1$ and $\frak v_2$.
				The vertex $\frak v_1$ has color ${\rm str}$ and $\frak v_2$ has color ${\rm d_1}$.
				There is an edge $e$ connecting $\frak v_1$ to $\frak v_2$.
		\item[(2)] The graphs $R_1$ has exactly $k_{1,1}+1$ exterior vertices connected to $\frak v_1$ and 
				$k_{2,1}$ exterior vertices connected to $\frak v_2$. 
				The edge $e$ is the $j$-th edge in $R_1$ connected to $\frak v_1$.
				Here we use the ribbon structure at $\frak v_1$ to label the 
				edges in $R_1$, which are connected to $\frak v_1$. The graph 
				$R_0$ has only $k_{0}$ exterior vertices connected to $\frak v_1$.
		\item [(3)]	$\mathcal S(\frak v_1)$ (resp. $\mathcal S(\frak v_2)$)
			is the unique SD-tree (resp. DD-tree) of type $(p,q;\beta_1;k_{0},k_{1,1}+1)$ (resp. 
			$(\beta_2;k_{2,1})$) with only one vertex (which is necessarily labeled with d).
	\end{itemize}
	We can similarly define boundary SD-ribbon trees of type (3), which have the form:
	\[\mathcal R^{(3)}(p,q;\beta_1,\beta_2;k_1;k_{1,0};k_{2,0}+1,j).\]
\end{definition}
Figure \ref{FigureSplit(2)(3)} gives examples of boundary SD-ribbon trees of types (2) and (3).
\begin{figure}[h]
\centering
\includegraphics[scale=0.3]{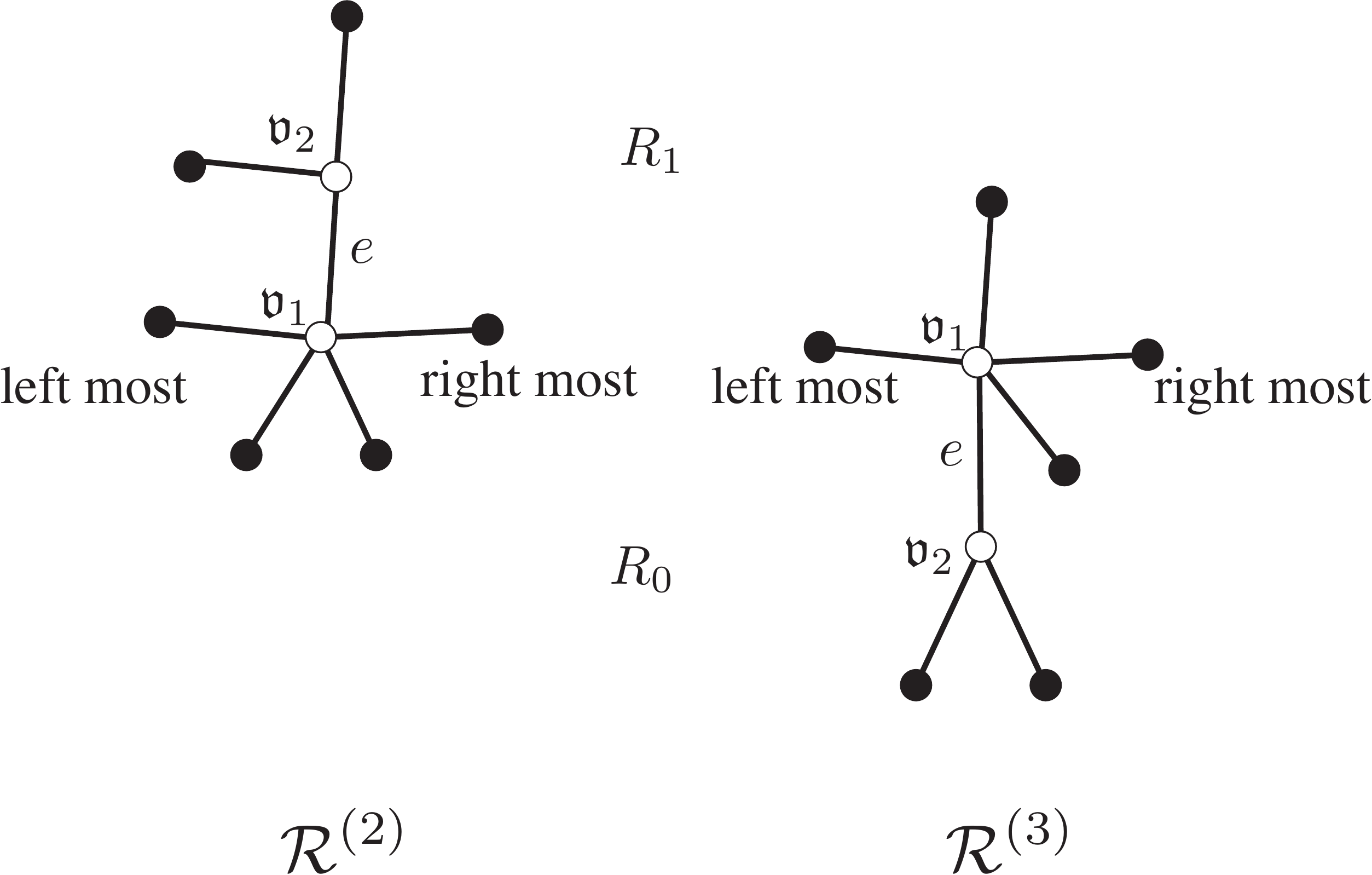}
\caption{The ribbon trees $\mathcal R^{(2)}$ and $\mathcal R^{(3)}$}
\label{FigureSplit(2)(3)}
\end{figure}

\begin{definition}\label{defn239}
	We define
	\begin{equation}\label{form365}
		\mathcal M^{\rm RGW}_{k_{1,1}+1,k_{0}}(L_1,L_0;p,q;\beta_1)
		\,{}_{{\rm ev}_{1,j}}\hat\times_{{\rm ev_{0}}}\,\mathcal M^{\rm RGW}_{k_{2,1}+1}(L_1;\beta_2)
	\end{equation}
	to be the union of all ${\mathcal M}^0(\mathcal R')$ with $\mathcal R' \le \mathcal R^{(2)}
	(p,q;\beta_1,\beta_2;k_{1,1}+1,k_{2,1},j;k_{0})$.
	Similarly, we define
	\begin{equation}\label{form3656}
		\mathcal M^{\rm RGW}_{k_{1},k_{1,0}+1}(L_1,L_0;p,q;\beta_1)
		\,{}_{{\rm ev}_{0,j}}\hat\times_{{\rm ev_{0}}}\,\mathcal M^{\rm RGW}_{k_{2,0}+1}(L_0;\beta_2).
	\end{equation}
\end{definition}
\begin{prop}\label{lem364}
	There exists a map
	\begin{equation}\label{form371}
		\aligned\Pi : &\mathcal M^{\rm RGW}_{k_{1,1}+1,k_{0}}(L_1,L_0;p,q;\beta_1)
		\,{}_{{\rm ev}_{1,j}}\hat\times_{{\rm ev_{0}}}\,\mathcal M^{\rm RGW}_{k_{2,1}+1}(L_1;\beta_2)\\
		&\to\mathcal M^{\rm RGW}_{k_{1,1}+1,k_{0}}(L_1,L_0;p,q;\beta_1)\,{}_{{\rm ev}_{1,j}}\times_{{\rm ev_{0}}}\,
		\mathcal M^{\rm RGW}_{k_{2,1}+1}(L_1;\beta_2)
		\endaligned
	\end{equation}
	satisfying the following properties.
	\begin{enumerate}
		\item On the inverse image of the complement of 
			\[
			  \aligned
			  &\left(\mathcal M^{\rm RGW}_{k_{1,1},k_{0}}(L_1,L_0;p,q;\beta_1)^{(1)}\right)\,{}_{{\rm ev}_{1,j}}					\times_{{\rm ev_{0}}}\,
			\left(\mathcal M^{\rm RGW}_{k_{2,1}+1}(L_1;\beta_2)\right)\\
			&\cup\left(\mathcal M^{\rm RGW}_{k_{1,1},k_{0}}(L_1,L_0;p,q;\beta_1)\right)
			\,{}_{{\rm ev}_{1,j}}\times_{{\rm ev_{0}}}\,\left(\mathcal M^{\rm RGW}_{k_{2,1}+1}(L_1;\beta_2)^{(1)}\right)
			\endaligned
			\]
	$\Pi$ is induced by the isomorphism of Kuranishi structures. Here the fiber product is taken over $L_1$.
	\item The same property as in Proposition \ref{lema362} {\rm (2)} holds.
	\item The same property as in Proposition \ref{lema362} {\rm (3)} holds.
\end{enumerate}
	A similar statement holds for:
	\[\mathcal M^{\rm RGW}_{k_{1},k_{1,0}}(L_1,L_0;p,q;\beta_1)\,{}_{{\rm ev}_{0,j}}\hat\times_{{\rm ev_{0}}}\,
	\mathcal M^{\rm RGW}_{k_{2,0}+1}(L_0;\beta_2).\]
\end{prop}
The proof is similar to the proof of Proposition \ref{lema362} and is omitted.
\begin{prop}\label{lem365}
	The normalized boundary of $\mathcal M^{\rm RGW}_{k_1,k_0}(L_1,L_0;p,q;\beta)$
	is the disjoint union of \eqref{form360}, \eqref{form365}, \eqref{form3656}.
\end{prop}
\begin{proof}
	By the calculation of dimension we gave in the proof of Proposition \ref{prop356},
	the only codimension 1 strata are $\mathcal R^{(1)}(p,r,q;\beta_1,\beta_2;k_{1,0},k_{1,1};k_{2,0},k_{2,1})$, 
	$\mathcal R^{(2)}(p,q;\beta_1,\beta_2;k_{1,1}+1,k_{2,1},j;k_{0})$ and 
	$\mathcal R^{(3)}(p,q;\beta_1,\beta_2;k_1; k_{1,0};k_{2,0}+1,j)$.
	This gives the lemma.
\end{proof}

\subsection{Construction of multisection}
\label{subsub:constmulti}

	In this subsection, we prove Theorem \ref{prop61111}.
	The proof is by induction on $\omega \cap \beta$ and $(k_0,k_1)$. 
	In this inductive process we construct multi-valued perturbations for all moduli spaces 
	$\omega\cap \beta\leq E$ and $k_0+k_1\leq N$ for some constants $E$ and $N$.
	In particular, we may construct perturbations for the moduli spaces with dimension greater than $1$.
	But  the conditions (1) and (3) hold only for moduli spaces with dimension at most $1$.
	We assume that we already constructed  
	multi-valued perturbations for the moduli spaces of type 
	$(\beta;k_0,k_1;p,q)$ such that $\omega \cap \beta < \omega \cap \alpha$ or
	$\omega \cap \beta=\omega \cap \alpha$ and $k_1+k_1<j_1+j_2$.

	We fix a continuous multi-section on 
	$\partial\mathcal M^{\rm RGW}_{j_1,j_0}(L_1,L_0;p,q;\alpha)$ using the induction hypothesis.
	For example, part of the boundary of the moduli space
	$\mathcal M^{\rm RGW}_{j_1,j_0}(L_1,L_0;p,q;\alpha)$
	is described by the moduli spaces of the form
	\begin{equation}\label{bdry-type-20}
	  \mathcal M^{\rm RGW}_{j_1',j_0}(L_1,L_0;p,q;\beta_1) \hat \times_{L_1} 
	  \mathcal M^{\rm RGW}_{j_1''+1}(L_1;\beta_2),
	\end{equation}
	where $j_1'+j_1''=j_1+1$, $\beta_i\cap \mathcal D=0$ and $\beta_1\#\beta_2=\alpha$.
	Assuming $ \mathcal M^{\rm RGW}_{k_1''+1}(L_1;\beta_2)$ is non-empty, we have 
	$\omega\cap \beta_1< \omega\cap \alpha$ or $j_1'<j_1$ and $\omega\cap \beta_1\leq \omega\cap \alpha$ . Therefore, the induction hypothesis
	implies that we already fixed a multi-valued perturbation 
	for $\mathcal M^{\rm RGW}_{j_1',j_0}(L_1,L_0;p,q;\beta_1)$. We use the fiber product\footnote{
	See \cite[Proposition 20.5]{fooonewbook} for the definition of fiber product of
	multi-valued perturbations.} 
	of this perturbation and the trivial perturbation for $\mathcal M^{\rm RGW}_{j_1''+1}(L_1;\beta_2)$
	to define a multi-valued perturbation for:
	\[
	  \mathcal M^{\rm RGW}_{j_1',j_0}(L_1,L_0;p,q;\beta_1) \times_{L_1} 
	  \mathcal M^{\rm RGW}_{j_1''+1}(L_1;\beta_2)
	\]	
	Now we use the map $\Pi$ in Proposition \ref{lem364} 
	to pull-back this perturbation to 
	the space in \eqref{bdry-type-20}. More generally, we can use the already constructed perturbations
	for the moduli spaces of strips and trivial perturbations for the moduli spaces of discs to 
	define a multi-valued perturbation for any boundary component and corner of 
	$\mathcal M^{\rm RGW}_{j_1,j_0}(L_1,L_0;p,q;\alpha)$. The induction hypothesis implies that 
	these perturbations are compatible. 
	
	In the case that the virtual dimension
	of $\mathcal M^{\rm RGW}_{j_1,j_0}(L_1,L_0;p,q;\alpha)$ is greater than $1$, we extend 
	the chosen multi-valued perturbation on the boundary into a $C^0$ perturbation defined over the whole
	moduli space. 
	In the case that the virtual dimension
	is at most $1$, we need to choose this extension such that the conditions in (1) and (3) of 
	Theorem \ref{prop61111} are satisfied. To achieve this goal, we analyze the vanishing locus
	of the multi-valued perturbation over the boundary of 
	$\mathcal M^{\rm RGW}_{j_1,j_0}(L_1,L_0;p,q;\alpha)$.
	
	Let $\mathcal M^{\rm RGW}_{j_1,j_0}(L_1,L_0;p,q;\alpha)$ have virtual dimension not greater than $1$.
	On the stratum of the boundary where there exists at least one disk bubble
	on which the map is non-constant, the assumption implies that the Maslov number of the disk bubble 
	is at least $4$. This implies that there is at least one irreducible component, 
	which is a strip with homology class $\beta_1$ and is contained in a moduli space with negative
	 virtual dimension. 
	 Therefore, our multi-valued perturbation on this boundary component does not vanish.
	 
	 The rest of the proof is divided into two parts.
	We firstly consider the case where $\dim(\mathcal M^{\rm RGW}_{j_1,j_0}(L_1,L_0;p,q;\alpha))$
	is non-positive. The part of the boundary corresponding to splitting into two or more strips has a 
	strip component which has negative virtual dimension. Therefore, our multi-valued perturbation
	does not vanish on this part of the boundary, too. As a consequence, we can extend the perturbation 
	in the $C^0$ sense to a neighborhood of the boundary such that it is still non-vanishing 
	in this neighborhood. We approximate the perturbation by a section which is 
	$C^1$ outside a smaller neighborhood of the boundary. 
	Now we can extend this multi-section in a way which is transversal to $0$ on 
	$\mathcal M^{\rm RGW}_{k_1,k_0}(L_1,L_0;p,q;\beta)$ and 
	$\mathcal M^{\rm RGW}_{k_1,k_0}(L_1,L_0;p,q;\beta)^{(1)}$
	using the existence 
	theorem of multi-valued perturbations. (See, for example, \cite[Theorem 13.5]{fooonewbook}.)
	If the virtual dimension is $0$, there exists finitely many zeros which do not belong to 
	$\mathcal M^{\rm RGW}_{k_1,k_0}(L_1,L_0;p,q;\beta)^{(1)}$.
	If the  virtual dimension is negative, the multi-valued perturbation does not vanish.
	This completes the proof in the case that
	$\dim(\mathcal M^{\rm RGW}_{j_1,j_0}(L_1,L_0;p,q;\alpha))\le 0$.

	Next, we consider the case that $\mathcal M^{\rm RGW}_{j_1,j_0}(L_1,L_0;p,q;\alpha)$ is 
	$1$-dimensional. The constructed multi-valued perturbation on the boundary does not vanish 
	except on the boundary components of the form
	\begin{equation}\label{form6197}
		\mathcal M^{\rm RGW}_{j'_1,j'_0}(L_1,L_0;p,r;\beta_1) \hat \times 
		\mathcal M^{\rm RGW}_{j''_1,j''_0}(L_1,L_0;r,q;\beta_2),
	\end{equation}
	or 
	\begin{equation}\label{form6198}
		\aligned
		&\mathcal M^{\rm RGW}_{j_1-1,j_0}(L_1,L_0;p,q;\alpha) 
		\hat \times_{L_1} \mathcal M^{\rm RGW}_{3}(L_1;0), \\
		&\mathcal M^{\rm RGW}_{j_1,j_0-1}(L_1,L_0;p,q;\alpha) 
		\hat \times_{L_0} \mathcal M^{\rm RGW}_{3}(L_0;0).
		\endaligned
	\end{equation}
	In \eqref{form6197}, both of the factors have virtual dimension $0$.
	Therefore, we may assume that the zero set of the multi-valued perturbation
	we have defined on those factors do not lie in the strata of codimension at least $2$.
	Since the multi-sections on the two factors have 
	finitely many zeros and the map $\Pi$ in Proposition \ref{lem364}
	is an isomorphism, the first factor gives rise to
	finitely many points in the boundary. In the case of (\ref{form6198}), the 
	first factor has virtual dimension $0$ and the multi-section 
	there vanishes only at finitely many points.
	The second factor is identified with $L_1$ or $L_0$.
	Therefore, the fiber product is identified with the first factor. In summary, 
	the multi-valued perturbation has finitely many zeros on the boundary and item (3) holds.
        
        We fix Kuranishi charts at the finitely many zeros on the boundary.
        Since these are boundary points, we can easily extend 
        our multi-valued perturbation to the interior of the chosen 
        Kuranishi charts so that it is transversal to $0$.
        Now we extend the multi-valued perturbation further to a neighborhood of 
        the boundary of $\mathcal M^{\rm RGW}_{j_1,j_0}(L_1,L_0;p,q;\alpha)$
        in the $C^0$ sense such that the multi-valued perturbation does not vanish 
        except at those finitely many charts.
	We may assume that the multi-valued perturbation is $C^1$ outside a 
        smaller neighborhood of the boundary. We use again the existence theorem of multi-sections
	that are transversal to zero everywhere to complete the construction of the multi-valued perturbation
	on $\mathcal M^{\rm RGW}_{j_1,j_0}(L_1,L_0;p,q;\alpha)$.
%
%
\qed
\begin{remark}
	This proof never uses the smoothness of the coordinate change
	with respect to the gluing parameters $\sigma_{e}$ when $\sigma_{e}=0$.
	In most part of the proof, we extend the multi-section at the
	boundary to the interior only in the $C^0$ sense.
	When we extend the multi-section near a point of the 
	boundary where  it vanishes, we fix a 
	chart there and extend it on that chart.
	We use other charts to extend the multi-section in $C^0$ sense to a 
	neighborhood of the boundary. 
	(Recall that we only need the differentiability of the multi-valued perturbation in a neighborhood of 
	its vanishing set to define virtual fundamental chain.)
	The key point here is that the multi-valued perturbation on the boundary has only isolated zeros.
\end{remark}
\begin{remark}
	Even though we do not perturb the moduli spaces of disks of virtual dimension 
	greater than $1$
	in the proof of Theorem \ref{prop61111}, we used the fact that these 
	moduli spaces admit Kuranishi structures.
\end{remark}


\subsection{Completion of the proof of Theorem \ref{lem48}}
\label{subsub:Floer2}

	We consider $p,q\in L_0\cap L_1$  and $\beta \in \Pi_2(L_1,L_0;p,q)$ such that
	\[\dim\mathcal M^{\rm RGW}_{0,0}(L_1,L_0;p,q;\beta) = 1.\]
	By Corollary \ref{Cor27} (3), 
	we have
	\begin{equation}\label{form44}
		[\partial \mathcal M^{\rm RGW}_{0,0}(L_1,L_0;p,q;\beta),\widehat{\frak s}^n] = 0.
	\end{equation}
	We will prove that the sum of \eqref{form44} over all possible choices of $\beta$ becomes the coefficient of $[q]$ in $\partial\circ \partial ([p])$, 
	which will show that $\partial$ is a differential.

	We study various types of boundary elements appearing in Theorem \ref{theorem30}.
	The contribution of the elements in part (1) of Theorem \ref{theorem30} is given as
	\begin{equation}\label{form450}
		[(\mathcal M^{\rm RGW}_{0,0}(L_1,L_0;p,r;\beta_1) \hat \times \mathcal M^{\rm RGW}_{0,0}(L_1,L_0;r,q;\beta_2),\widehat{\frak s}^n)].
	\end{equation}
	We claim that this is equal to
	\begin{equation}\label{form460}
		[(\mathcal M^{\rm RGW}_{0,0}(L_1,L_0;p,r;\beta_1),\widehat{\frak s}^n)]
		[(\mathcal M^{\rm RGW}_{0,0}(L_1,L_0;r,q;\beta_2),\widehat{\frak s}^n)].
	\end{equation}
	The number $[(\mathcal M^{\rm RGW}_{0,0}(L_1,L_0;p,r;\beta_1),\widehat{\frak s}^n)]$ is a weighted count of the zeroes of $\frak s^n$ on this moduli space.
	By Theorem \ref{prop61111} (3), this zero set is disjoint from $\mathcal M^{\rm RGW}_{0,0}(L_1,L_0;p,r;\beta_1)^{(1)}$.
	The same conclusion holds for the virtual chain $[(\mathcal M^{\rm RGW}_{0,0}(L_1,L_0;r,q;\beta_2),\widehat{\frak s}^n)]$.
	Therefore, the map $\Pi$ in (\ref{form371}) induces an isomorphism of Kuranishi structures on a neighborhood of the zero set of $\widehat{\frak s}^n$.
	It follows that the weighted count of (\ref{form450}) is equal to (\ref{form460}).
	To complete the proof of Theorem \ref{lem48}, it suffices to show the contribution of the other boundary components vanishes.

	We firstly consider the contribution from Item (2) of Theorem \ref{theorem30}, which is
	\begin{equation}\label{form47}
		\mathcal M^{\rm RGW}_{0,1}(L_1,L_0;p,q;\beta')\, \hat\times_{L_0}\, \mathcal M^{\rm RGW}_1(L_0;\alpha).
	\end{equation}
	If $\mathcal M^{\rm RGW}_1(L_0;\alpha)$  is nonempty, then 
	$\omega(\alpha) > 0$. Using monotonicity and the fact that the minimum 
	Maslov number of $L_0$ is at least $4$, we can conclude that 
	$\mu(\alpha) \ge 4$. This inequality together with $\mu(\beta)=2$ implies that $\mu(\beta')\leq -2$. Thus the moduli space 
	$\mathcal M^{\rm RGW}_{0,1}(L_1,L_0;p,q;\beta')$ has a negative dimension.
	By Corollary \ref{Cor27}, $\widehat{\frak s}^{n}$ never vanishes on $\mathcal M^{\rm RGW}_{0,1}(L_1,L_0;p,q;\beta')$.
	Therefore, \eqref{form47}  is empty after the perturbation given by $\widehat{\frak s}^n$, and hence it does not contribute to \eqref{form44}.
	In the same way, we can show Item (3) of Theorem \ref{theorem30}  does not contribute to \eqref{form44}. This completes the proof of 
	Theorem \ref{lem48}, except that we still need to construct Kuranishi structures required for Propositions \ref{lema362}, \ref{lem364}, \ref{lem365}
	which will be done in Section \ref{sub:systemconst}.
\qed

\begin{definition}
If  Condition \ref{cond420} is satisfied, then we define the Lagrangian Floer homology of $L_0$ and $L_1$ as
$$
HF(L_1,L_0;o;X \setminus \mathcal D) \cong \frac{{\rm Ker} \left( \partial : CF(L_1,L_0;\bbQ;o) 
\to CF(L_1,L_0;\bbQ;o)\right)}{{\rm Im} \left( \partial : CF(L_1,L_0;\bbQ;o) 
\to CF(L_1,L_0;\bbQ;o)\right)}.
$$
\end{definition}
We remark that Item (1) of Theorem \ref{theorem30},
$$
{\rm rank} HF(L_1,L_0;o;X \setminus \mathcal D) \le \# L_0\cap_{o} L_1,
$$
is immediate from the definition.

\section{Construction of a System of Kuranishi Structures}
\label{sub:systemconst}

\subsection{Statement}
\label{subsub:statecoma}

In \cite{part2:kura}, we constructed a Kuranishi structure for each moduli space $\mathcal M_{k+1}^{\rm RGW}(L;\beta)$. In this section, we study how these Kuranishi structures are related to each other at their boundaries and corners. More specifically, we prove the disk moduli version of Propositions \ref{lema362}, \ref{lem364}, \ref{lem365}, stated as Theorem \ref{lema362rev}.
The notation $\hat\times_L$ is discussed in Subsection \ref{c}. Recall also from Subsection \ref{c} that $\mathcal M^{\rm RGW}_{k_1+1}(L;\beta_1)^{(1)}$ is the union of the strata of $\mathcal M_{k+1}^{\rm RGW}(L;\beta)$ which are described by DD-ribbon trees with at least one positive level. The proof of Propositions \ref{lema362}, \ref{lem364}, \ref{lem365} is similar to that of Theorem \ref{lema362rev}. Thus, we only focus on the proof of Theorem \ref{lema362rev}. 

\begin{theorem}\label{lema362rev}
	Suppose $E$ is a positive real number and $N$ is a positive integer.
	There is a system of Kuranishi structures on the moduli spaces 
	$\{\mathcal M^{\rm RGW}_{k+1}(L;\beta)\}_{k,\beta}$ with $\omega\cap \beta\leq E$ and $k\leq N$
	such that if $\beta_1 + \beta_2 = \beta$, $k_1+k_2 =k$, then the space 
	$$
	  \mathcal M^{\rm RGW}_{k_1+1}(L;\beta_1)\,\hat\times_L\,
	  \mathcal M^{\rm RGW}_{k_2+1}(L;\beta_2)
	$$
	is a codimension one stratum of $\mathcal M^{\rm RGW}_{k+1}(L;\beta)$
	with the following properties.
	There exists a continuous map
	\begin{equation}\label{form6179}
		\aligned
		\Pi : &\mathcal M^{\rm RGW}_{k_1+1}(L;\beta_1)
		\,\hat\times_L\,\mathcal M^{\rm RGW}_{k_2+1}(L;\beta_2) \\
		&\to\mathcal M^{\rm RGW}_{k_1+1}(L;\beta_1)\,\times_L\,
		\mathcal M^{\rm RGW}_{k_2+1}(L;\beta_2)
		\endaligned
	\end{equation}
	which has the same properties as Proposition \ref{form371} (1),(2) and (3).                      
\end{theorem}


The proof of Theorem \ref{lema362rev} occupies the rest of this section.
For the proof, first we formulate the notion of obstruction bundle data\footnote
{This is related to but is different from the notion of obstruction data 
introduced in \cite[Definition 8.4]{part2:kura}.} in Definition \ref{defn684684}.
It is a way to associate an obstruction bundle to a neighborhood of each element of the moduli space. 
In  \cite{part2:kura}, we defined and used a version of such obstruction bundle data. For the proof of  Theorem \ref{lema362rev}, we need to slightly modify our choices of obstruction bundles so that they satisfy certain compatibility conditions at the boundary and corners.
The notion of  obstruction bundle data is introduced to state 
the required compatibility condition in a precise was as Definition \ref{defn687nnew} (disk-component-wise-ness).
Then the proof of Theorem \ref{lema362rev} is divided into two parts. 
We first show that the system of Kuranishi structures induced from a system of
disk-component-wise obstruction bundle data has the property stated in Theorem \ref{lema362rev}.
We then show the existence of disk-component-wise obstruction bundle data.

\begin{remark}
	Theorem \ref{lema362rev} concerns only the behavior of Kuranishi structures at 
	codimension one boundary components. 
	In fact, there is a similar statement for the behavior of our Kuranishi structures 
	at higher co-dimensional corners. This generalization to higher co-dimensional corners
	are counterparts of \cite[Condition 16.1 X, XI, XII]{fooonewbook}
	in the context of the stable map compactification\footnote{In \cite{foooast}, 
	the corresponding statement is called corner compatibility conditions.}.
	The main difference is that we need to replace $\times_L$ with $\hat\times_L$.
	To work out the whole construction of simultaneous perturbations, we need 
	the generalization of 	Theorem \ref{lema362rev} to the 
	higher co-dimensional corners.
	
	In Subsection \ref{subsub:componentwise}, 
	we will formulate a condition (Definition \ref{defn687nnew}) for the obstruction spaces, 
	which implies the consistency of Kuranishi structures 
	at the corners of arbitrary codimension.
	Since the proof and the statement for the case of corners is a straightforward 
	generalization of the case of boundary (but cumbersome to write in detail), 
	we focus on the case of codimension one boundary
	components.
\end{remark}

\begin{remark}
	The proof in this section is different from the approach in 
	\cite[Section 8]{fooo:const2}, where 
	the case of stable map compactification is treated. 
	In this section, we use target  parallel transportation. 
	On the other hand, in \cite[Section 8]{fooo:const2} extra marked points are added 
	to $\frak p \in \frak P(k+1,\beta)$ and are used to fix a diffeomorphism between open subsets of 
	the source domains. 
\end{remark}

\subsection{Disk component-wise-ness of the Obstruction Bundle Data}
\label{subsub:componentwise}
A {\it disk splitting tree} is defined to be a very detailed DD-ribbon tree $\mathcal S$ such that the color of all  vertices is ${\rm d}$. We say a detailed DD-ribbon tree $\check R$ belongs to a disk splitting tree $\mathcal S$ if $\mathcal S$ is obtained from $\check R$ by level shrinking and fine edge shrinking. (See \cite[Section 8]{part2:kura} for the definition of these combinatorial objects and operations.) In other words, 
geometrical objects with combinatorial type $\check R$ are limits of objects with type $\mathcal S$ such  that new disc bubble does not occur. However, it is possible to have sphere bubbles.

Let $\frak u \in \mathcal M^{\rm RGW}_{k+1}(L;\beta)$ and $\check R$ be the associated very detailed tree. Suppose $\mathcal S$ is a disk splitting tree such that $\check R$ belongs to $\mathcal S$. Let also $\lambda$ be the level  function assigned to $\check R$. For each interior vertex $\frak v$ of $\mathcal S$, let $\check R_{\frak v}$ be the subtree of $\check R$ given by the connected component of
\[
  \check R\setminus \bigcup_{e\in C^{\rm int}_1({\check R}),\, \lambda(e) = 0}e
\]
which contains the vertex $\frak v$.
Let $\overline{\mathcal S}$ be a disk splitting tree obtained from $\mathcal S$ by a sequence of shrinking of level $0$ edges \cite[Definition 3.104]{part1:top}. Let $\pi : \mathcal S \to \overline{\mathcal S}$ be the associated contraction map. For each $\frak w \in C^{\rm int}_{0}(\overline{\mathcal S})$, let ${\check R}(\frak w)$ be the very detailed DD-ribbon tree defined as
\begin{equation}\label{bewfirn182}
	{\check R}(\frak w)= \bigcup_{\pi(\frak v) = \frak w}\check R_{\frak v}
	\cup \bigcup_{\substack{e\in C^{\rm int}_1({\check R}),\, \lambda(e) = 0,\\ \, \text{$\pi(e)$ is adjacent to $\frak w$}}} e.
\end{equation}
Clearly we have $C^{\rm int}_0({\check R}(\frak w)) \subseteq C^{\rm int}_0({\check R})$ and $C^{\rm int}_1({\check R}(\frak w)) \subseteq C^{\rm int}_1({\check R})$. 

The restriction of the quasi-order\footnote{See \cite[Definition 3.55]{part1:top} for the definition of a quasi-order.} of $C^{\rm int}_0({\check R})$ to the set $C^{\rm int}_0({\check R}(\frak w))$ determines\footnote{See  \cite[Lemma 3.56]{part1:top}.} a level function $\lambda_{\frak w}$ for ${\check R}(\frak w)$. The tree ${\check R}(\frak w)$ also inherits a multiplicity function, a homology class assigned to each interior vertex and a color function from $\check R$, which turn ${\check R}(\frak w)$  into a very detailed tree associated to a detailed DD-ribbon tree $\mathcal R(\frak w)$.
There is a map
\begin{equation}\label{pifrakw}
  \pi_{\frak w} : \{1,\dots,\vert\lambda\vert\} \to \{1,\dots,\vert\lambda_{\frak w}\vert\}
\end{equation}
such that $i \le j$ implies $\pi_{\frak w}(i) \le \pi_{\frak w}(j)$ and for any $v \in C^{\rm int}_0({\mathcal R}(\frak w)) \subseteq C^{\rm int}_0({\mathcal R})$
\begin{equation}\label{form618010}
	\lambda_{\frak w}(v) = \pi_{\frak w}(\lambda(v)).
\end{equation}

Let $\Sigma_{\frak u}$ be the source curve of $\frak u$, and $\Sigma_{\frak u,v}$ denote the irreducible component of $\Sigma_{\frak u}$ corresponding to an interior vertex $v$ of $\check R$. For any $\frak w\in C^{\rm int}_{0}(\overline{\mathcal S})$, we define $\Sigma_{\frak u,\frak w}$ to be the union of irreducible components $\Sigma_{\frak u,v}$ where $v \in C^{\rm int}_0({\check R}(\frak w))$. A boundary marked point of $\Sigma_{\frak u,\frak w}$ 
is either a boundary marked point of a disc component $\Sigma_{\frak u,v}$ 
in $\Sigma_{\frak u,\frak w}$ or a boundary nodal point of $\Sigma_{\frak u}$ which joins an irreducible component of $\Sigma_{\frak u,\frak w}$ to an irreducible component of $\Sigma_{\frak u}$, which is not in $\Sigma_{\frak u,\frak w}$. The $0$-th boundary marked point $z_{0,\frak w}$ of $\Sigma_{\frak u,\frak w}$ is defined as follows. If the 0-th boundary marked point $z_0$ of  $\Sigma_{\frak u}$ is contained in $\Sigma_{\frak u,\frak w}$ then $z_{0,\frak w} = z_0$. If not, $z_{0,\frak w}$ is the boundary nodal point such that $z_0$ and $\Sigma_{\frak u,\frak w} \setminus \{z_{0,w}\}$ are contained in the different connected component of $\Sigma_{\frak u} \setminus \{z_{0,\frak w}\}$.

The restriction of $u_{\frak u} : (\Sigma_{\frak u},\partial\Sigma_{\frak u})  \to (X,L)$ to $\Sigma_{\frak u,\frak w}$ defines a map $u_{\frak u,\frak w} : (\Sigma_{\frak u,\frak w},\partial\Sigma_{\frak u,\frak w})  \to (X,L)$. The bordered nodal curve $\Sigma_{\frak u,\frak w}$ together with the boundary marked points described above, the choice of the $0$-th boundary marked point $z_{0,\frak w}$ and the map $u_{\frak u,\frak w}$ determines an element of the moduli space $\mathcal M_{k_{\frak w}+1}^{\rm RGW}(L;\beta(\frak w))$ where $\beta(\frak w) = \sum_{v \in C^{\rm int}_0(\check R(\frak w))}
\alpha(v)$ and $k_{\frak w}+1$ is the number of the boundary marked points of $\Sigma_{\frak u,\frak w}$.
We denote this element by $\frak u_{\frak w}$.

Let $\Xi_{\frak u} = (\vec w_{\frak u},(\mathcal N_{\frak u,v}), (\phi_{\frak u,v}), (\varphi_{\frak u,v,e}))$ be a TSD\footnote{See \cite[Definition 8.16]{part2:kura} for the definition of a TSD.} for $\frak u$. This induces a TSD $\Xi_{\frak u_{\frak w}}$ for $\frak u_{\frak w}$ in an obvious way. Let 
\begin{equation}\label{florm680}
	\frak y = (\vec{\frak x},\vec{\sigma},(u'_{v}),(U'_{v}),(\rho_{e}),(\rho_{i}))
\end{equation}
be an inconsistent map\footnote{See \cite[Definition 8.28]{part2:kura} for the definition of inconsistent maps.} with respect to $\Xi_{\frak u}$. Let $\mathcal S'$ be a disc splitting tree such that the very detailed tree of $\frak y$ belongs to $\mathcal S'$. We assume that $\overline {\mathcal S}$ is obtained from $\mathcal S'$ by a sequence of shrinking of level $0$ edges. Given $\frak w \in C^{\rm int}_{0}(\overline{\mathcal S})$, let $\sigma_{e}=0$ for any level $0$ edge $e\in C^{\rm int}_1({\check R})$ that corresponds to an exterior edge of ${\check R}(\frak w)$. Then we can define an inconsistent map $\frak y({\frak w})$ with respect to $\Xi_{\frak u_{\frak w}}$ in the following way. Since $C^{\rm int}_0({\check R}(\frak w)) \subseteq C^{\rm int}_0({\check R})$, $C^{\rm int}_1({\check R}(\frak w)) \subseteq C^{\rm int}_1({\check R})$, the restriction of the data of $\frak y$ determine $\vec{\frak x}_{\frak w}$, $\vec{\sigma}_{\frak w}$, $(u'_{v,\frak w})$, $(U'_{v,\frak w})$ and $(\rho_{e,\frak w})$. We also define:
$$
\rho_{\frak w,i} = \prod_{\hat i: \pi_{\frak w}(\hat i)=i} \rho_{\hat i}
$$
where $\pi_{\frak w}$ is given in \eqref{pifrakw}. 
\begin{lemma}
	The following element is an inconsistent map with respect to $\Xi_{\frak u_{\frak w}}$:
	\begin{equation}\label{form6182}
		\frak y(\frak w) = (\vec{\frak x}_{\frak w},\vec{\sigma}_{\frak w},(u'_{v,\frak w}),(U'_{v,\frak w}),
		(\rho_{e,\frak w}),(\rho_{\frak w,i})).
	\end{equation}
\end{lemma}

Next, we shall formulate a condition on the obstruction spaces so that the resulting system of Kuranishi structures satisfy the claims in Theorem \ref{lema362rev}. For this purpose, we firstly introduce the notion of {\it obstruction bundle data}.
\begin{definition}\label{defn684684}
	Suppose we are given vector spaces $\{E_{\frak u,\Xi}(\frak y)\}$ for any $\frak u\in \mathcal M^{\rm RGW}_{k+1}(L;\beta)$, any small enough TSD $\Xi$ at $\frak u$, and an inconsistent map $\frak y$ with
	respect to $\Xi$. This data is called an {\it obstruction bundle data} for $\mathcal M^{\rm RGW}_{k+1}(L;\beta)$ if the following holds.
	\begin{enumerate}
		\item  We have:
		\[
		 E_{\frak u,\Xi}(\frak y) = \bigoplus_{v \in C^{\rm int}_0(\check R)}E_{\frak u,\Xi,v}(\frak y)
		\]
		where $E_{\frak u,\Xi,v}(\frak y) \subset 
		L^2_{m,\delta}(\Sigma_{\frak y,v},\Lambda^{0,1} \otimes T)$.
		\item $E_{\frak u,\Xi,v}(\frak y)$ is a finite dimensional subspace.
			The supports of its elements are subsets of $\Sigma^{-}_{\frak y,v}$ and are 
			away from the boundary.
		\item $E_{\frak u,\Xi,v}(\frak y)$ is independent of $\Xi$ 
			in the sense of Definition \ref{defn6ten86}.
		\item $E_{\frak u,\Xi,v}(\frak y)$ is semi-continuous with respect to $\frak u$ in the sense of 
			Definition \ref{defn6ten86r3ev}.
		\item $E_{\frak u,\Xi,v}(\frak y)$ is of $C^{\ell}$ class with respect to $\frak y$ in the sense of 
			Definition \ref{defn6888}.
		\item The linearization of the Cauchy-Riemann equation is transversal to 
			$E_{\frak u,\Xi,v}(\frak y)$ in the sense of Definition \ref{defn68899}.
		\item $E_{\frak u,\Xi,v}(\frak y)$ is invariant under the $\Gamma_{\frak u}$-action.
			(See \cite[Definition 5.1 (5)]{fooo:const1}.)
	\end{enumerate}
\end{definition}
Definition \ref{defn684684} is the RGW counterpart of \cite[Definition 5.1]{fooo:const1} for the stable map compactification. Before discussing the precise meaning of (3), (4), (5) and (6), we define {\it disk-component-wise-ness} of a system of obstruction bundle data. This is the analogue of \cite[Definition 4.2.2]{foooast} for the stable map compactification:

\begin{definition}\label{defn687nnew}
	Suppose $E$ is a positive real number and $N$ is a positive integer.
	Suppose $\{E_{\frak u,\Xi}(\frak y)\}$ is a system of obstruction bundle data for the spaces
	$\{\mathcal M^{\rm RGW}_{k+1}(L;\beta)\}_{k,\beta}$ where 
	$k=0,1,2,\dots,N$, $\beta \in H_2(X,L)$ and $\beta \cap [\mathcal D] = 0$ with $\omega\cap \beta\leq E$.
	We say this system is {\it disk-component-wise} if we always have the identification
	\begin{equation}\label{form618383}
		 E_{\frak u,\Xi}(\frak y)=
		\bigoplus_{\frak w \in C^{\rm int}_0(\overline {\mathcal S})}E_{\frak u_{\frak w},\Xi_{\frak w}}(\frak y(\frak w)),
	\end{equation}
	where $\overline {\mathcal S}$ is a detailed DD-ribbon tree as in the beginning of the 
	subsection and  $\frak y(\frak w)$ is as in \eqref{form6182}.
\end{definition}

\subsubsection*{Explanation of Definition \ref{defn684684} (3)}
We pick two TSDs at $\frak u$ denoted by 
\[
  \Xi_{(j)} = (\vec w_{(j)},(\mathcal N_{(j),v}), (\phi_{(j),v}), (\varphi_{(j),v,e}),\kappa_{(j)}).
\] 
If $\Xi_{(2)}$ is small enough in compare to $\Xi_{(1)}$, then as in \cite[(9.56), (9.57)]{part2:kura}, we can assign to any inconsistent map: 
\[
  \frak y_{(2)}=(\vec{\frak x}_{(2)},\vec{\sigma}_{(2)},(u'_{(2),v}),(U'_{(2),v}),(\rho_{(2),e}),(\rho_{(2),i}))
\] 
with respect to $\Xi_{(2)}$ an inconsistent map
\[
  \frak y_{(1)}=(\vec{\frak x}_{(1)},\vec{\sigma}_{(1)},(u'_{(1),v}),(U'_{(1),v}),(\rho_{(1),e}),(\rho_{(1),i}))
\] 
with respect to $\Xi_{(1)}$. In particular, there is a bi-holomorphic embedding
$$
I_{v;\Xi_{(2)}\Xi_{(1)}} : \Sigma_{(1),v}^{-}(\vec{\frak x}_{(1)},\vec{\sigma}_{(1)}) \to \Sigma_{(2),v}^{-}(\vec{\frak x}_{(2)},\vec{\sigma}_{(2)})
$$
as in \cite[(9.62)]{part2:kura} such that
$$
\aligned
u'_{(2),v} \circ I_{v;\Xi_{(2)}\Xi_{(1)}} &= u'_{(1),v} 
\qquad \text{if $\lambda(v) = 0$}, \\
U'_{(2),v} \circ I_{v;\Xi_{(2)}\Xi_{(1)}} &= U'_{(1),v} 
\qquad \text{if $\lambda(v) > 0$}. 
\endaligned
$$
It induces a map
$$
  \frak I_{v;\Xi_{(2)}\Xi_{(1)}} : L^2_{m}(\Sigma^-_{(2),v}(\vec{\frak x}_{(2)},\vec{\sigma}_{(2)}),T \otimes \Lambda^{0,1})
  \to  L^2_{m}(\Sigma^-_{(1),v}(\vec{\frak x}_{(1)},\vec{\sigma}_{(1)}),T \otimes \Lambda^{0,1}).
$$
\begin{definition}\label{defn6ten86}
	We say the system $\{E_{\frak u,\Xi}(\frak y)\}$
	is {\it independent of $\Xi$}, if we always have:
	\begin{equation}
		\frak I_{v;\Xi_{(2)}\Xi_{(1)}}(E_{\frak u,\Xi_{(2)}}(\frak y_{(2)})=
		E_{\frak u,\Xi_{(1)}}(\frak y_{(1)}).
	\end{equation}
\end{definition}
The choices of obstruction bundles that we made in the previous section have this property. In fact, this property was used in the proof of \cite[Lemma 9.58]{part2:kura}.

\subsubsection*{Explanation of Definition \ref{defn684684} (4)}
Let $\frak u_{(1)}\in \mathcal M^{\rm RGW}_{k+1}(L;\beta)$ and $\Xi_{(1)}$ be a small enough TSD at $\frak u_{(1)}$. Let also $\frak u_{(2)} \in \mathcal M^{\rm RGW}_{k+1}(L;\beta)$ be in a neighborhood of $\frak u_{(1)}$ determined by $\Xi_{(1)}$ and $\Xi_{(2)}$ be a TSD at $\frak u_{(2)}$. We assume that $\Xi_{(1)}$, $\Xi_{(2)}$ satisfy \cite[Conditions 9.2 and 9.12]{part2:kura}. Let $\check R_{(j)}$ be the very detailed tree associated to $\frak u_{(j)}$. Our assumption implies that there is a map $\pi:\check R_{(1)} \to \check R_{(2)}$. Let $\frak y_{(2)}$ be an inconsistent map with respect to $\Xi_{(2)}$. We use \cite[Lemma 9.22]{part2:kura} to associate an inconsistent map $\frak y_{(1)}$ with respect to $\Xi_{(1)}$. In particular, for any $\hat v\in C^{\rm int}_0(\check R_{(1)})$ with $v:=\pi(\hat v)$, we have a bi-holomorphic isomorphism
$$
  I_{\hat v}:\Sigma_{(1),\hat v}^{-} \to \Sigma_{(2),v}^{-}
$$
such that 
$$
  \aligned
  u'_{(2),v} \circ I_{\hat v} &= u'_{(1),\hat v} 
  \qquad \text{if $\lambda(v) = 0$}, \\
  U'_{(2),v} \circ I_{\hat v} &= U'_{(1),\hat v} 
  \qquad \text{if $\lambda(v) > 0$}. 
  \endaligned
$$
It induces a map:
$$
  \frak I_{v;\frak y_{(1)}\frak y_{(2)}} : L^2_{m}(\Sigma^-_{(2),v},\Lambda^{0,1} \otimes T)
  \to \bigoplus_{\pi(\hat v)=v}L^2_{m}(\Sigma^-_{(1),\hat v},\Lambda^{0,1} \otimes T).
$$
\begin{definition}\label{defn6ten86r3ev}
	We say that $\{E_{\frak u,\Xi}(\frak y)\}$
	 {\it is semi-continuous with respect to $\frak u$} if the following property is satisfied.
	If $\frak u_{(1)}$, $\frak u_{(2)}$, $\frak y_{(1)}$, $\frak y_{(2)}$, $\Xi_{(1)}$ and $\Xi_{(2)}$ are as above, then we have
	\begin{equation}
		\frak I_{v;\frak y_{(1)}\frak y_{(2)}}(E_{\frak u_{(2)},\Xi_{(2)}}(\frak y_{(2)}))\subseteq
		E_{\frak u_{(1)},\Xi_{(1)}}(\frak y_{(1)}).
	\end{equation}
\end{definition}
\cite[Lemmas 9.39 and 9.41]{part2:kura} imply that our choices of obstruction bundles in \cite[Section 9]{part2:kura} satisfy the above property.

\subsubsection*{Explanation of Definition \ref{defn684684} (5)}
Let $\frak u \in \mathcal M^{\rm RGW}_{k+1}(L;\beta)$, $\check R$ be the very detailed tree associated to $\frak u$, and $\Xi$ be a choice of TSD at $\frak u$. Let also $\frak y= (\vec{\frak x},\vec{\sigma},(u'_{v}),(U'_{v}),(\rho_{e}),(\rho_{i}))$ be an inconsistent map with respect to $\Xi$. For $v \in C_{0}^{\rm int}(\check R)$, the TSD $\Xi$ determines an isomorphism $I_{\frak y,v} : \Sigma^-_{v}(\vec{\frak x},\vec{\sigma}) \to \Sigma_{v}^-(\vec{\frak x},\vec{\sigma_0}) $. Here $\vec{\sigma_0}$ is a vector with zero entries. If $\Xi$ is small enough, then $u_v \circ I_{\frak y,v}$ (resp. $U_v \circ I_{\frak y,v}$ in the case $c(v)={\rm D}$) is $C^2$-close to $u'_{\frak y,v}$ (resp. $U'_{\frak y,v}$).
Therefore, we obtain:
\[
  \frak I_{\frak y,v} : L^2_{m}(\Sigma^-_{\frak y,v}(\vec{\frak x},\vec{\sigma_0}) ;\Lambda^{0,1} \otimes T)
  \to L^2_{m}(\Sigma_v^{-}(\vec{\frak x},\vec{\sigma}) ;\Lambda^{0,1} \otimes T).
\]
Let $\mathscr L^2_{m+\ell}(u;v)$ be a small neighborhood of $u_v\vert_{\Sigma_v^{-}}$ or $U_v\vert_{\Sigma_v^{-}}$ with respect to the $L^2_{m+\ell}$-norm. 
\begin{definition}\label{defn6888}
	We say $\{E_{\frak u,\Xi}(\frak y)\}$
	is in $C^{\ell}$ with respect to $\frak y$, 
	if there exists a $C^{\ell}$ map
	$$
	\frak e_i : \prod_{e  \in  C^{\rm int}_1(\mathcal S)}\mathcal V_{e}^{{\rm deform}}
	\times \prod_{v  \in  C^{\rm int}_0(\mathcal S)}\mathcal V_{v}^{{\rm source}}
	\times \mathscr L^2_{m+\ell+1}(u;v)\to L^2_{m}(\Sigma_v^-;\Lambda^{0,1} \otimes T)
	$$
	for $i=1$, $\dots$, $\dim (E_{\frak u,\Xi}(\frak y))$ with the following properties.
	For the inconsistent map $\frak y$ with respect to $\Xi$ and $v \in C_{0}^{\rm int}(\check R)$,
	let $\frak y(v) \in  \mathscr L^2_{m+\ell+1}(u;v)$ be the map $u'_{\frak y,v} \circ I_{\frak y,v}$ or $U'_{\frak y,v} \circ I_{\frak y,v}$.
	Then the set of elements:
	$$
	  \frak I_{\frak y,v}\circ \frak e_i(\vec{\frak x},\vec{\sigma},\frak y(v))
	$$
	for $i=1$, $\dots$, $\dim (E_{\frak u,\Xi}(\frak y))$
	forms a basis for $E_{\frak u,\Xi}(\frak y)$.
\end{definition}
This condition is mostly the analogue of \cite[Definition 8.7]{foooexp} in the context of the stable map compactifications, and we refer the reader to the discussion there for a more detailed explanation. If this condition is satisfied, then the gluing analysis in \cite{part2:kura} gives rise to $C^{\ell}$-Kuranishi charts and $C^{\ell}$-coordinate changes. The proof of the fact that the choices of obstruction data in the previous section and Subsection \ref{subsub:existobst} satisfy this condition is similar to \cite[Subsection 11.4]{foooexp} and hence is omitted.
\begin{remark}\label{rem691new}
	We discussed the notion of $C^{\ell}$-obstruction data. There is also the notion of 
	smooth obstruction data which is slightly stronger.
	This is related to \cite[Definition 8.7 (3)]{foooexp}, and we 
	do not discuss this point in this paper. 
	This condition is necessary to construct smooth Kuranishi structures rather 
	than $C^{\ell}$-Kuranishi structures. 
	Kuranishi structures of class $C^{\ell}$ would suffice for our purposes of this paper.
	Smooth Kuranishi structures would be essential to study the Morse-Bott 
	case and/or construct filtered $A_{\infty}$-category based on de-Rham model.
\end{remark}

\subsubsection*{Explanation of Definition \ref{defn684684} (6)}

We consider $\frak u \in \mathcal M^{\rm RGW}_{k+1}(L;\beta)$ and $\Xi$. A system $\{E_{\frak u,\Xi}(\frak y)\}$ determines the vector spaces $E_{\frak u,\Xi}(\frak u)$ in the case that $\frak y =\frak u$.

\begin{definition}\label{defn68899}
	We say the linearization of the Cauchy-Riemann equation is transversal to
	$E_{\frak u,\Xi}(\frak u)$, if $L^2_{m,\delta}(\frak u;T \otimes \Lambda^{0,1})$
	is generated by the image of the operator $D_{\frak u}\overline\partial$ 
	in \cite[(8.3)]{part2:kura} and $E_{\frak u,\Xi}(\frak u)$.
\end{definition}
%
%
%

\subsubsection*{From Disk-component-wise-ness to Theorem \ref{lema362rev}}
The construction of \cite[Section 9]{part2:kura} implies that we can use an obstruction bundle data to construct a 
Kuranishi structure. The next lemma shows that to prove Theorem \ref{lema362rev}, it suffices to find a
system of obstruction bundle data which is disk-component-wise.

\begin{lemma}\label{lem684}
	If a system of obstruction bundle data is disk-component-wise,
	then the Kuranishi structures constructed in the last section
	on moduli spaces $\mathcal M^{\rm RGW}_{k+1}(L;\beta)$
	satisfy the claims in Theorem \ref{lema362rev}.
\end{lemma}

\begin{proof}
	This is in fact true by tautology. For the sake of completeness, we give the proof below.
	Let $\frak u \in \mathcal M^{\rm RGW}_{k+1}(L;\beta)$, $\check R$ be the 
	very detailed DD-ribbon tree associated to $\frak u$, and $\mathcal S$ be the 
	disk splitting tree such that $\check R$ belongs to $\mathcal S$.
	We assume that $\frak u$ is a boundary point, i.e., there are $k_1$, $k_2$, $\beta_1$ and 
	$\beta_2$ such that $\frak u$ is contained in
	$\mathcal M^{\rm RGW}_{k_1+1}(L;\beta_1)\,\hat\times_L\,
	\mathcal M^{\rm RGW}_{k_2+1}(L;\beta_2).$
	 In particular, the disk splitting tree $\overline{\mathcal S}$ 
	 in Figure \ref{Figuresimplegraph} is obtained 
	from $\mathcal S$ by shrinking of level 0 edges. We also have a map
	$\pi : \mathcal S \to \overline{\mathcal S}$.
	The construction of the beginning of this subsection allows us to from 
	$\frak u_{\frak w_1}\in \mathcal M^{\rm RGW}_{k_1+1}(L;\beta_1)$ and 
	$\frak u_{\frak w_2}\in \mathcal M^{\rm RGW}_{k_2+1}(L;\beta_2)$ from $\frak u$.
	Here $\frak w_1$, $\frak w_2$ are the two interior vertices of $\overline{\mathcal S}$.
	(See Figure \ref{Figuresimplegraph}.)
	The map $\Pi$ in (\ref{form6179}) is given by
	$\Pi(\frak u) = (\frak u_{\frak w_1},\frak u_{\frak w_2})$.
	Let $\overline{\frak u} = (\frak u_{\frak w_1},\frak u_{\frak w_2})$.

        \begin{figure}[h]
        \centering
        \includegraphics[scale=0.6]{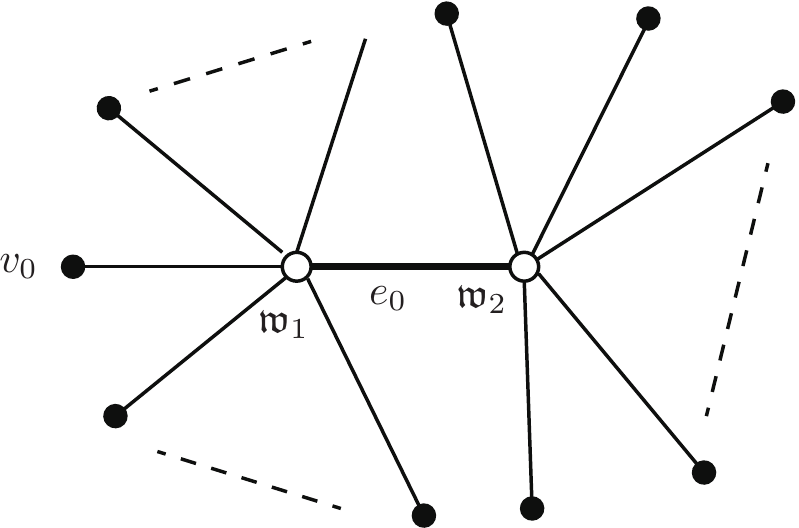}
        \caption{$\overline{\mathcal S}$.}
        \label{Figuresimplegraph}
        \end{figure}

	A Kuranishi neighborhood of $\frak u$ in $\mathcal M^{\rm RGW}_{k_1+1}
	(L;\beta_1)\,\hat\times_L\,\mathcal M^{\rm RGW}_{k_2+1}(L;\beta_2)$
	gives a Kuranishi neighborhood of $\frak u$ 
	in a {\it normalized boundary} of $\mathcal M^{\rm RGW}_{k+1}(L;\beta)$.
	It contains inconsistent solutions 
	$\frak y = (\vec{\frak x},\vec{\sigma},(u'_{v}),(U'_{v}),(\rho_{e}),(\rho_{i}))$
	with respect to $\Xi$ such that $\sigma_{e_0} = 0$.
	Here $e_0$ is the unique interior edge of level 0 of $\overline{\mathcal S}$.
	We may regard $e_0$ as an edge of ${\mathcal S}$ and $\check R$, too.
        We denote this set by $\partial_{e_0} \mathcal U(\frak u;\Upsilon)$, 
        where $\Upsilon= (\Xi,\{E_{\frak u,\Xi}\})$.

        The TSD $\Xi$ induces the TSD $\Xi_j$ on $\frak u_{\frak w_j}$ for $j=1,2$.
        We define  $\Upsilon_j = (\Xi_j,\{E_{\frak u_{\frak w_j},\Xi_j}\})$, which is a TSO.
        Then we obtain a Kuranishi neighborhood $\mathcal U(\frak u_{\frak w_j};\Upsilon_j)$
        of $\frak u_{\frak w_j}$ 
        in the moduli space $\mathcal M^{\rm RGW}_{k_j+1}(L;\beta_j)$, for $j=1,2$.
        We can define evaluation maps ${\rm ev}_{j,i} : \mathcal U(\frak u_{\frak w_j};\Upsilon_j) 
        \to L$ for $i=0,\dots,k_j$ and define 
        \begin{equation}\label{form6186}
        \mathcal U(\frak u_{\frak w_1};\Upsilon_1) \,{}_{{\rm ev}_{1,i}}\times_{{\rm ev}_{2,0}} \, 
        \mathcal U(\frak u_{\frak w_2};\Upsilon_2).
        \end{equation}
        Here $i$ is determined so that the edge $e_0$ is the $i$-th edge of $\frak w_1$.
        (\ref{form6186}) is a Kuranishi neighborhood of $(\frak u_{\frak w_1},\frak u_{\frak w_2})$ 
        in the fiber product Kuranishi structure of 
        $\mathcal M^{\rm RGW}_{k_1+1}(L;\beta_1)
        \,\times_L\,
        \mathcal M^{\rm RGW}_{k_2+1}(L;\beta_2)$.
        \par
        We next define a map
        $$
        F_{\frak u} : \partial_{e_0} \mathcal U(\frak u;\Upsilon)
        \to 
        \mathcal U(\frak u_{\frak w_1};\Upsilon_1) \,{}_{{\rm ev}_{1,i}}\times_{{\rm ev}_{2,0}} \, 
        \mathcal U(\frak u_{\frak w_2};\Upsilon_2).
        $$
	For $j=1,2$, let $\check R(\frak w_j)$ be the very detailed DD-ribbon tree associated to 
	$\frak w_j$, defined in the beginning of this subsection.
	Given an inconsistent solution $\frak y\in \partial_{e_0} \mathcal U(\frak u;\Upsilon)$, we can define
	$\frak y_{(j)} = (\vec{\frak x}_{(j)},\vec{\sigma}_{(j)},(u'_{(j),v}),
        (U'_{(j),v}),(\rho_{(j),e}),(\rho_{(j),i}))$, an inconsistent solution with respect to 
        $\Upsilon_{j}$, as in \eqref{form6182}. Identity \eqref{form618383} implies that 
        $\frak y_{(j)}$ satisfies \cite[(8.33), (8.34)]{part2:kura}, 
        the thickened non-linear Cauchy-Riemann equations.
        Thus $\frak y_{(j)}$ is an inconsistent solution with respect to $\Upsilon_j$ for $j=1,2$.
	Since $\frak y$ is an inconsistent solution with $\sigma_{e_0}=0$, we also have
        (see \cite[Definition 8.28 (10)]{part2:kura})
        $$
        {\rm ev}_{1,i}(\frak y_{(1)}) = {\rm ev}_{2,0}(\frak y_{(2)}).
        $$
        Let $F_{\frak u}(\frak y)=(\frak y_{(1)},\frak y_{(2)})$.
        We have
        \begin{equation}\label{inc-isot-gp}
	        {\rm Aut}(\frak u) \subseteq {\rm Aut}(\frak u_{\frak w_1}) \times {\rm Aut}(\frak u_{\frak w_2}).
        \end{equation}
        because the restriction of all automorphisms to disk components are identity maps.
        Thus any $\gamma\in {\rm Aut}(\frak u)$ maps the sources curves of $\frak u_{\frak w_1}$
        and $\frak u_{\frak w_2}$ to themselves. Consequently, $\gamma$ induces 
        $(\gamma_1,\gamma_2)\in {\rm Aut}(\frak u_{\frak w_1}) \times {\rm Aut}(\frak u_{\frak w_2})$
        which determines $\gamma$ uniquely\footnote{However,
        any $(\gamma_1,\gamma_2)\in {\rm Aut}(\frak u_{\frak w_1}) \times {\rm Aut}(\frak u_{\frak w_2})$ does not
        necessarily determine an element of ${\rm Aut}(\frak u)$. For example, we could have two vertices $v_1$ and $v_2$ with the same positive levels such that $v_i$ belongs to  
        $C^{\rm int}_0(\check R(\frak w_i))$. 
        Then there is $c_i\in \bbC_*$ such that $U_{v_i}\circ \gamma_i=c_i \cdot U_{\gamma_i(v_i)}$. In the case that $c_1\neq c_2$, we cannot produce an automorphism of $\frak u$ using $\gamma_1$, $\gamma_2$.}.
        It is then easy to see that $F_{\frak u}$ is ${\rm Aut}(\frak u)$-invariant.
        
        By (\ref{form618383}) we have
        $$
        \mathcal E_{0,\frak u,\Xi}(\frak y)
        \cong
        \bigoplus_{j=1,2}\mathcal E_{0,\frak u_{\frak w_j},\Xi_j}(\frak y_{(j)})
        $$
        We also have
        $$
        \bigoplus_{e \in C^{\rm int}_{\rm th}(\check R), \lambda(e) > 0} \mathscr L_e\cong
        \bigoplus_{j=1,2}\bigoplus_{e \in C^{\rm int}_{\rm th}(\check R(\frak w_j)), \lambda(e) > 0} 
        \mathscr L_e.
        $$
        This is because the set of the edges of positive level of $\check R$ is 
        the union of the set of the edges of positive level of $\check R(\frak w_1)$ 
        and $\check R(\frak w_2)$. Therefore, we obtain a bundle map
        $$
        \tilde F_{\frak u} : \mathcal E_{\frak u,\Xi} \to \mathcal E_{\frak u_{\frak w_1},\Xi_{1}}
        \oplus \mathcal E_{\frak u_{\frak w_2},\Xi_{2}},
        $$
        which is a lift of $F_{\frak u}$. The bundle map $\tilde F_{\frak u}$ is an isomorphism on each fiber.
        Therefore, we proved (a), (b) and (c) of Theorem \ref{lema362rev} (2).
        Item (d), compatibility with Kuranishi maps, and (e), compatibility with the parametrization maps,
        are obvious from the construction.
        Item (3), compatibility with the coordinate change,
        is also an immediate consequence of the definitions.
        
        It remains to prove that $F_{\frak u}$ is an isomorphism outside the strata 
        of codimension 2. For this purpose, it suffices to consider the cases where $
        \check R(\frak w_1)$ and $\check R(\frak w_2)$
        have no vertex of positive level. Note that if we ignore the parameter $\rho_i$, 
        then the map $F_{\frak u}$ is a bijection. 
        In the present case where there is no vertex of positive level,
        there is no parameter $\rho_i$.
        This completes the proof of Lemma \ref{lem684}.
\end{proof}

\subsection{Existence of disk-component-wise Obstruction Bundle Data}
\label{subsub:existobst}
The main goal of this subsection is to prove:
\begin{prop}\label{lem685}
	There is a system of obstruction bundle data which is disk-component-wise.
\end{prop}
The proof is divided into 5 parts. In the first three parts, we define various objects (OBI)-(OBIII) and formulate certain conditions we require them to satisfy. In Part 4, we show that one can use these objects to obtain a system of obstruction bundle data which is disk-component-wise. Finally, in Part 5, we show the existence of objects satisfying the required conditions.

\subsubsection*{Disk-component-wise Obstruction Bundle Data: Part 1}

Suppose $E$ is a positive real number and $N$ is a positive integer. Let $\mathscr{TP}$ be the set of all pairs $(k,\beta)$ such that $\mathcal M_{k+1}^{\rm RGW}(L;\beta) \ne \emptyset$, $\omega\cap \beta\leq E$ and $k\leq N$. Let $(k,\beta), (k',\beta') \in \mathscr{TP}$ we say $(k',\beta') < (k,\beta)$ if $\beta'\cap \omega < \beta \cap \omega$ or $\beta'\cap \omega = \beta \cap \omega$, $k' < k$. We also say $(k',\beta') \le (k,\beta)$ if $(k',\beta') < (k,\beta)$ or $(k',\beta') = (k,\beta)$. By Gromov compactness theorem, for each  $(k,\beta) \in \mathscr{TP}$ the set $\{ (k',\beta') \in \mathscr{TP} \mid (k',\beta') < (k,\beta)\}$ is a finite set.
\par\medskip
\noindent {\bf (OBI):} 
For $(k,\beta) \in \mathscr{TP}$, $\frak P(k+1,\beta)$ is a finite subset of ${\rm Int}(\mathcal M_{k+1}^{\rm RGW}(L;\beta))$, the interior of the moduli space $\mathcal M_{k+1}^{\rm RGW}(L;\beta)$. To be more specific, the space ${\rm Int}(\mathcal M_{k+1}^{\rm RGW}(L;\beta))$ consists of elements that their source curves have only one disc component.
\par\smallskip
Let $\frak p \in \frak P(k+1,\beta)$. We write $\Sigma_{\frak p}$ for the source curve of $\frak p$ and $u_{\frak p} : (\Sigma_{\frak p},\partial \Sigma_{\frak p}) \to (X,L)$ for the map part of $\frak p$. Let  $\check R_{\frak p}$ be the very detailed  tree describing the combinatorial type of $\frak p$.
For $v \in C^{\rm int}_{0}(\check R_{\frak p})$, we denote the corresponding component of $\Sigma_{\frak p}$ by $\Sigma_{\frak p_v}$ and the restriction of $\frak p$ to $\Sigma_{\frak p_v}$ by $\frak p_{v}$.
\par\smallskip
\noindent {\bf (OBII):} 
For any $v \in C^{\rm int}_{0}(\check R_{\frak p})$, we 
take a finite dimensional subspace
$$
  E_{\frak p_{v}} \subseteq 
  \begin{cases}
  C^{\infty}(\Sigma_{\frak p_{v}};u^*_{\frak p_{v}}TX \otimes \Lambda^{0,1})&\text{if $c(v)={\rm d}$ or ${\rm s}$}, \\
  C^{\infty}(\Sigma_{\frak p_{v}};u_{\frak p_{v}}^*T\mathcal D \otimes \Lambda^{0,1})&\text{if $c(v)={\rm D}$},
  \end{cases}
$$ 
whose support is away from nodal and marked points and the boundary of $\Sigma_{\frak p_v}$. 
\par\smallskip
We require:
\begin{conds}
	The restriction of $u_{\frak p}$ to a neighborhood of the support of ${\rm Supp}(E_{\frak p_{v}})$ 
	is a smooth embedding. In particular, if ${\rm Supp}(E_{\frak p_{v}})$ is nonzero, $u_{{\frak p}_{v}}$ is non-constant.
\end{conds}

\subsubsection*{Disk-component-wise Obstruction Bundle Data: Part 2} Fix $\frak u=(\Sigma_{\frak u,v},z_{\frak u, v},u_{\frak u, v}; v\in C_0^{\rm int}(\check R_{\frak u}))\in \mathcal M^{\rm RGW}_{k+1}(L;\beta)$, where $\check R_{\frak u}$ is the very detailed tree assigned to $\frak u$. 

There is a forgetful map from the moduli space $\mathcal M^{\rm RGW}_{k+1}(L;\beta)$ to the moduli space of stable discs $\mathcal M_{k+1}^{\rm d}$, where for any $\frak u\in \mathcal M^{\rm RGW}_{k+1}(L;\beta)$ we firstly forget all the data of $\frak u$ except the source curve $\Sigma_{\frak u}$, and then shrink the unstable components. There is a metric space $\mathcal C_{k+1}^{\rm d}$, called the {\it universal family}, with a map $\pi:\mathcal C_{k+1}^{\rm d} \to \mathcal M_{k+1}^{\rm d}$ such that $\pi^{-1}(\zeta)$, for $\zeta\in \mathcal M_{k+1}^{\rm d}$, is a representative for $\zeta$. (See, for example, \cite[Section 2]{fooo:const1} or 
\cite[Subsection 4.1]{part1:top}.) We pull-back $\mathcal C_{k+1}(L;\beta)$ to $\mathcal M^{\rm RGW}_{k+1}(L;\beta)$ via the forgetful map to obtain the space $\mathcal C^{\rm RGW}_{k+1}(L;\beta)$ with the projection map $\pi_{\rm RGW}:\mathcal C^{\rm RGW}_{k+1}(L;\beta) \to \mathcal M^{\rm RGW}_{k+1}(L;\beta)$. The pull-back of the metric on $\mathcal C_{k+1}(L;\beta)$ to $\mathcal C^{\rm RGW}_{k+1}(L;\beta)$ defines a quasi metric\footnote{A quasi-metric is a distance function which satisfies the reflexive property and triangle inequality. But we allow for two distinct points to have distance zero.} on $\mathcal C^{\rm RGW}_{k+1}(L;\beta)$. Here we obtain a quasi-metric because the forgetful map from $\mathcal M^{\rm RGW}_{k+1}(L;\beta)$ to $\mathcal M_{k+1}^{\rm d}$ is not injective. Note that this quasi metric is in fact a metric in each fiber $\pi_{\rm RGW}^{-1}(\frak u)$. The fiber $\pi_{\rm RGW}^{-1}(\frak u)$ can be identified with a quotient of $\Sigma_{\frak u}$. Thus by pulling back the metric on each fiber $\pi_{\rm RGW}^{-1}(\frak u)$, we define a quasi metric on the source curve $\Sigma_{\frak u}$ of $\frak u$. 

\begin{lemma}\label{delta-k-b}
	For each $\beta$ and $k$, there is a positive constant $\delta(k,\beta)$ with the following property. If $\frak u\in  \mathcal M^{\rm RGW}_{k+1}(L;\beta)$
	and $v\in C^{\rm int}_0(\check R_{\frak u})$ is a vertex with $u_{\frak u,v}:\Sigma_{\frak u,v}\to X$ being a  non-constant map, then there is $x\in \Sigma_{\frak u,v}$
	such that the distance between $x$ and any nodal point and boundary point of $\Sigma_{\frak u}$ is greater than $\delta(k,\beta)$. Moreover, if $x'\in \Sigma_{\frak u}$ is chosen such 
	that $u_{\frak u}(x)=u_{\frak u}(x')$, then the distance between $x$ and $x'$ is greater than $\delta(k,\beta)$.
\end{lemma}
\begin{proof}
	Given any $\frak u$, there is a constant $\delta(k,\beta,\frak u)$ such that the lemma holds for any non-constant irreducible component $u_{\frak u,v}$ of $u_{\frak u}$.
	In fact, there is a neighborhood $\mathcal U(\frak u)$ such that the lemma holds for the constant $\delta(k,\beta,\frak u)$ and any $\frak u' \in \mathcal U(\frak u)$.
	Now we can conclude the lemma using compactness of $\mathcal M^{\rm RGW}_{k+1}(L;\beta)$.
\end{proof}

In the following definition, $\epsilon(k',\beta')$ is a constant which shall be fixed later.

\begin{definition}\label{def1014} 
	A triple $(\frak u,\frak p,\phi)$ is said to be a {\it local approximation} to ${\frak u}$, if we have:
	\begin{enumerate}
		\item There is $(k',\beta')\le (k,\beta)$ such that $\frak p \in \frak P(k'+1,\beta')$.
		\item $\phi$ is a smooth embedding from a neighborhood of 
		$\bigcup_v {\rm Supp}(E_{\frak p_{v}})$ to $\Sigma_{\frak u}$. 
		If $x$ belongs to the image of $\phi$, 
		then its distance to the nodal points in $\Sigma_{\frak u}$ is greater than $\delta(k',\beta')$. 
		For each $v\in C_0^{\rm int}(\check R_{\frak p})$, there is $v'\in C_0^{\rm int}(\check R_{\frak u})$
		such that $\phi$ maps ${\rm Supp}(E_{\frak p_{v}})$ to $\Sigma_{\frak u,v'}$.
		Furthermore, if $x'$ is another point in the source curve of $\frak u$ such that 
		$u_{\frak u}(x)=u_{\frak u}(x')$, then the distance between $x$ and $x'$ is greater than 
		$\delta(k',\beta')$.
		\item For each $v$, we require:
			$$
			  d_{C^2;{\rm Supp}(E_{\frak p_{v}})}(u_{\frak u} \circ \phi,u_{\frak p}) < \epsilon(k',\beta').
			$$
		\item $\phi$ satisfies the following point-wise inequality:
			$$
			  \vert \overline\partial \phi \vert < \vert \partial \phi \vert/100.
			$$
	\end{enumerate}
\end{definition}
The next definition is similar to \cite[Definition 4.8]{part2:kura}.
(See also the discussion right after \cite[Definition 8.17]{part2:kura}.)
\begin{definition}\label{defn1015}
	Let $(\frak u,\frak p,\phi)$ be a local approximation to ${\frak u}$. 
	We say a map $\hat\phi$ from a neighborhood of $\bigcup_v  {\rm Supp}(E_{\frak p_{v}})$ to 
	$\Sigma_{\frak u}$ is a {\it normalization} of $\phi$ if the following holds.
	\begin{enumerate}
	\item  If $x$ belongs to the image of $\hat \phi$, then its distance to the nodal points in 
		$\Sigma_{\frak u}$ is greater than $\frac{1}{3}\delta(k',\beta')$. Furthermore, if $x'$ is another point in the source 
		curve of $\frak u$ such that $u_{\frak u}(x)=u_{\frak u}(x')$, then the distance between $x$ and $x'$ 
		is greater than $\frac{1}{3}\delta(k',\beta')$.
	\item For each $v$, we require
		$$
		d_{C^2;{\rm Supp}(E_{\frak p_{v}})}(u_{\frak u} \circ \hat\phi,u_{\frak p}) < 2\epsilon(k',\beta'),
		$$
		and
		\[
		d_{C^0;{\rm Supp}(E_{\frak p_{v}})}(\hat\phi,\phi) < \frac{\delta(k',\beta')}{3}.
		\]		
	\item Let $z$ be in the neighborhood of ${\rm Supp}(E_{\frak p_{v}})$.
		\begin{enumerate}
			\item Suppose $z$ is in a component with color ${\rm d}$ or ${\rm s}$.
				We take the unique minimal geodesic $\gamma$ in $X\setminus \mathcal D$ 
				(with respect to the metric $g$), which joins $u_{\frak p}(z)$
				to $(u_{\frak u} \circ \hat\phi)(z)$.
				Then
				$$
				  \frac{d\gamma}{dt}(0) \perp T_{u_{\frak p}(z)} u_{\frak p}(\Sigma_{\frak p}).
				$$
			\item Suppose $z$ and $\hat\phi(z)$ are in a component with color ${\rm D}$.
				We take the unique minimal geodesic $\gamma$ in $\mathcal D$ 
				(with respect to the metric $g'$), which joins $u_{\frak p}(z)$
				to $(u_{\frak u} \circ \hat\phi)(z)$.
				Then
				$$
				  \frac{d\gamma}{dt}(0) \perp T_{u_{\frak p}(z)} u_{\frak p}(\Sigma_{\frak p}).
				$$
			\item Suppose $z$ is in a component with color ${\rm D}$ and 
				$\hat\phi(z)$ is in a component with color ${\rm d}$ or ${\rm s}$.
				We take the unique minimal geodesic $\gamma$ in $\mathcal D$ 
				(with respect to the metric $g'$), which joins $u_{\frak p}(z)$
				to $(\pi\circ u_{\frak u} \circ \hat\phi)(z)$.
				Then
				$$
				  \frac{d\gamma}{dt}(0) \perp T_{u_{\frak p}(z)} u_{\frak p}(\Sigma_{\frak p}).
				$$
		\end{enumerate}
	\end{enumerate}
\end{definition}

\begin{lemma}
	If the constant $\epsilon(k',\beta')$ is small enough, then for any local approximation $(\frak u,\frak p,\phi)$ to 
	${\frak u} \in \mathcal M^{\rm RGW}_{k+1}(L;\beta)$, there exists a normalization $ \hat\phi$ of $\phi$ and for 
	any other normalization $ \hat\psi$ of $\phi$, we have:
	\[
	  \hat\phi|_{\bigcup_v  {\rm Supp}(E_{\frak p_{v}})} = \hat\psi|_{\bigcup_v  {\rm Supp}(E_{\frak p_{v}})}.
	\]
\end{lemma}
From now on, we assume that the constant $\epsilon(k',\beta')$ satisfies the assumption of this lemma.
\begin{proof}
	This follows from the implicit function theorem and compactness of $\mathcal M^{\rm RGW}_{k+1}(L;\beta)$.
\end{proof}

\begin{definition}
	For $j=1,2$, suppose $(\frak u,\frak p_j,\phi_j)$ is a local approximation to $\frak u$. We say these two 
	approximations are {\it equivalent} if $\frak p_1 =\frak p_2$ and 
	$$
	  \hat\phi_1|_{\bigcup_v  {\rm Supp}(E_{\frak p_{v}})} = \hat\phi_2|_{\bigcup_v  {\rm Supp}(E_{\frak p_{v}})}.
	$$ 
	This is obviously an equivalence relation. Each equivalence class is called a {\it quasi component} (of $\frak u$). See Figure \ref{quasi-comp} 
	for the schematic picture of a quasi-component. 
\end{definition}

\begin{figure}[h]
	\centering
	\includegraphics[scale=0.6]{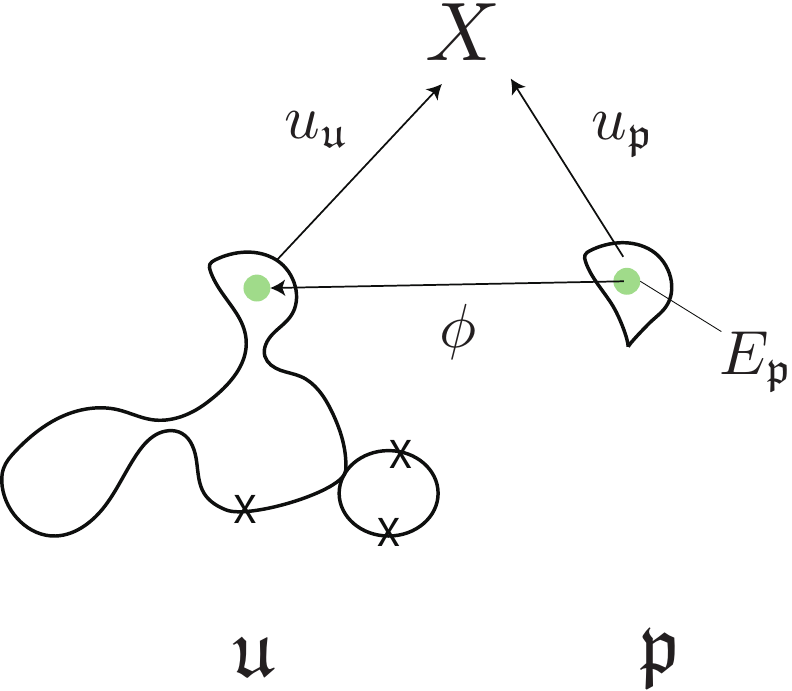}
	\caption{$[\frak u,\frak p,\phi]$ is a quasi-component of $\frak u$}
	\label{quasi-comp}
\end{figure}

We next define obstruction spaces $E_{\frak u,{\bf p},\Xi}(\frak y)$ where ${\bf p}$ is a quasi component of $\frak u$ and ${\frak y}= (\vec{\frak x},\vec{\sigma},(u'_{v}),(U'_{v}),(\rho_{e}),(\rho_{i}))$ is an inconsistent map with respect to a TSD $\Xi$ at $\frak u$. The definition is similar to the corresponding definitions in \cite[Subsection 9.2]{part2:kura}. The TSD $\Xi$ induces a holomorphic embedding
\[
  \psi_{\frak y,\frak u}: \bigcup_{v\in C^{\rm int}_0(\check R_{\frak u})}\Sigma_{\frak u,v}^-\to \Sigma(\vec{\frak x},\vec{\sigma}).
\]
By taking $\Xi$ to be small enough, we may assume the image of $\phi$ is contained in the domain of $\psi_{\frak y,\frak u}$. Define
$$
\phi_{\frak y,\frak p} = \psi_{\frak y,\frak u} \circ \phi.
$$
Using Implicit function theorem, we can modify $\phi_{\frak y,\frak p}$ to obtain $\hat\phi_{\frak y;\frak p}$ from a neighborhood of $\bigcup_v {\rm Supp}(E_{\frak p_{v}})$ to $\Sigma_{\frak y}$ such that the analogue of Definition \ref{defn1015} is satisfied.
This map is clearly independent of the representative of ${\bf p}$ and
also independent of $\Xi$ if this TSD is sufficiently small.

By replacing $I^{\rm t}_{\rm d}(x)$, $I^{\rm t}_{\rm s}(x)$ and $I^{\rm t}_{\rm D}(x)$ with $\hat\phi_{\frak y;\frak p}$ and using the vector spaces $E_{\frak p_{v}}$, we obtain
\begin{equation}\label{form6189}
	 E_{\frak u,{\bf p},\Xi}(\frak y)
	\subset \bigoplus_{v \in C^{\rm int}_0(\check R_{\frak y})} L^2_{m}(\Sigma_{\frak y,v};T \otimes \Lambda^{0,1})
\end{equation}
in the same way as in \cite[(4.11)]{part2:kura}. Here $\check R_{\frak y}$ is the very detailed tree describing the combinatorial type of $\frak y$.

\subsubsection*{Disk-component-wise Obstruction Bundle Data: Part 3}
Our obstruction bundle data $E_{\frak u,\Xi}(\frak y)$ is a direct sum of $E_{\frak u,{\bf p},\Xi}(\frak y)$ for an appropriate set of quasi components ${\bf p}$ of $\frak u$. Our next task is to find a way to choose this set of quasi components.
\begin{definition}
	For ${\frak u} \in \mathcal M_{k+1}^{\rm RGW}(L;\beta)$, we denote by $\mathscr{QC}(k,\beta)(\frak u)$
	the set of all quasi components of ${\frak u}$. Let
	$$
	  \mathscr{QC}(k,\beta) := \bigcup_{{\frak u} \in \mathcal M_{k+1}(L;\beta)} \mathscr{QC}(k,\beta)(\frak u).
	$$
	The map
	\begin{equation} \label{Pi}	
	  \Pi : \mathscr{QC}(k,\beta) \to \mathcal M_{k+1}^{\rm RGW}(L;\beta)
	\end{equation}
	is the obvious projection.
\end{definition}

\begin{lemma}
	If the constant $\epsilon(k',\beta')$ is small enough, then $\mathscr{QC}(k,\beta)(\frak u)$ is a finite set.
\end{lemma}
\begin{proof}

	By Gromov compactness, there is only a finite number of $(k',\beta') \le (k,\beta)$ such that
	$\mathcal M^{\rm RGW}_{k'+1}(L;\beta') \ne \emptyset$.
	Let $(k',\beta')$ be such a pair and $\frak p$ be an element of the finite set $\frak P(k'+1,\beta')$.
	Assume $y$ is an element of the source curve of $\frak p$. There is a neighborhood $U_y$ of $y$
	in the source curve of $y$ such that if $\epsilon'(k',\beta')$ is small enough, then the following holds.
	Let $[\frak u,\frak p,\phi]$ and $[\frak u,\frak p,\psi]$ be two quasi components of an element 
	$\frak u\in \mathcal M^{\rm RGW}_{k+1}(L;\beta)$ with $\hat \phi$ and $\hat \psi$ being the normalizations of 
	$\phi$ and $\psi$. 
	If  $\hat \phi(y)\neq\hat \psi(y)$, then $\hat \phi|_{U_y}$ and $\hat \psi|_{U_y}$ are disjoint. This would imply that
	given the element $\frak u$, there are only finitely many possibilities for the restriction of the normalization 
	map to $U_y$. Therefore, there are finitely many quasi components of the form 
	$[\frak u,\frak p,\phi]$ for $\frak u$. Finiteness of the sets $\frak P(k'+1,\beta')$ completes the proof.
%
%
\end{proof}
We next define a topology on $\mathscr{QC}(k,\beta)$.
Let ${\frak u} \in \mathcal M_{k+1}^{\rm RGW}(L;\beta)$, $\Xi$ be a TSD at $\frak u$ and $\mathfrak U({\frak u},\Xi)$ be the associated set of inconsistent solutions. (See \cite[Definition
8.28]{part2:kura}.)
We construct a map
$$
  \frak I : \mathfrak U({\frak u},\Xi)\cap\mathcal M_{k+1}^{\rm RGW}(L;\beta) 
  \to \mathscr{QC}(k,\beta).
$$
such that $\Pi \circ \frak I = {\rm id}$ assuming $\Xi$ is small enough. The TSD $\Xi$ induces a map
$$
  \psi_{\frak u',\frak u}: \bigcup_{v\in C^{\rm int}_0(\check R_{\frak u})}
  \Sigma_{\frak u,v}^-\to \Sigma_{\frak u'}
$$
for $\frak u' \in \mathfrak U({\frak u},\Xi)$. Let $(\frak u,\frak p,\phi)$ be a  local approximation to ${\frak u}$. If $\Xi$ is sufficiently small, then $(\frak u,\frak p,\psi_{\frak u',\frak u}\circ \phi)$ is a  local approximation to ${\frak u}'$. Using Implicit function theorem, it is easy to see that the equivalence class of $(\frak u,\frak p,\psi_{\frak u',\frak u}\circ \phi)$ depends only on the equivalence class of $(\frak u,\frak p,\phi)$. We thus obtain the map $\frak I$. This map in a small neighborhood of $\frak u$ is independent of the choice of $\Xi$. The map $\frak I$ is also injective. 

\begin{definition}\label{def1020}
	A neighborhood system of a quasi component ${\bf p} = [\frak u,\frak p,\phi]$ 
	of $\frak u$ in $\mathscr{QC}(k,\beta)$ is given 
	by mapping the neighborhood system of $\frak u$ in  $\mathcal M_{k+1}^{\rm RGW}(L;\beta)$ 
	via the map $\frak I$.
\end{definition}

The following is a straightforward consequence of the definition.
\begin{lemma}
 	$\mathscr{QC}(k,\beta)$ is Hausdorff and metrizable with respect to this topology.
	 For each quasi component ${\bf p}$  of $\frak u$, there exists a neighborhood of ${\bf p}$
	in $\mathscr{QC}(k,\beta)$ such that the restriction of $\Pi$ to this neighborhood 
	is a homeomorphism onto an open subset.
\end{lemma}

Let $\mathscr F$ be a subset of $\mathscr{QC}(k,\beta)$. For  $\frak u \in \mathcal M_{k+1}(L;\beta)$, we define
$$
\mathscr F(\frak u) = \Pi^{-1}(\frak u) \cap \mathscr F.
$$
It is a map which assigns to $\frak u$ a finite set of quasi components of $\frak u$. Justified by this, we call $\mathscr F$ a {\it quasi component choice map}.
\begin{definition}
	A quasi component choice map $\mathscr F$ is open (resp. closed) if it is an open 
	(resp. closed) subset of $\mathscr{QC}(k,\beta)$ (with respect to the topology of
	Definition \ref{def1020}). 
	We say $\mathscr F$ is proper if the restriction of $\Pi$ to $\mathscr F$ is proper.
\end{definition}

\par\smallskip
\noindent {\bf (OBIII):} For each $(k,\beta) \in \mathscr{TP}$, we take quasi component choice maps $\mathscr F_{k,\beta}$ and  $\mathscr F^{\circ}_{k,\beta}$.
\par\medskip
We are mainly concerned with the objects as in (OBIII) which satisfy the following condition.
\begin{conds}\label{conds1023}
	The quasi component choice map $\mathscr F^{\circ}_{k,\beta}$ is open and is 
	a subset of $\mathscr F_{k,\beta}$. The quasi component map $\mathscr F_{k,\beta}$ is proper.
\end{conds}

The next condition is related to the disk-component-wise-ness. Let ${\frak u} \in \mathcal M_{k+1}^{\rm RGW}(L;\beta)$ and $\check R$ be the very detailed tree associated to $\frak u$. Let $\check R$ belong to the disk splitting tree $\mathcal S$
. Let $\frak w$ be an interior vertex of ${\mathcal S}$. Following the discussion of the beginning of Subsection \ref{subsub:componentwise}, we obtain an element $\frak u_{\frak w} \in \mathcal M_{k_{\frak w}+1}^{\rm RGW}(L;\beta(\frak w))$ for each interior vertex $\frak w$ of $\mathcal S$. Define
\begin{equation}\label{form1012}
	\mathscr I_{\frak w} : \mathscr{QC}(k_{\frak w},\beta(\frak w))(\frak u_{\frak w}) \to 
	\mathscr{QC}(k,\beta)(\frak u)
\end{equation}
to be the map given by:
$$
  \mathscr I_{\frak w}([\frak u_{\frak w},\frak p,\phi]) = [\frak u,\frak p,\phi].
$$
The target of $\phi$ on the left hand side is $\Sigma_{\frak u_{\frak w}}$, the source curve of $\frak u_{\frak w}$, which is a subset of the source curve $\Sigma_{\frak u}$ of $\frak u$. So $\phi$ on the right hand side is regarded as a map to $\Sigma_{\frak u}$. It is clear that $\mathscr I_{\frak w}$ maps equivalent objects to equivalent ones and hence the above definition is well-defined.

\begin{lemma}
	The map $\mathscr I_{\frak w}$ is injective.
	If $\frak w \ne \frak w'$, then the image of $\mathscr I_{\frak w}$ is disjoint from the 
	image of $\mathscr I_{\frak w'}$.
\end{lemma}
\begin{proof}
	The injectivity is obvious from the definition.
	If $[\frak u,\frak p,\phi]$ is in the image of $\mathscr I_{\frak w}$,
	then the image of $\phi$ is contained in the component $\Sigma_{\frak u,\frak w}$
	corresponding to $\frak w$. Therefore, for $\frak w \ne \frak w'$, the image of the maps
	$\mathscr I_{\frak w}$ and $\mathscr I_{\frak w'}$ are disjoint.
\end{proof}
\begin{conds}\label{conds1025}
	Let ${\frak u} \in \mathcal M_{k+1}^{\rm RGW}(L;\beta)$ and 
	$\check R$ and $\mathcal S$ 
	 be given as above.
	Then we require
	\begin{equation}
		\aligned
		\mathscr F_{k,\beta}({\frak u}) 
		&= \bigcup_{\frak w} \mathscr I_{\frak w} \left( \mathscr F_{k_{\frak w},\beta_{\frak w}}
		({\frak u}_{\frak w})  \right),\\
		\mathscr F^{\circ}_{k,\beta}({\frak u}) 
		&= \bigcup_{\frak w} \mathscr I_{\frak w} 
		\left( \mathscr F^{\circ}_{k_{\frak w},\beta_{\frak w}}({\frak u}_{\frak w})  \right).
		\endaligned
	\end{equation}
\end{conds}
We need the following definition to state the next condition:
\begin{definition}
	Let ${\frak u} \in \mathcal M_{k+1}^{\rm RGW}(L;\beta)$ and $\frak y$ be an 
	inconsistent map with respect to a TSD $\Xi$ at $\frak u$. We assume $\Xi$ is sufficiently small
	such that the vector spaces 
	$E_{\frak u,{\bf p},\Xi}(\frak y)$ in \eqref{form6189} are well-defined. We define
	\begin{equation}\label{form19223rev}
		\aligned
		E_{\frak u,\mathscr F,\Xi}(\frak y):=
		&\sum_{[\bf p] \in \mathscr F_{k,\beta}({\frak u})}E_{\frak u,{\bf p},\Xi}(\frak y)
		\subset \bigoplus_{v \in C^{\rm int}_0(\check R)} L^2_{m}(\Sigma_{\frak y,v}
		;T \otimes \Lambda^{0,1}),
		\endaligned
	\end{equation}
	where $\Sigma$ denotes the sum of vector subspaces of a vector space.
	Similarly, we define
	\begin{equation}\label{form19223revrev}
		\aligned
		E_{\frak u,\mathscr F^{\circ},\Xi}(\frak y):=
		&\sum_{[\bf p] \in \mathscr F^{\circ}_{k,\beta}({\frak u})}E_{\frak u,{\bf p},\Xi}(\frak y)
		\subset 
		\bigoplus_{v \in C^{\rm int}_0(\check R)} L^2_{m}(\Sigma_{\frak y,v}
		;T \otimes \Lambda^{0,1}).
		\endaligned
	\end{equation}
	Note that $E_{\frak u,\mathscr F^{\circ},\Xi}(\frak y) \subseteq E_{\frak u,\mathscr F,\Xi}(\frak y)$.
\end{definition}

\begin{conds}\label{conds1027}
	We require that the sum in \eqref{form19223rev} is a direct sum for 
	$\frak y = \frak u$. Namely,
	\begin{equation}\label{form19223revrevrev}
		\aligned
		E_{\frak u,\mathscr F,\Xi}(\frak u)=
		&\bigoplus_{[\bf p] \in \mathscr F_{k,\beta}({\frak u})}E_{\frak u,{\bf p},\Xi}(\frak y)
		.
		\endaligned
	\end{equation}
\end{conds}
Note that the above condition implies that the sum in \eqref{form19223revrev} for $\frak y = \frak u$ is also a direct sum.
\begin{definition}\label{defn68899rev}
	We say the linearization of the non-linear Cauchy-Riemann equation is {\it transversal} to 
	$E_{\frak u,\mathscr F^\circ,\Xi}(\frak u)$
	if the sum of the images of the operator $D_{\frak u}\overline\partial$ 
	in \cite[(8.3)]{part2:kura} and 
	$E_{\frak u,\mathscr F^\circ,\Xi}(\frak u)$ is $L^2_{m,\delta}(\frak u;T \otimes \Lambda^{0,1})$.
\end{definition}

\begin{definition}\label{defn68899revref}
	Consider the operator:
	\[
	  \mathcal{EV}_{z_0} :  W^2_{m+1,\delta}(\frak u;T) \to T_{u(z_0)}L
	\]
	given  by evaluation at the point $z_0$, the 0-th boundary marked point of the source of $\frak u$.
	Recall that the Hilbert space $W^2_{m+1,\delta}(\frak u;T)$ is the domain of the operator 
	$D_{\frak u}\overline\partial$ in \cite[(8.3)]{part2:kura}.
	 We say $E_{\frak u,\mathscr F^\circ,\Xi}(\frak u)$ satisfies the {\it mapping transversality}
	property, if the restriction
	\[
	  \mathcal{EV}_{z_0}\vert_{D_{\frak u}\overline\partial^{-1}(E_{\frak u,\mathscr F^\circ,\Xi}(\frak u))} : 
	  D_{\frak u}\overline\partial^{-1}(E_{\frak u,\mathscr F^\circ,\Xi}(\frak u))\to T_{u(z_0)}L,
	\]
	is surjective.
\end{definition}
\begin{conds}\label{conds30}
	We require that the linearization of the non-linear Cauchy-Riemann equation is 
	transversal to $E_{\frak u,\mathscr F^\circ,\Xi}(\frak u)$
	and $E_{\frak u,\mathscr F^\circ,\Xi}(\frak u)$ satisfies the mapping transversality property.
\end{conds}

Let ${\rm Aut}(\frak u)$ be the group of automorphisms of $\frak u$. If $\gamma \in {\rm Aut}(\frak u)$ and $[\frak u,\frak p,\phi]$ is a quasi component of $\frak u$, then $[\frak u,\frak p,\gamma\circ\phi]$ is also a quasi component of $\frak u$.
Thus ${\rm Aut}(\frak u)$ acts on $\mathscr{QC}(k,\beta)(\frak u)$.
\begin{conds}\label{conds31}
	We require that $\mathscr F_{k,\beta}({\frak u})$, $\mathscr F^{\circ}_{k,\beta}({\frak u})$ 
	are invariant with respect to the action of ${\rm Aut}(\frak u)$.
\end{conds}

\subsubsection*{Disk-component-wise Obstruction Bundle Data: Part 4} Given the above objects satisfying the mentioned conditions, we can construct the desired obstruction bundle data:

\begin{lemma}
	Suppose we are given the objects in {\rm (OBI)}-{\rm (OBIII)}, which satisfy 
	Conditions \ref{conds1023}, \ref{conds1025}, \ref{conds1027}, \ref{conds30}, \ref{conds31}.
	Then $\{E_{\frak u,\mathscr F,\Xi}(\frak y)\}$ is a disk-component-wise system of 
	obstruction bundle data.
\end{lemma}
\begin{proof}
	The system $\{E_{\frak u,\mathscr F,\Xi}(\frak y)\}$ satisfies Definition \ref{defn684684} (1), (2), 
	(3) and (5) as immediate consequences of the construction.
	Definition \ref{defn684684} (3) is a consequence of the properness of $\mathscr F_{k,\beta}$.
	(Compare to \cite[Lemma 9.39]{part2:kura}.)
	Definition \ref{defn684684} (5) is a consequence of Condition \ref{conds30}.
	Definition \ref{defn684684} (6) is a consequence of Condition \ref{conds31}.
	Disc-component-wise-ness is an immediate consequence of 
	Condition \ref{conds1025}.
\end{proof}

\subsubsection*{Disk-component-wise Obstruction Bundle Data: Part 5}
To complete the proof of Proposition \ref{lem685}, it suffices to 
prove the next lemma.
\begin{lemma}\label{eximanuyobj}
There exist objects {\rm (OBI)}-{\rm (OBIII)} which satisfy 
Conditions \ref{conds1023}, \ref{conds1025}, \ref{conds1027}, \ref{conds30} and \ref{conds31}.
\end{lemma}

\begin{proof}
	The proof is by induction on $(k,\beta)$ and is given in several steps.
	Here we denote $(k,\beta) < (k',\beta')$ when either $\beta\cap [\omega] < \beta'\cap [\omega]$
	or $\beta\cap [\omega] = \beta'\cap [\omega]$, $k<k'$.
	
	{\bf Step 1} ({\it The base of induction}):
	We assume that $(k,\beta)$ is minimal in this 
	order $<$. In this case, the moduli space $\mathcal M_{k+1}^{\rm RGW}(L;\beta)$
	has no boundary. We follow a similar argument as in \cite[Subsection 9.2]{part2:kura}.
	For each $\frak p \in \mathcal M_{k+1}^{\rm RGW}(L;\beta)$, 
	we fix a vector space $E_{\frak p, v}$ for $v \in C^0_{\rm int}(\check R_{\frak p})$ 
	as in (OBII). We require that the linearization of the non-linear Cauchy-Riemann equation is 
	transversal to
	\begin{equation}\label{form1017}
		\bigoplus_{v \in C^0_{\rm int}(\check R_{\frak p})} E_{\frak p, v}
	\end{equation}
	at $\frak p$ (Definition \ref{defn68899rev}) 
	and \eqref{form1017} has the mapping transversality property at $\frak p$ 
	(Definition \ref{defn68899revref}). Using Lemma \ref{delta-k-b}, we may assume that the distance 
	of any point $x$ in the support of the elements of 
	$E_{\frak p, v}$ to nodal points of $\Sigma_{\frak p}$ is at least $\delta(k,\beta)$.
	Moreover, if $x'$ is another point in the source curve of $\frak p$ such that 
	$u_{\frak p}(x)=u_{\frak p}(x')$, then the distance between $x$ and $x'$ is greater than 
	$\delta(k',\beta')$.

	For each $\frak p$, we also pick a TSD $\Xi_{\frak p}$ 
	and a compact neighborhood $\mathcal K(\frak p)$ of $\frak p$ 
	in $\mathcal M_{k+1}^{\rm RGW}(L;\beta)$, which satisfy the following conditions. 
	Firstly we require that the compact neighborhood $\mathcal K(\frak p)$ is included in 
	the set of inconsistent maps determined by $\Xi_{\frak p}$. 
	Thus for any $\frak u \in \mathcal K(\frak p)$, there is a holomorphic embedding 
	$\phi_{\frak u,\frak p}$ from a neighborhood of $\bigcup_v  {\rm Supp}(E_{\frak p_{v}})$ to 
	$\Sigma_{\frak u}$, assuming that $\Xi_{\frak p}$ is small enough. We may also assume that 
	we can normalize $\phi_{\frak u,\frak p}$ as in Definition \ref{defn1015} to obtain:
	\[
	  \hat\phi_{\frak u,\frak p} :  \bigcup_v  {\rm Supp}(E_{\frak p_{v}}) \to \Sigma_{\frak u}.
	\]
	We use $\hat\phi_{\frak u,\frak p}$ to transport $E_{\frak p, v}$ to the point $\frak u$ and 
	obtain $E_{\frak p, v}(\frak u)$. We may choose $\Xi_{\frak p}$ small enough such that 
	the linearization of the non-linear Cauchy-Riemann equation is transversal to
	\begin{equation}\label{form1018}
		E_{\frak p, v}(\frak u) = \bigoplus_{v \in C^0_{\rm int}(\check R_{\frak p})} E_{\frak p,v}(\frak u)
	\end{equation}
	at $\frak u$ and \eqref{form1018} satisfies the mapping transversality property at $\frak u$, 
	for any $\frak u \in \mathcal K(\frak p)$.

	Now we take a finite set $\frak P(k+1,\beta)  \subset \mathcal M_{k+1}^{\rm RGW}(L;\beta)$ such that
	\begin{equation}\label{form1019}
		\bigcup_{\frak p \in \frak P(k+1,\beta)} {\rm Int}\, \mathcal K(\frak p) =
		\mathcal M_{k+1}^{\rm RGW}(L;\beta).
	\end{equation}
	We define
	\begin{equation}
		\aligned
		\mathscr F_{k,\beta}(\frak u) &= \{[\frak u,\frak p,\phi_{\frak u,\frak p}] \mid \frak p \in 
		\frak P(k+1,\beta),\,\,
		\frak u \in \mathcal K(\frak p)\},\\
		\mathscr F^{\circ}_{k,\beta}(\frak u) &= \{[\frak u,\frak p,\phi_{\frak u,\frak p}] \mid \frak p 
		\in \frak P(k+1,\beta),\,\,
		\frak u \in  {\rm Int}\, \mathcal (K(\frak p))\}.
		\endaligned
	\end{equation}
	Condition \ref{conds1023} is immediate from the definition. Condition \ref{conds1025} 
	is void in this case. We can perturb $E_{\frak p,v}$ arbitrarily small in $C^2$ topology 
	so that Condition \ref{conds1027} holds.  (See \cite[Lemma 9.9]{fooo:const2}.) 
	Condition \ref{conds30} follows from the 
	choice of  $E_{\frak p,v}$ and \eqref{form1019}. We can take $E_{\frak p,v}$ to be invariant
	under the action of ${\rm Aut}(\frak p)$ and hence Condition \ref{conds31} holds. 
	Thus we complete the first step of the induction.

	Next, we suppose that the required objects in {\rm (OBI)}-{\rm (OBIII)}
	are defined for $(k',\beta')$ with $(k',\beta') < (k,\beta)$.
	We use Condition \ref{conds1025} to define $\mathscr F'_{k,\beta}(\frak u)$, 
	$\mathscr F^{\prime \circ}_{k,\beta}(\frak u)$ for $\frak u \in \partial  \mathcal M_{k+1}(L;\beta)$.
	
	\vspace{3pt}
	{\bf Step 2:}{\it\, The set:
	\[
	  \bigcup_{\frak u \in \partial  \mathcal M_{k+1}(L;\beta)} \mathscr F^{\prime \circ}_{k,\beta}(\frak u)
	\]
	is an open subset of $\Pi^{-1}(\partial  \mathcal M_{k+1}(L;\beta))$, where $\Pi$ is the map in \eqref{Pi}.}
	\vspace{3pt}
	
	Let ${\bf p}_j \in \mathscr{QC}(k,\beta)$ with $\frak u_j = \Pi({\bf p}_j) \in  \partial  
	\mathcal M^{\rm RGW}_{k+1}(L;\beta)$.
	Suppose also $\lim_{j\to \infty} {\bf p}_j = {\bf p} \in  \mathscr F^{\prime \circ}_{k,\beta}(\frak u)$
	with $ \frak u=\lim_{j\to \infty} \frak u_j$. We need to show that
	${\bf p}_j \in \mathscr F^{\prime \circ}_{k,\beta}(\frak u_j)$ for sufficiently large values of $j$.
	Let the combinatorial type of $\frak u$ be given by a very detailed DD-ribbon tree $\check R$ which
	belongs to the disc splitting tree $\mathcal S$. We may assume that 
	the very detailed tree associated to $\frak u_j$ is independent of $j$ because there are finitely many
	very detailed tree obtained by level shrinking, level $0$ edge shrinking and fine edge shrinking.
	We denote this very detailed DD-ribbon tree by $\check R'$. We also assume that $\check R'$
	belongs to the disc splitting tree $\mathcal S'$. Since $\mathcal S'$ is obtained from $\mathcal S$
	by shrinking level $0$ edges, there is a standard shrinking map $\pi:\mathcal S\to \mathcal S'$.
	Note that $\mathcal S$ and $\mathcal S'$ have at least two interior vertices.
	
        By Condition \ref{conds1025}, there exists $\frak w \in C^0_{\rm int}(\mathcal S)$ 
        and ${\bf p}_{\frak w} \in \mathscr F^{\circ}_{k,\beta}(\frak u_{\frak w})$
        such that 
        $$
        {\bf p} = \mathscr I_{\frak w} ({\bf p}_{\frak w}).
        $$
        Let ${\frak w}' = \pi(\frak w)$,
        which is an interior vertex of $\mathcal S'$.
        We also define $ \mathcal S(\frak w'):=\pi^{-1}(\frak w')$, which is a subtree of $\mathcal S$.
        Let ${\frak u}_{ \mathcal S(\frak w')}$ be an object of $\mathcal M^{\rm RGW}_{k_{\frak w'}+1}(L;\beta_{\frak w'})$
        which is obtained from $\frak u$ and $\mathcal S(\frak w')$ 
        in the same way as in the beginning of Subsection \ref{subsub:componentwise}.
        Convergence of $\frak u_j$ to $\frak u$ implies
        $$
        \lim_{j\to \infty} \frak u_{j,{\frak w}'} = {\frak u}_{ \mathcal S(\frak w')}
        $$
        by the definition of the RGW topology.

	Since ${\frak w}$ is a vertex of $ \frak R(\frak w')$, there exists:
        $$
        \mathscr I'_{\frak w} : \mathscr{QC}(k_{\frak w},\beta_{\frak w})(\frak u_{\frak w}) \to 
        \mathscr{QC}(k_{\frak w'},\beta_{\frak w'})( {\frak u}_{ \mathcal S(\frak w')}).
        $$
        We define
        $
        {\bf p}_{\frak w'} = \mathscr I'_{\frak w}({\bf p}_{\frak w}).
        $
        Using the definition of the topology of  $\mathscr{QC}(k,\beta)$
        and of $\mathscr I_{\frak w}$,
        it is easy to see that 
        there exists 
        ${\bf p}_{j,\frak w'} \in \mathscr{QC}(k_{\frak w'},\beta_{\frak w'})(\frak u_{j,{\frak w}'})$
        such that 
        $$
        \lim_{j\to \infty}{\bf p}_{j,\frak w'} = {\bf p}_{\frak w'}
        $$
        in $\mathscr{QC}(k_{\frak w'},\beta_{\frak w'})$
        and 
        $$
        \mathscr I'_{\frak w} ({\bf p}_{j,\frak w'}) = {\bf p}_{j}.
        $$
        Now by induction hypothesis
        $$
        {\bf p}_{j,\frak w'} \in \mathscr F^{\circ}_{k_{\frak w'},\beta_{\frak w'}}(\frak u_{j,{\frak w'}})
        $$
        for sufficiently large $j$.
        Condition \ref{conds1025} implies 
        ${\bf p}_j \in \mathscr F^{\prime \circ}_{k,\beta}(\frak u_j)$ for large $j$, as required.
	
	\vspace{3pt}
	{\bf Step 3:}{\,\it The restriction of $\Pi$ to
        $$
        \bigcup_{\frak u \in \partial  \mathcal M_{k+1}(L;\beta)} \mathscr F^{\prime}_{k,\beta}(\frak u)
        $$ 
        is a proper map to $\partial  \mathcal M_{k+1}^{\rm RGW}(L;\beta)$.}
	\vspace{3pt}
	
        Let ${\bf p}_j \in \mathscr F'_{k,\beta}(\frak u_j)$ with $\frak u_j  
        \in  \partial  \mathcal M^{\rm RGW}_{k+1}(L;\beta)$.
        Suppose $\lim_{j\to \infty} \frak u_j = \frak u \in  \partial  \mathcal M^{\rm RGW}_{k+1}(L;\beta)$.
        It suffices to find a subsequence of ${\bf p}_j$ which 
        converges to an element of  $\mathscr F'_{k,\beta}(\frak u)$.
	Let the combinatorial type of $\frak u$ be given by a very detailed DD-ribbon tree $\check R$ which
	belongs to the disc splitting tree $\mathcal S$. After passing to a subsequence, we may assume that 
	the very detailed tree associated to $\frak u_j$ is independent of $j$. We denote this tree by $\check R'$
	which belongs to the disc splitting tree $\mathcal S'$. Let $\pi:\mathcal S\to \mathcal S'$ be defined 
	as in the previous step.

        By Condition \ref{conds1025} and after passing to a subsequence, we may assume that there exist
	$\frak w \in C^{0}_{\rm int}(\mathcal S')$ and ${\bf p}_{j,\frak w} \in\mathscr F'_{k_{\frak w},
	\beta_{\frak w}}(\frak u_{j,\frak w})$ such that 
        $$
        \mathscr I_{\frak w}({\bf p}_{j,\frak w}) = {\bf p}_j.
        $$
	Let $\mathcal S_{\frak w}$ be the subtree $\pi^{-1}(\frak w)$ of $\mathcal S$.
        We obtain $\frak u_{\frak R_{\frak w}}$ from $\frak u$ in the same way as in 
        the beginning of Subsection \ref{subsub:componentwise}. Convergence of $\frak u_j$ to $\frak u$ implies
        that
	\[
          \lim_{j \to \infty} \frak u_{j,\frak w} = \frak u_{\frak R_{\frak w}}
	\]
        by the definition of the RGW topology. Now we use the induction hypothesis to find a subsequence 
        such that 
        ${\bf p}_{j,\frak w} \in 
        \mathscr F'_{k_{\frak w},\beta_{\frak w}}(\frak u_{j,\frak w})$
        converges to 
        $
        {\bf p}_{\frak w} \in \mathscr F'_{k_{\frak w},\beta_{\frak w}}(\frak u_{\frak R_{\frak w}}).
        $
        Therefore
        $$
        \lim_{j\to \infty} {\bf p}_{j} = 
        \mathscr I_{\frak w}({\bf p}_{\frak w}) \in \mathscr F'_{k,\beta}(\frak u).
        $$
	This completes the proof of this step.

	\vspace{3pt}
	{\bf Step 4:\,}({\it Extension to a neighborhood of the boundary})
        In the previous steps, we defined $\mathscr F^{\prime}_{k,\beta}$ and $\mathscr F^{\prime\circ}_{k,\beta}$ 
        on the boundary. Next, we extend these quasi component choice maps to a neighborhood of the boundary.
	We fix $\rho > 0$ sufficiently small such that 
        if $d(\frak u,\frak u') < 5\rho$,
        $\frak u' \in \partial \mathcal M_{k+1}^{\rm RGW}(L;\beta)$
        and $[\frak u',\frak p,\phi] \in \mathscr F'_{k,\beta}(\frak u')$,
        then $[\frak u,\frak p,\psi_{\frak u,\frak u'}\circ \phi]$ is a quasi component.
        Then for $\frak u\in \mathcal M_{k+1}^{\rm RGW}(L;\beta)$ with $d(\frak u,\partial \mathcal M_{k+1}^{\rm RGW}(L;\beta)) < 2\rho$,
        we define
        \begin{enumerate}
        \item $\mathscr F^{\prime}_{k,\beta}(\frak u)$ 
	        is the set of $[\frak u,\frak p,\phi]$ such that there are
	        $\frak u' \in \partial \mathcal M^{\rm RGW}_{k+1}(L;\beta)$
	        and $[\frak u',\frak p,\phi'] \in \mathscr F^{\prime}_{k,\beta}(\frak u')$ with the following properties:
		\begin{enumerate}
			\item $d(\frak u,\frak u') \le 2d(\frak u,\partial \mathcal M^{\rm RGW}_{k+1}(L;\beta))\le\rho$.
			\item $(\frak u,\frak p,\phi)$ is equivalent to $(\frak u,\frak p,\psi_{\frak u,\frak u'}\circ \phi')$.
		\end{enumerate}
        \item
        $\mathscr F^{\prime\circ}_{k,\beta}(\frak u)$ is the set of $[\frak u,\frak p,\phi]$
        such that there are $\frak u' \in \partial \mathcal M^{\rm RGW}_{k+1}(L;\beta)$
        and $[\frak u',\frak p,\phi'] \in \mathscr F^{\prime\circ}_{k,\beta}(\frak u')$ with the following properties:
        \begin{enumerate}
   	         \item $\frak u = \frak u'$ or
	         $d(\frak u,\frak u') < 2d(\frak u,\partial \mathcal M^{\rm RGW}_{k+1}(L;\beta)) <  \rho$.
	     \item $(\frak u,\frak p,\phi)$ is equivalent to $(\frak u,\frak p,\psi_{\frak u,\frak u'}\circ \phi')$.
        \end{enumerate}
        \end{enumerate}

        We put 
        $$
        \mathscr F^{\prime\circ}_{k,\beta} 
        = \bigcup_{\frak u} \mathscr F^{\prime\circ}_{k,\beta}(\frak u),
        \qquad
        \mathscr F^{\prime}_{k,\beta} 
        = \bigcup_{\frak u} \mathscr F^{\prime}_{k,\beta}(\frak u).
        $$
        It follows easily from Step 2 that $\mathscr F^{\prime\circ}_{k,\beta}$ is open.
        It follows easily from Step 3 that the restriction of $\Pi$ to $\mathscr F^{\prime\circ}_{k,\beta}$ is proper.

        Items (b) in the above definition implies that
        $\mathscr F^{\prime}_{k,\beta}(\frak u)$ and $\mathscr F^{\prime\circ}_{k,\beta}(\frak u)$ coincide with 
        the previously defined spaces for $\frak u \in \partial \mathcal M^{\rm RGW}_{k+1}(L;\beta)$.
        Therefore, Condition \ref{conds1025} holds.
        Thus we have constructed objects for $(k,\beta)$ which satisfy 
        all the required conditions except Condition \ref{conds30}.
        By taking a smaller value of $\rho$ if necessary,  we can guarantee that Condition \ref{conds30} is also satisfied.

	\vspace{3pt}
	{\bf Step 5:\,}({\it Extension to the rest of the moduli space $\mathcal M_{k+1}^{\rm RGW}(L;\beta)$})
        The rest of the proof is similar to Step 1. 
        For each $\frak p \in {\rm Int}(\mathcal M^{\rm RGW}_{k+1}(L;\beta))$
        we choose $\Xi_{\frak p}$ ,  $E_{\frak p,v}$ and $\mathcal K(\frak p)$
        as in the first step of the induction.
        We take a finite set $\frak P(k+1,\beta)$ such that
        \begin{equation}\label{form1019rev}
        \bigcup_{\frak p \in \frak P(k+1,\beta)} {\rm Int}\, \mathcal K(\frak p) 
        = \mathcal M^{\rm RGW}_{k+1}(L;\beta)
        \setminus B_{\rho}(\partial \mathcal M^{\rm RGW}_{k+1}(L;\beta)).
        \end{equation}
        is satisfied instead of (\ref{form1019}).
        Now we define
        \begin{equation}
        \aligned
        \mathscr F_{k,\beta}(\frak u) &= 
        \mathscr F'_{k,\beta}(\frak u)  \cup \{[\frak u,\frak p,\phi_{\frak u,\frak p}] \mid \frak p \in \frak P(k,\beta),\,\,
        \frak u \in \mathcal K(\frak p)\}\\
        \mathscr F^{\circ}_{k,\beta}(\frak u) &= 
        \mathscr F^{\prime\circ}_{k,\beta}(\frak u)  \cup  
        \{[\frak u,\frak p,\phi_{\frak u,\frak p}]\mid \frak p \in \frak P(k,\beta),\,\,
        \frak u \in  {\rm Int}\, \mathcal K(\frak p)\}.
        \endaligned
        \end{equation}
        They satisfy all the required conditions including Condition \ref{conds30}.
        We thus have completed the inductive step.
\end{proof}

We verified Proposition \ref{lem685} and hence Theorem \ref{lema362rev}. This completes the construction of a system of Kuranishi structures on $\mathcal M_{K+1}^{\rm RGW}(L;\beta)$ which are compatible at the boundary components and corners. For the proof of Theorem \ref{mainthm-part3}, we also need to construct a system of Kuranishi structures on the moduli space of strips, that are compatible at the boundary components and corners, that is, Propositions \ref{lema362}, \ref{lem364}, \ref{lem365}. The proof in the case of strips is similar to the case of disks and we omit it here.


\section{Floer homology of Lagrangians with Minimal Maslov Number $2$}
\label{sec:forget}

In this section, we relax Condition \ref{cond420} (1)  and consider Lagrangians with minimal Maslov number $2$.

\subsection{Disk potential and the square of the boundary operator}
\label{maslovtwo}

In this subsection, we closely follow the arguments in \cite[Section 16]{ohbook}, except we use 
virtual fundamental classes.
Let $L\subset X \setminus \mathcal D$ be a connected monotone spin Lagrangian submanifold and $p \in L$ be a point.
Suppose $\alpha \in \Pi_2(X,L)$. We consider the fiber product
$$
\mathcal M^{\rm RGW}(L;\alpha;p) = \{p\} \,\,{}_{L}\times_{{\rm ev}_0} \mathcal M^{\rm RGW}_1(L;\alpha)
$$
The virtual dimension of this space is
$$
n + \mu(\alpha) - 2 - n = \mu(\alpha) - 2.
$$
Therefore, in the case that $\mu(\alpha) = 2$, $\mathcal M^{\rm RGW}(L;\alpha;p) $ has virtual dimension $0$. We take a multi-valued perturbation $\{\widehat{\frak s}^n\}$, which is transversal to $0$, and define
\begin{equation}\label{eq:defPO}
\frak{PO}^L_{\alpha,p} = [\mathcal M^{\rm RGW}(L;\alpha;p),\widehat{\frak s}^n]
\in \bbQ.
\end{equation}
(There is a related construction in \cite{T2}.)
\begin{lemma}\label{lem49}	$\frak{PO}^L_{\alpha,p}$ is independent of the choice of $p$, $\widehat{\frak s}^n$ and the other auxiliary choices in the definition of the Kuranishi structure if $n$ is sufficiently large.
	Moreover, if $F : X \to X$ is a symplectic diffeomorphism, which is the identity map in a neighborhood of $\mathcal D$, then 
	$$
	  \frak{PO}^{F(L)}_{F_*(\alpha),F(p)} = \frak{PO}^{L}_{\alpha,p}.
	$$
\end{lemma}
\begin{proof}
        Let $p,q \in L$ and $\gamma : [0,1] \to X$ be a path joining $p$ to $q$.
        Let $J_0,J_1$ be two almost complex structures and 
        $\mathcal J = \{J_t \mid t \in [0,1] \}$ be a 
        family of almost complex structures\footnote{We always use 
        almost complex structures which are constructed as in \cite[Subsection 3.2]{part1:top}. 
        As we mentioned in \cite[Remark 3.9]{part1:top}, the space of such complex structures has trivial homotopy groups.}
         joining $J_0$ to $J_1$ parametrized by $[0,1]$.
        If we use the almost complex structure $J_t$ to define the moduli space 
        $\mathcal M^{\rm RGW}_1(L;\alpha)$, then the 
        resulting moduli space is denoted by $\mathcal M^{\rm RGW}_1(L;\alpha;J_t)$. We consider
	\begin{equation} \label{point-moduli-space-family}
        \bigcup_{t \in [0,1]} 
        \left(\{\gamma(t)\} \,\,{}_{L}\times_{{\rm ev}_0} \mathcal M^{\rm RGW}_1(L;\alpha;J_t)
        \right) \times \{t\},
        \end{equation}
        which we denote by $\mathcal M^{\rm RGW}(L;\alpha;\mathcal J;\gamma)$.
	As a straightforward generalization of Proposition \ref{prop356}, 
        one can show that this space has a 1-dimensional Kuranishi structure with boundary.
        In fact, we can fix a Kuranishi structure so that its restriction to $t =0,1$ 
        agrees with the ones that are used to define $\frak{PO}^L_{\alpha,p}$ and $\frak{PO}^L_{\alpha,q}$,
        respectively.
        We next define a  multi-valued perturbation $\{\widehat{\frak t}^n\}$ transversal to $0$.
        We  require that it agrees with the ones that are used to 
        define $\frak{PO}^L_{\beta,p}$ and $\frak{PO}^L_{\beta,q}$
        at $t=0,1$.

      We claim that the normalized boundary of $\mathcal M^{\rm RGW}(L;\alpha;\mathcal J;\gamma)$ is given by the 
      union of $\mathcal M^{\rm RGW}(L;\alpha;J_t;\gamma(t))$ for $t=0$ and $1$. In fact, we can follow the proof of  Theorem \ref{mainthm-part3} given in Sections \ref{sec:minimas3} 
      and \ref{sub:systemconst} of this paper to show that the other possibility for boundary elements is given by the elements of the fiber product:
      \begin{equation} \label{bdry-element-2}
        \mathcal M^{\rm RGW}_{2}(L;\alpha_1;\mathcal J;\gamma){}_{{\rm ev}_1}\hat\times_{{\rm ev}_0} 
	\mathcal M^{\rm RGW}_{1}(L;\alpha_2;\mathcal J)
      \end{equation}
      Here $\alpha=\alpha_1\#\alpha_2$. The first factor in \eqref{bdry-element-2} is defined by replacing 
      the moduli space $\mathcal M^{\rm RGW}_1(L;\alpha;J_t)$ in \eqref{point-moduli-space-family} with 
      $\mathcal M^{\rm RGW}_2(L;\alpha_1;J_t)$. In this definition, we use the evaluation at the $0$-th boundary marked point of 
      the elements of $\mathcal M^{\rm RGW}_2(L;\alpha_1;J_t)$. Evaluation at the first boundary marked point determines the map
      ${\rm ev}_1:\mathcal M^{\rm RGW}_{2}(L;\alpha_1;\mathcal J;\gamma) \to L$.
      The second factor in \eqref{bdry-element-2} is the union of the moduli spaces $\mathcal M^{\rm RGW}_1(L;\alpha;J_t)$
      for $t\in [0,1]$.
      If the moduli space \eqref{bdry-element-2} is non-empty, then $\omega(\alpha_1)$ and $\omega(\alpha_2)$ are both 
      positive. This implies that $\mu(\alpha_i)\geq 2$ because $L$ is a orientable monotone Lagrangian. However, this 
      contradicts the assumption that $\mu(\alpha)=\mu(\alpha_1) + \mu(\alpha_2)$ is equal to $2$.
      
	The first half of Lemma \ref{lem49} follows from the description of the normalized boundary of  
	$\mathcal M^{\rm RGW}(L;\alpha;\mathcal J;\gamma)$ and Corollary \ref{Cor27} (3).
	In order to verify the second part, let $J$ be an almost complex structure on $X$, and push it forward using $F$ 
	to obtain the complex structure $F_*J$. Clearly, we have the following isomorphism of spaces with Kuranishi structures:
	\[
	  \mathcal M^{\rm RGW}_1(L;\alpha;J) \cong \mathcal M^{\rm RGW}_1(F(L);F_*(\alpha);F_*(J)).
	\]
	The second half of the lemma follows from this isomorphism.
\end{proof}

\begin{definition}\label{defn41414}
	Let  $\rho : H_1(L) \to \bbQ_*$ be a group homomorphism. 
	We define a {\it potential function} $\frak{PO}$ as follows:
	\[
	  \frak{PO}_L(\rho)  = \sum_{\beta} \rho(\partial \beta) \frak{PO}^L_{\beta,p} \in \bbQ.
	\]
	Here the sum is taken over all $\beta \in \Pi_2(X,L)$ with $\mu(\beta) = 2$, and $\rho(\partial \beta)$ denotes the image of the boundary of $\beta$ in $H_1(L)$ with respect to $\rho$.
Lemma \ref{lem49} implies that this is an invariant of $L$.
\end{definition}
We now generalize Theorem \ref{lem48} as follows.
\begin{theorem}\label{lem413}
	Suppose Condition \ref{cond420} {\rm (2)} holds. 
	Then the boundary operator 
	\[
	  \partial : CF(L_1,L_0;\bbQ) \to CF(L_1,L_0;\bbQ)
	\]
	defined by \eqref{form43} satisfies the following identity:
	\begin{equation}
		\partial \circ \partial = \frak{PO}_{L_1}(1) - \frak{PO}_{L_0}(1).
	\end{equation}
\end{theorem}
The proof of this theorem in the case that $\mathcal D = \emptyset$, can be found in \cite[Chapter 16]{ohbook}.
\begin{proof}
	We follow a similar proof as in Theorem \ref{lem48}. We also use the same notations as before.
	\par
	First, in this subsection, we prove the theorem assuming the existence of Kuranishi structures 
	and multi-valued perturbations compatible with forgetful maps (which are 
	defined in Subsection \ref{subsec:forget}), whose existence are proved as Theorem \ref{comp-forg} and Proposition \ref{prop12866}.
	
	The same argument as in Section \ref{sec:minimas3} shows that \eqref{form460} is responsible for part of the boundary terms in the left hand side of 
	\eqref{form44}. This term gives rise to the 
	coefficient of $[q]$ in $\partial\circ\partial([p])$.
	Next, we study the contribution of the terms of the form
	\begin{equation}\label{form47-b}
		\mathcal M^{\rm RGW}_{0,1}(L_1,L_0;p,q;\beta')\, \hat\times_{L_0}\, \mathcal M^{\rm RGW}_1(L_0;\alpha),
	\end{equation}
	where $\mu(\beta')+\mu(\alpha)=2$. The moduli space $\mathcal M^{\rm RGW}_1(L_0;\alpha)$ is non-empty only if 
	$\omega(\alpha)>0$. Monotonicity of $\alpha$ implies that $\mu(\alpha)\geq 2$. In particular, $\mu(\beta') \le 0$.
	 We have two different possibilities 
	to consider:
	\par\medskip
	\noindent{\bf Case 1}: ($p\ne q$ or $\beta' \ne 0$)
	Consider the forgetful map
	$$
	\frak{fg} : \mathcal M^{\rm RGW}_{0,1}(L_1,L_0;p,q;\beta')\to \mathcal M^{\rm RGW}_{0,0}(L_1,L_0;p,q;\beta').
	$$
	We have
	$$
	\dim \mathcal M^{\rm RGW}_{0,0}(L_1,L_0;p,q;\beta') = \mu(\beta') - 1 < 0.
	$$
	This implies that $\widehat{\frak s}^n$ does not vanish on this space by Theorem \ref{prop61111} (3).
	Therefore, by Theorem \ref{comp-forg} and Proposition \ref{prop12866},
	the multi-valued perturbation $\widehat{\frak s}^n$ neither vanishes on 
	$\mathcal M^{\rm RGW}_{0,1}(L_1,L_0;p,q;\beta')$.
	Consequently, the contribution of \eqref{form47-b} is trivial in this case.
	\par\smallskip
	\noindent{\bf Case 2}: ($p = q$ and $\beta'  = 0$)
	In this case, (\ref{form47-b}) has the form
	$$
	\{p\} \times_{L_0} \mathcal M^{\rm RGW}_1(L_0;\alpha)
	$$
	and $0 \# \alpha = \beta$. Therefore, $\mu(\alpha) = 2$.
	Thus the contribution of the terms as in \ref{form47-b} is equal to
	$$
	\sum_{\alpha}[ \mathcal M^{\rm RGW}_1(L_0;\alpha),\widehat{\frak s}^n] =  \frak{PO}_{L_0}(1).
	$$
Similarly, we can show that the  contribution of the remaining part of the boundary, given 
in Definition \ref{defn239},
is $- \frak{PO}_{L_0}(1)$. (See \cite[Theorem 8.8.10 (2)]{fooobook2} for the sign.)
\end{proof}

\begin{definition}
If  Condition \ref{cond420} (2) is satisfied and 
$\frak{PO}_{L_0}(1) = \frak{PO}_{L_1}(1)$, we define Floer homology by
$$
HF(L_1,L_0;o;X \setminus \mathcal D) \cong \frac{{\rm Ker} \left( \partial : CF(L_1,L_0;\bbQ;o) 
\to CF(L_1,L_0;\bbQ;o)\right)}{{\rm Im} \left( \partial : CF(L_1,L_0;\bbQ;o) 
\to CF(L_1,L_0;\bbQ;o)\right)}.
$$
\end{definition}

\subsection{Forgetful Map of the Boundary Marked Points}
\label{subsec:forget}

In the last subsection, we used compatibility of our Kuranishi structures with the forgetful map of the boundary marked points to show that the boundary elements given by Definition \ref{defn239} do not affect various constructions in the case that our Lagrangians are monotone with minimal Maslov numbers $2$. (See the proof of Case 1 of Theorem \ref{lem413}.) We will describe the compatibility in this subsection. 

We start with the definition of the case that the source curve has no disk bubble. We consider the moduli spaces of pseudo-holomorphic disks $\mathcal M_{k+1}^{\rm reg, d}(\beta;{\bf m})$ and of pseudo-holomorphic strips  $\mathcal M_{k_1,k_0}^{\rm reg}(L_1,L_0;p,q;\beta;{\bf m})$. In the following, we assume that $\beta$ is a non-trivial homology class.

\begin{definition}\label{defn366}
Let $1 \le j \le k$. We define 
\begin{equation}\label{map732}
	\frak{fg}^{\partial}_j : \mathcal M_{k+1}^{\rm reg, d}(\beta;{\bf m})
	\to \mathcal M_{k}^{\rm reg, d}(\beta;{\bf m})
\end{equation}
as follows. Let $[(\Sigma,\vec z,\vec w),u] \in  \mathcal M_{k+1}^{\rm reg, d}(\beta;{\bf m})$. We put $\vec z = (z_0,\dots,z_k)$, $z_i \in \partial \Sigma$. We set $\vec z^{\,\prime} = (z_0,\dots,z_{j-1},z_{j+1},\dots,z_k)$. Then
$$
\frak{fg}^{\partial}_j ([(\Sigma,\vec z,\vec w),u])
= [(\Sigma,\vec z^{\,\prime},\vec w),u].
$$
Let $1 \le j \le k_0$ (resp. $1 \le j \le k_1$). We define
\begin{equation}\label{map733}
\aligned
\frak{fg}^{\partial}_{0,j} : \mathcal M_{k_1,k_0}^{\rm reg}(L_1,L_0;p,q;\beta;{\bf m})
\to \mathcal M_{k_1,k_0-1}^{\rm reg}(L_1,L_0;p,q;\beta;{\bf m})
\endaligned
\end{equation}
(resp.
\begin{equation}\label{map734}
\aligned
\frak{fg}^{\partial}_{1,j} : \mathcal M_{k_1,k_0}^{\rm reg}(L_1,L_0;p,q;\beta;{\bf m}) 
\to \mathcal M_{k_1-1,k_0}^{\rm reg}(L_1,L_0;p,q;\beta;{\bf m}))
\endaligned
\end{equation}
in a similar way by forgetting $z_{0,j}$ (resp. $z_{1,j}$).
\end{definition}

\begin{remark}\label{rem367}
	We do not consider the forgetful map for the $0$-th marked point
	of the elements of $\mathcal M_{k+1}^{\rm reg, d}(\beta;{\bf m})$ because we do not need it in the present paper.
\end{remark}

Note that:
$$
\mathcal M_{k+1}^{\rm reg, d}(\beta;\emptyset) \subset \mathcal M_{k+1}^{\rm RGW}(\beta),
$$
and 
$$
\mathcal M_{k_1,k_0}^{\rm reg}(L_1,L_0;p,q;\beta;\emptyset) \subset \mathcal M^{\rm RGW}_{k_1,k_0}(L_1,L_0;p,q;\beta).
$$

\begin{lemma}\label{forgetful-bdry-RGW}
	The map \eqref{map732} in the case ${\bf m} = \emptyset$ can be extended to a continuous map:
	\begin{equation}\label{map732exted}
		\frak{fg}^{\partial}_j : \mathcal M_{k+1}^{\rm RGW}(\beta)\to \mathcal M_{k}^{\rm RGW}(\beta).
	\end{equation}
	Moreover, \eqref{map733} and \eqref{map734} for ${\bf m} = \emptyset$ can be extended to 
	\eqref{map733exted} and \eqref{map734exted} below, respectively:
	\begin{equation}\label{map733exted}
		\frak{fg}^{\partial}_{0,j} : \mathcal M_{k_1,k_0}^{\rm reg}(L_1,L_0;p,q;\beta)\to 
		\mathcal M_{k_1,k_0-1}^{\rm reg}	(L_1,L_0;p,q;\beta),
	\end{equation}
	\begin{equation}\label{map734exted}
		\frak{fg}^{\partial}_{1,j} : \mathcal M_{k_1,k_0}^{\rm reg}(L_1,L_0;p,q;\beta)\to 
		\mathcal M_{k_1-1,k_0}^{\rm reg}(L_1,L_0;p,q;\beta).
	\end{equation}
\end{lemma}
\begin{proof}
	We will construct (\ref{map732exted}) in detail. 
	The construction of (\ref{map733exted}) and (\ref{map734exted}) is similar.
	Let $\mathcal R$ be a DD-ribbon tree of type $(\beta;k)$.
	Let $\frak v^j$ be the $j$-th exterior vertex corresponding to the boundary marked point that we will forget.
	There is a unique interior vertex $\frak v$ which is connected to $\frak v^j$. The color of the vertex $\frak v$ is necessarily  d.
	Let $\mathcal S(\frak v)$ be the DD-tree associated to $\frak v$.
	Suppose $\mathcal S(\frak v)$ is of type $(\beta(\frak v);k_{\frak v})$ and its number of levels is $\vert \lambda_{\frak v} \vert$.
	We consider the next three cases separately.
	\par\smallskip
	\noindent {\bf Case 1}: ($\beta(\frak v) \ne 0$.) We consider the root vertex $v$ of $\mathcal S(\frak v)$, which has color d. 
	If $\beta(v) = 0$, then the $v$-part of any element of $\mathcal M^0(\mathcal R)$ is a stable map 
	$((\Sigma_v,\vec z_v,\vec w_v),u_v)$ with $u_v$ being a constant map. 
	Hence $u_v$ does not intersect $\mathcal D$.
	Therefore, $\mathcal S(\frak v)$ does not have an inside vertex.
	It implies that $\beta(\frak v) = \beta(v) = 0$, which is a contradiction. Thus $\beta(v) \ne 0$ and 
	removing the $j$-th external vertex and the edge incident to it gives rise to a DD-tree $\mathcal R'$ 
	of type $(\beta;k-1)$.
	Moreover, forgetting the $j$-th boundary marked point determines
	a map $\mathcal M^0(\mathcal R) \to \mathcal M^0(\mathcal R')$ and we define the extension of 
	$\frak{fg}^{\partial}_j$ to $\mathcal M^0(\mathcal R)$ to be given by this map.
%
	\par\smallskip
	\noindent {\bf Case 2}: ($\beta(\frak v) = 0$, $k_{\frak v} \ge 3$.)
	Using the argument of the previous case, we can show that $\mathcal S(\frak v)$ has no inside vertex. 
	So it has only one interior vertex $v$. 
	Since $v$ is incident to at least 4 edges, after removing $\frak v^j$ and 
	the edge incident to it, $\mathcal S(\frak v)$ is still stable. 
	The rest of the construction is similar to Case 1.
	\par\smallskip
	\noindent {\bf Case 3}: ($\beta(\frak v) = 0$, $k_{\frak v} = 2$.)
	As in Case 2, we can conclude that $\mathcal S(\frak v)$ has only one interior vertex denoted by $v$ and  
	the stable map associated to $v$ is a constant map.
	In addition to $\frak v^j$, there are two other vertices $\frak v'$ and $\frak v''$ of $\mathcal R$ which are connected 
	to $\frak v$.
	We remove $\frak v^j$, $\frak v$ and the edges containing them. Then we connect 
	$\frak v'$ and $\frak v''$ with a new edge to obtain a new DD-ribbon tree $\mathcal R'$.
	(See Figure \ref{FIgurelem368Case3}.)
	By forgetting the factor corresponding to $\frak v$, we also obtain a map 
	$\mathcal M^0(\mathcal R) \to \mathcal M^0(\mathcal R').$
	This gives the restriction of the map in \eqref{map732exted} to $\mathcal M^0(\mathcal R)$.
	The continuity of this map shall be obvious from the definition 
	of the RGW topology in \cite[Section 4]{part1:top}.
	\begin{figure}[h]
        \centering
        \includegraphics[scale=0.4]{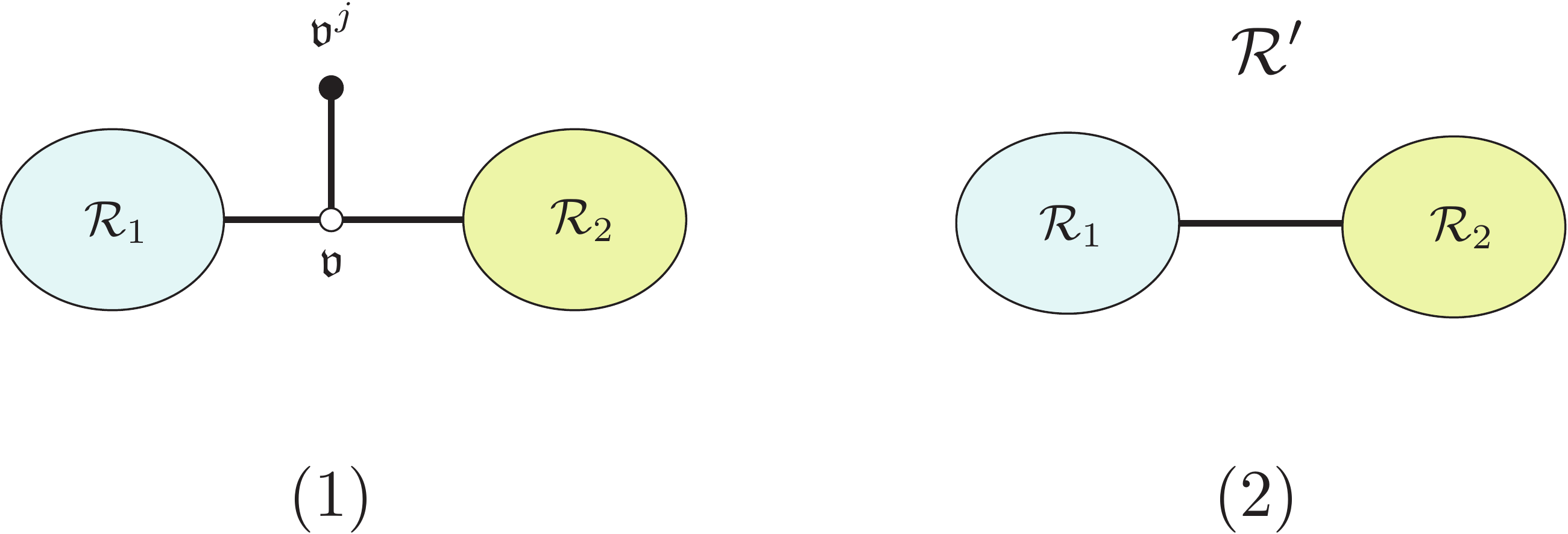}
        \caption{$\mathcal R'$ in Case 3}
        \label{FIgurelem368Case3}
        \end{figure}
\end{proof}

\subsection{Compatibility of Kuranishi Structures with Forgetful Maps}
\label{subsub:compforget}

We next discuss the relation between forgetful maps 
and Kuranishi structures.

\begin{definition}\label{lemma369}
        Let $\frak p \in \mathcal M_{k+1}^{\rm RGW}(\beta)$
        and $\overline{\frak p} = \frak{fg}^{\partial}_j(\frak p) \in \mathcal M_{k}^{\rm RGW}(\beta)$.
        Let $\mathcal U_{\frak p} = (U_{\frak p},E_{\frak p},s_{\frak p},\psi_{\frak p})$ 
        and $\mathcal U_{\overline{\frak p}} = 
        (U_{\overline{\frak p}},E_{\overline{\frak p}},s_{\overline{\frak p}},\psi_{\overline{\frak p}})$ be their Kuranishi charts.
        We say those Kuranishi charts are compatible with the forgetful map if the following holds.
        
        By shrinking the charts, we may assume $U_{\frak p} = V_{\frak p}/\Gamma_{\frak p}$ 
        and $U_{\overline{\frak p}} = V_{\overline{\frak p}}/\Gamma_{\overline{\frak p}}$,
        where $V_{\frak p}$, $V_{\overline{\frak p}}$ are manifolds and $\Gamma_{\frak p}$ 
        and $\Gamma_{\overline{\frak p}}$ are finite groups acting on them.
        \begin{enumerate}
                \item
                There exists a group homomorphism
                $\phi_{\frak p} : \Gamma_{\frak p} \to \Gamma_{\overline{\frak p}}$.
                \item
                There exists a $\phi_{\frak p}$ equivariant map
                $$
                F_{\frak p} : V_{\frak p} \to V_{\overline{\frak p}}
                $$
                that is a strata-wise smooth submersion.
                \item
                $E_{\frak p}$ is isomorphic to the pullback of $E_{\overline{\frak p}}$
                by $F_{\frak p}$. In other words,
                there exists fiberwise isomorphic lift 
                $$
                \tilde F_{\frak p} : E_{\frak p} \to E_{\overline{\frak p}}
                $$
                of $F_{\frak p}$, which is $\phi_{\frak p}$ equivariant.
                \item
                $\tilde F_{\frak p} \circ s_{\frak p} = s_{\overline{\frak p}} \circ F_{\frak p}$.
                \item
                $\psi_{\overline{\frak p}} \circ F_{\frak p} = \frak{fg}^{\partial}_j \circ \psi_{\frak p}$
                on $s_{\frak p}^{-1}(0)$.
                \item
                $\tilde F_{\frak p},F_{\frak p}$ is compatible with the coordinate 
                change in the same sense as in Item {\rm (3)} of Proposition \ref{lema362}.
        \end{enumerate}
        The same statement holds for \eqref{map732exted} and \eqref{map733exted}.
\end{definition}



\begin{definition}\label{defn61144}
	Let $\frak u \in \mathcal M_{k+1}^{\rm RGW}(L;\beta)$.
	Let $\check R$ be the very detailed tree describing the combinatorial type of $\frak u$.
	We fix a TSD $\Xi$ at $\frak u$.
	Let $\frak y = (\vec{\frak x},\vec{\sigma},(u'_{v}),(U'_{v}),(\rho_{e}),(\rho_{i}))$ 
	be an inconsistent map with respect to $\Xi$.
	\begin{enumerate}
		\item Remove all the edges $e$ of $\check R$ with $\sigma_e = 0$, and let $\check R_0$
		be one of the connected components of the resulting graph. The union of all the spaces 
		$\Sigma_{\frak y,v}$, where $v$ belongs to $\check R_0$, is called 
		{\it an irreducible component of $\frak y$}.
		If it does not make any confusion, the union of all the interior vertices $v$ of $\check R$,
		which belong to $\check R_0$, is also called an irreducible component.
		\item An irreducible component of $\frak y$ is called a {\it trivial component}
			if the following holds:
		\begin{enumerate}
			\item All the vertices in this component have color ${\rm d}$.
			\item All the homology classes assigned to the vertices in this component are $0$.
		\end{enumerate}
		\item We say $\frak y$ {\it preserves triviality} if for any interior vertex $v$ 
		in a trivial component, the map $u'_{v}$ is constant.
\end{enumerate}
\end{definition}

In this subsection we will construct Kuranishi structures on 
$\mathcal M_{k+1}^{\rm RGW}(L;\beta)$ which are compatible 
with forgetful map in the sense we defined above.

\begin{lemma}\label{lem115555}
	Given any element $\frak u\in\mathcal M_{k+1}^{\rm RGW}(L;\beta)$,
	the Kuranishi neighborhood of $\frak u$, constructed in \cite{part2:kura}, 
	is contained in the set of inconsistent maps which preserve triviality.
\end{lemma}

\begin{proof}
	Suppose $\Xi$ is a small enough TSD such that we can form the obstruction bundle 
	$E_{\frak u,\mathscr F,\Xi}(\frak y)$ over inconsistent maps $\frak y$ with respect to 
	$\Xi$. Let $\Upsilon$ be the TSO given as $(\Xi,\{E_{\frak u,\mathscr F,\Xi}\})$.
	We assume that $\frak y$ is chosen such that it represents an element of 
	$\widehat{\mathcal U}(\frak u,\Upsilon)$. If $[\frak u,\frak p,\phi]$ is a quasi component of
	$\frak u$, then the image of $\phi$ is away from the components of $\frak u$ 
	with trivial homology classes. Consequently, restriction of the obstruction bundle 
	$E_{\frak u,\mathscr F,\Xi}(\frak y)$ to any trivial component of $\frak y$ is trivial.
	Therefore, the restriction of $u_{\frak y}$ to any such component has trivial homology class
	and satisfies the Cauchy-Riemann equation with a trivial obstruction bundle, and hence it 
	is a constant map.
%
\end{proof}
Suppose $\frak y$ is an inconsistent map with respect to $\Xi$. Let $\Xi'$ be another TSD at the same point $\frak u$. If $\Xi'$ is small enough, then we obtain a corresponding inconsistent map $\frak y'$ with respect to $\Xi'$. (See the discussion preceding \cite[Lemma 9.22]{part2:kura}.)
It is clear that $\frak y$ preserves triviality if and only if $\frak y'$ preserves triviality.
\par
We form the forgetful map:
\begin{equation}\label{formdeffgggg}
	\frak{fgg} : \mathcal M^{\rm RGW}_{k+1}(L;\beta) \to \mathcal M^{\rm RGW}_{1}(L;\beta)
\end{equation}
by forgetting the boundary marked points other than the $0$-th one,
by composing the forgetful maps defined in Subsection \ref{subsec:forget}.
\begin{lemma}\label{lem166666}
	Let $\frak u \in \mathcal M^{\rm RGW}_{k+1}(L;\beta)$ and 
	$\frak u' =  \frak{fgg}(\frak u) \in \mathcal M_{1}^{\rm RGW}(L;\beta)$.
	For any TSD $\Xi'= (\vec w_{\frak u'},(\mathcal N_{\frak u',v,i}),(\phi_{\frak u',v}),
	(\varphi_{\frak u',v,e}),\delta')$ at $\frak u'$ there is a TSD 
	$\Xi=(\vec w_{\frak u},(\mathcal N_{\frak u,v,i}),(\phi_{\frak u,v}),
	(\varphi_{\frak u,v,e}),\delta)$ at $\frak u$ such that 
	any inconsistent map $\frak y$ with respect to $\Xi$ which preserves triviality
	induces an inconsistent map $\frak y'$ with respect to $\frak u'$.
\end{lemma}

\begin{proof}
        Let $\check R$ (resp. $\check R'$) be the very detailed DD-ribbon tree 
        describing the combinatorial type of $\frak u$ (resp. $\frak u'$).
	By construction we observe that the vertices of $\check R$
	with color $\rm s$ or $\rm D$ are in one to one correspondence with
	the vertices of the same color in $\check R'$.
        Also the set of vertices of $\check R'$ with color $\rm d$ is a subset of the
        vertices of $\check R$ with color $\rm d$. The difference 
        $C^{\rm int}_{0}(\check R)\setminus C^{\rm int}_{0}(\check R')$
        consists of vertices $v$ such that the map $u_{v}$ is constant on it.\footnote{This is 
        not a necessary and sufficient condition.} In particular, for any such vertex $v$, the component
        $\Sigma_{\frak u,v}$ together with the marked points and nodal points is already source stable.
        Therefore, we can require that the additional marked points $\vec w_{{\frak u}}$ of 
        the TSD $\Xi$ do not belong to such irreducible components of $\Sigma_{\frak u}$.
        Thus we may find marked points $\vec w_{\frak u}$ on $\Sigma_{\frak u'}$
        such that they are identified with $\vec w_{{\frak u'}}$ using the map  
        $\Sigma_{\frak u} \to \Sigma_{\frak u'}$ which collapses the components associated to  
        $C^{\rm int}_{0}(\check R)\setminus C^{\rm int}_{0}(\check R')$. We define $\vec w_{\frak u}$
        to be the set of the additional marked points of $\Xi$. We also assume that the set of transversals 
        of $\Xi$ are identified with that of $\Xi'$ in an obvious way.
        
        We also require that the trivialization of the universal families of irreducible components 
        associated to  $\Xi$ and to $\Xi'$ to be related to each other as follows. For any sphere component 
        we assume that the associated trivializations agree with each other.
	For $v \in C^{\rm int}_{0}({\check R}')
        \subset C^{\rm int}_{0}(\check R)$ with level $0$, let $\mathcal M^{{\rm source}}_{\frak u,v}$,
        $\mathcal M^{{\rm source}}_{\frak u',v}$ be the corresponding moduli spaces of marked disks
        and $\mathcal{C}^{{\rm source}}_{\frak u,v} \to \mathcal M^{{\rm source}}_{\frak u,v}$,
        $\mathcal{C}^{{\rm source}}_{\frak u',v} \to \mathcal M^{{\rm source}}_{\frak u',v}$ be the universal 
        families.
        We take a trivialization of $\mathcal{C}^{{\rm source}}_{\frak u,v}$
        over a sufficiently small neighborhood $ \mathcal V_{\frak u,v}^{{\rm source}}$ 
        of $\Sigma_{\frak u,v}$ so that the following diagram commutes:
        \begin{equation}\label{diadia}
	\begin{tikzcd}
		 \mathcal V_{\frak u,v}^{{\rm source}} \times \Sigma_v  \ar[r,"\phi_{\frak u,v}"]\ar[d]& \mathcal{C}^{{\rm source}}_{\frak u,v} \ar[d]\\
		  \mathcal V_{\frak u',v}^{{\rm source}} \times \Sigma_v  \ar[r,"\phi_{{\frak u'},v}"]& \mathcal{C}^{{\rm source}}_{\frak u',v}
	\end{tikzcd}
	\end{equation}
        where the vertical arrows are obvious forgetful maps.
        For the trivializations of the universal families of disk components
        corresponding to $v \in C^{\rm int}_{0}(\check R)
        \setminus C^{\rm int}_{0}(\check R')$, that are parts of $\Xi$, we take an arbitrary choice.

        For $v\in C^{\rm int}_{0}({\check R}')\subset C^{\rm int}_{0}(\check R)$ and an edge $e$
        incident to $v$, we pick $\varphi_{\frak u,v,e}$ to be the analytic family induced by $\varphi_{\frak u',v,e}$.
        In the case that $v\in C^{\rm int}_{0}({\check R})\setminus C^{\rm int}_{0}(\check R')$, the corresponding
        component $\Sigma_{\frak u,v}$ in addition to boundary marked points has at most two boundary nodes. 
        If there are two boundary nodes inducing edges $e_+$ and $e_-$ incident to $v$, then we can identify 
        $\Sigma_{\frak u,v}$ with the strip $[0,1]\times \bbR$ where the boundary node associated to 
        $e_{\pm}$ is in correspondence with the point at $\pm \infty$ on the boundary of $[0,1]\times \bbR$.
        We fix one such identification and let $[0,1]\times [T,\infty)$ and $[0,1]\times (-\infty,-T]$, for a large 
        value of $T$, induce the analytic families of coordinates $\varphi_{\frak u,v,\pm e}$. 
        In the case that there is only one interior
        edge incident to $v$, we follow a similar strategy with the difference that we only need to use the half 
        strip $[0,1]\times [T,\infty)$ to define the corresponding analytic family of coordinates.
        We also let $\delta=\delta'$.
%

        Now, let
        $$
          \frak y = (\vec{\frak x}_{\frak y},\vec{\sigma}_{\frak y},(u'_{{\frak y},v}),(U'_{{\frak y},v}),
          (\rho_{{\frak y},e}),(\rho_{{\frak y},i}))
        $$
        be an inconsistent map with respect to $\Xi$ which preserves triviality. We wish to define
        $$
        {\frak y'} = (\vec{\frak x}_{\frak y'},\vec{\sigma}_{{\frak y'}},
        (u'_{{\frak y'},v}),
        (U'_{{\frak y'},v}),(\rho_{{\frak y'},e}),(\rho_{{\frak y'},i}))
        $$
        an inconsistent map with respect to $\Xi'$.
        It is clear from the definition of $\Xi$ that there are $\vec{\frak x}_{{\frak y'}},\vec{\sigma}_{{\frak y'}}$ 
        such that
        \begin{equation}\label{form6200}
	        \Sigma_{\frak u}(\vec{\frak x}_{\frak y},\vec{\sigma}_{\frak y})
		\cong \Sigma_{{\frak u'}}(\vec{\frak x}_{{\frak y'}},\vec{\sigma}_{{\frak y'}}).
        \end{equation}
        We take $\rho_{{\frak y},e} = \rho_{{\frak y'},e}$, $\rho_{{\frak y},i} = \rho_{{\frak y'},i}.$
        Moreover, $U'_{{\frak y},v} = U'_{{\frak y'},v}$ and $u'_{{\frak y},v} = u'_{{\frak y'},v}$
        if the color of $v$ is ${\rm s}$.

        We consider a disk component $\Sigma_{{\frak y'},v}$.
        There exists a unique irreducible component (in the sense of 
        Definition \ref{defn61144}, where we use $\vec{\sigma}_{{\frak y'}}$) which contains this component.
        We denote by $\Sigma^+_{\frak u',v}(\vec{\frak x}_{{\frak y'}},\vec{\sigma}_{{\frak y'}})$ the union of 
        the disk components contained in this
        irreducible component.\footnote{See \cite[(8.11)]{part2:kura} for the meaning of the symbol $+$.}
        We take the irreducible components of $\frak y$ which correspond to it
        and define $\Sigma^+_{{\frak u},v}(\vec{\frak x}_{{\frak y}},\vec{\sigma}_{{\frak y}})$
        in the same way.
        By (\ref{form6200}) we have an isomorphism
        \begin{equation}\label{form6201}
        \Sigma^+_{\frak u',v}(\vec{\frak x}_{{\frak y'}},\vec{\sigma}_{{\frak y'}})
        \cong \Sigma^+_{{\frak u},v}(\vec{\frak x}_{{\frak y}},\vec{\sigma}_{{\frak y}}).
        \end{equation}
        The maps $u'_{\frak y,v}$ for various $v$ in this irreducible component 
        induces a map 
        \begin{equation}\label{form6202}
        (\Sigma^+_{{\frak u},v}(\vec{\frak x}_{{\frak y}},\vec{\sigma}_{{\frak y}}),
        \partial \Sigma^+_{{\frak u},v}(\vec{\frak x}_{{\frak y}},\vec{\sigma}_{{\frak y}})) 
        \to (X,L).
        \end{equation} 
        This map is smooth. (Since
        $\Sigma^+_{{\frak u},v}(\vec{\frak x}_{{\frak y}},\vec{\sigma}_{{\frak y}})$ is obtained by gluing along 
        the components associated to the level $0$ edges, the maps $u'_{\frak y,v}$ are consistent on overlaps.)
	We use \eqref{form6201} and \eqref{form6202} to define $u'_{{\frak y'},v}$.
        Using the fact that $\frak y$ is an consistent map preserving triviality,
        it is easy to see that $u'_{{\frak y'},v}$ for various $v$ are consistent 
        at the nodal points corresponding to the level 0 edges $e$ with $\sigma_{{\frak y'},e} = 0$, and 
        ${\frak y'} = (\vec{\frak x}_{\frak y'},\vec{\sigma}_{{\frak y'}},(u'_{{\frak y'},v}),(U'_{{\frak y'},v}),
        (\rho_{{\frak y'},e}),(\rho_{{\frak y'},i}))$ is an inconsistent map with respect to $\Xi'$.
 \end{proof}

\begin{remark}
        The notion of preserving triviality plays an important role in the proof.
        The other important point is that we do not put any obstruction bundle on the components where 
        the maps are constant.
\end{remark}

Let $\frak u$, $\frak y$, $\frak u'$ and $\frak y'$ be as in Lemma \ref{lem115555} and $\check R$, $\check R'$ be the very detailed tree describing the combinatorial types of $\frak u$, $\frak u'$, respectively. We define
\[
  \aligned
  L^2_{m,\delta,{\rm nontri}}(\frak y,{\frak u})=
  &\bigoplus_{v \in C^{\rm int}_{0}(\check R) ; c(v) = {\rm s}}
L^2_{m,\delta}(\Sigma^+_{\frak y,v};(u'_{\frak y,v})^*TX \otimes \Lambda^{0,1}) \\
&\oplus \bigoplus_{v \in C^{\rm int}_{0}(\check R) ; c(v) = {\rm D}}
L^2_{m,\delta}(\Sigma^+_{\frak y,v};(\pi\circ U'_{\frak y,v})^*T\mathcal D \otimes \Lambda^{0,1})
\\
&\oplus 
\bigoplus_{v \in C^{\rm int}_{0}(\check R) ; c(v) = {\rm d}, 
\atop\text{$u'_{\frak y,v}$ is not constant}}
L^2_{m,\delta}(\Sigma^+_{\frak y,v};(u'_{\frak y,v})^*TX \otimes \Lambda^{0,1}) 
\endaligned
\]
$$
\aligned
L^2_{m,\delta,{\rm nontri}}
({\frak y'},{\frak u'})
=
&\bigoplus_{v \in C^{\rm int}_{0}(\check R') ; c(v) = {\rm s}}
L^2_{m,\delta}(\Sigma^-_{{\frak y'},v};(u'_{{\frak y'},v})^*TX \otimes \Lambda^{0,1}) \\
&\oplus \bigoplus_{v \in C^{\rm int}_{0}(\check R') ; c(v) = {\rm D}}
L^2_{m,\delta}(\Sigma^-_{{\frak y'},v};(\pi\circ U'_{{\frak y'},v})^*T\mathcal D \otimes \Lambda^{0,1})
\\
&\oplus 
\bigoplus_{v \in C^{\rm int}_{0}(\check R') ; c(v) = {\rm d}, 
\atop\text{$u'_{{\frak y'},v}$ is not constant}}
L^2_{m,\delta}(\Sigma^-_{{\frak y'},v};(u'_{{\frak y'},v})^*TX \otimes \Lambda^{0,1}).
\endaligned
$$
There are canonical identification between the components appearing in the above two formulas. Therefore, there exists a canonical map
\begin{equation}
	I_{{\frak y}\frak y'} : L^2_{m,\delta,{\rm nontri}}(\frak y',\frak u')
	\to L^2_{m,\delta,{\rm nontri}}(\frak y,{\frak u}).
\end{equation}
\begin{definition}\label{defn118}
	Let $\{E_{\frak u,\Xi}(\frak y)\}$ and $\{E_{\frak u',\Xi'}(\frak y')\}$ be obstruction bundle data for the moduli spaces
	$\mathcal M_{k+1}^{\rm RGW}(L;\beta)$ and $\mathcal M_{1}^{\rm RGW}(L;\beta)$, respectively.
	We say that they are {\it compatible with the forgetful map} if 
	\begin{equation}\label{form19199}
		I_{{\frak y}\frak y'} (E_{\frak u',\Xi'}(\frak y'))=E_{\frak u,\Xi}(\frak y)
	\end{equation}
	when $\frak u'$, $\Xi'$, $\frak y'$ are related to $\frak u$, $\Xi$, $\frak y$ as in Lemma \ref{lem115555}.
\end{definition}
\begin{definition}\label{defn119}
	A system of obstruction bundle data for
	$\{\mathcal M_{k+1}^{\rm RGW}(L;\beta)\}_{\omega\cap \beta\leq E}$ is said to be 
	{\it compatible with the forgetful map} if Definition \ref{defn118}
	holds for each of the moduli spaces $\mathcal M_{k+1}^{\rm RGW}(L;\beta)$ and 
	$\mathcal M_{1}^{\rm RGW}(L;\beta)$ with $\omega\cap \beta\leq E$.
\end{definition}

Suppose $\{E_{\frak u,\Xi}(\frak y)\}$ is a system of obstruction bundle data for $\{\mathcal M_{k+1}^{\rm RGW}(L;\beta)\}_{\omega\cap \beta\leq E}$, which is disk-component-wise and is compatible with forgetful map. Let $\frak u$ be an element of $\mathcal M_{k+1}^{\rm RGW}(L;\beta)$ and $\frak u'=\frak {fgg}(\frak u)$. Suppose $\Xi$, $\Xi'$ are TSDs at $\frak u$, $\frak u'$ which are related to each other as in Lemma \ref{lem166666}. Using Lemmas \ref{lem115555} and \ref{lem166666} and consistency of obstruction bundle data with the forgetful map, we can define a map:
\[
  F_{\frak u} : \widehat {\mathcal U}(\frak u;\Upsilon) \to \widehat {\mathcal U}(\frak u';\Upsilon').
\]
Here $\Upsilon = (\Xi,E_{\frak u,\Xi})$, $\Upsilon' = (\Xi',E_{\frak u',\Xi'})$.
In the process of forgetting boundary marked points and passing from $\frak u$ to $\frak u'$, we only might collapse disc components. Since the elements of $\Gamma_{\frak u}$ and $\Gamma_{\frak u'}$ act as identity on disc components, the isotropy groups $\Gamma_{\frak u}$ and $\Gamma_{\frak u'}$ are isomorphic. The map $F_{\frak u}$ is also $\Gamma_\frak u$-equivariant. 

It is straightforward to lift the map $F_{\frak u}$ to a $\Gamma_{\frak u}$-equivariant map:
\[
  \tilde F_{\frak u}:\mathcal E_{\frak u} \to \mathcal E_{\frak u'}
\]
such that
\[
  \tilde F_{\frak u}\circ \frak s_{\frak u}=\frak s_{\frak u'} \circ F_{\frak u},
\]
and for any $\frak y\in \frak s_{\frak u}^{-1}(0)/\Gamma_{\frak u}$ we have
\[
  \psi_{\frak u'}\circ F_{\frak u}(\frak y)=\frak {fgg} \circ \psi_{\frak u}(\frak y)
\]
The maps $F_{\frak u}$ and $\tilde F_{\frak u}$ are also compatible with coordinate changes. We can summarize this discussion as follows:

\begin{theorem}\label{comp-forg}
	Suppose a system of obstruction bundle data $\{E_{\frak u,\Xi}(\frak y)\}$ for moduli spaces
	$\{\mathcal M_{k+1}^{\rm RGW}(L;\beta)\}_{\omega\cap \beta\leq E}$
	is disk-component-wise and is compatible with forgetful map.
	Then the resulting system of Kuranishi structures 
	is compatible at the boundary components and corners (in the sense of Propositions \ref{lema362}, \ref{lem364} and \ref{lem365})
	and compatible with the forgetful map in the sense of Definition \ref{lemma369}.
\end{theorem}
\begin{proof}
	Compatibility of the forgetful map with Kuranishi structures of moduli spaces 
	$\mathcal M_{k+1}^{\rm RGW}(L;\beta)$ is equivalent to the existence of the maps 
	$F_{\frak u}$ and $F_{\frak u'}$ with the above properties. We just need to point out that in Definition \ref{lemma369}
	we consider the map 
	\[
	  \frak{fg}_j^\partial  : \mathcal M_{k+1}^{\rm RGW}(L;\beta) \to  \mathcal M_{k}^{\rm RGW}(L;\beta) .
	\] 
	given by forgetting the $j$-th marked point. The proof of a similar result for the map $\frak{fg}_j^\partial$
	is essentially the same. Let $\frak u_1\in \mathcal M_{k+1}^{\rm RGW}(L;\beta)$,
	$\frak u_2=\frak{fg}_j^\partial(\frak u_1)$. Starting with a TSO $\Upsilon_2$
	at $\frak u_2$, we can follow the proof of Lemma \ref{lem166666} to define a TSO $\Upsilon_1$ at $\frak u_1$,
	and form a map from $\widehat {\mathcal U}(\frak u_1,\Upsilon_1)$ to
	$\widehat {\mathcal U}(\frak u_2,\Upsilon_2)$. The remaining properties can be verified in 
	a similar way.
\end{proof}
\begin{remark}
	In general, one needs to be careful about the differentiability 
	of $F_{\frak u}$.
	The strata-wise smoothness is easy to show by elliptic regularity.
	The issue of differentiability, when a change of stratum happens, is discussed
	in \cite[page 778]{fooobook2}.
	This issue is relevant to the application of Theorem \ref{comp-forg},
	when we want to pull-back a multi-valued perturbation by the forgetful map.

	There are two ways to resolve this issue.
	First we can consider multi-sections which have exponential decay in the gluing parameter $T$. 
	(We use $T,\theta$ where $\sigma = \exp(-(T+\theta\sqrt{-1}))$.)
	Even though the forgetful map $F_{\frak u}$ may not be smooth,
	the pull back of a multi-section with exponential decay 
	is a multi-section which is not only smooth but also has an exponential decay. (See also \cite[page 778]{fooobook2}.)
	
	In our situation discussed in the next subsection,
	we can use a simpler method to resolve this issue.
	For the purpose of this paper,
	we need to pull back a never vanishing multi-section. 
	Thus pulling back the multi-section in $C^0$ sense is enough. 
	This is because we need differentiability of the multi-section only 
	in a neighborhood of its zero set.
\end{remark}

To complete our construction of a system of Kuranishi structures which is compatible with the forgetful map, it remains to prove the following result.
\begin{lemma}
There exists a system of obstruction bundle data which is disk component wise and 
is compatible with forgetful map.
\end{lemma}

\begin{proof}
	The proof is essentially the same as the proof in Section \ref{sub:systemconst}.
	As before, we construct the system of obstruction bundle data by induction on $\beta \cap [\omega]$.
	In each step of the induction, we firstly construct an 
	obstruction bundle data on $\mathcal M^{\rm RGW}_{1}(L;\beta)$.
	This system automatically induces an obstruction bundle data 
	on  $\mathcal M^{\rm RGW}_{k+1}(L;\beta)$ by requiring Condition \eqref{form19199}.
	

	To be more detailed, we fix a finite subset $\frak P(\beta)$ of 
	$\mathcal M_{1}^{\rm RGW}(L;\beta)$ as in (OBI) and a vector space $E_{\frak p}$ for 
	$\frak p \in  \frak P(\beta)$ as in (OBII).
	We also fix spaces $\mathscr F_{\beta}$, $\mathscr F^{\circ}_{\beta}$ 
	which fixes a set of quasi components for each $\frak u \in \mathcal M_{1}^{\rm RGW}(L;\beta)$.
	We require these objects satisfy Conditions
	 \ref{conds1023}, \ref{conds1025}, \ref{conds1027}, \ref{conds30}, \ref{conds31}.
	If $\frak u \in \mathcal M_{k+1}^{\rm RGW}(L;\beta)$, then we define 
	$\mathscr F_{\beta}(\frak u)$, $\mathscr F^{\circ}_{\beta}(\frak u)$
	to be $\mathscr F_{\beta}(\frak u')$, $\mathscr F^{\circ}_{\beta}(\frak u')$
	where $\frak u'=\frak{fgg}(\frak u)$. Since the obstruction bundle data for 
	$\mathcal M_{1}^{\rm RGW}(L;\beta)$ satisfies Conditions
	 \ref{conds1023}, \ref{conds1025}, \ref{conds1027}, \ref{conds30} and \ref{conds31}
	the induced obstruction bundle for $\mathcal M_{k+1}^{\rm RGW}(L;\beta)$ satisfies 
	the corresponding conditions.
	\end{proof}
	
\subsection{Perturbation and forgetful map}
\label{subsub:constmulti2}

We next discuss compatibility of perturbations and forgetful map. 

\begin{prop}\label{prop12866}
	Let $L_0$, $L_1$ be a pair of transversal compact monotone 
	Lagrangians in $X \setminus \mathcal D$ such that their minimal Maslov numbers are $2$. 
	For a positive number $E$, there is a system of multi-valued perturbations 
        $\{\frak s_{n}\}$ on the moduli spaces $\mathcal M^{\rm RGW}_{k_1,k_0}(L_1,L_0;p,q;\beta)$
        of virtual dimension $\le 1$ and $\omega \cap \beta \le E$ such that: 
        \begin{enumerate}
        \item It satisfies Item (1) of Theorem \ref{prop61111}.
        \item The multi-valued perturbations $\{\frak s_{n}\}$ are compatible with the description of the 
        boundary given by Propositions \ref{lema362}, \ref{lem364} and \ref{lem365}.
        \item The multi-valued perturbations $\{\frak s_{n}\}$ are compatible with the forgetful map of the 
		marked points given by Theorem \ref{comp-forg}.
	\item It satisfies Item (3) of Theorem \ref{prop61111}.
        \end{enumerate}       
\end{prop}

\begin{proof}
	For $j=1,\,2$ and $\alpha\in \Pi_2(X;L_j)$ with Maslov index $2$, we 
	fix a multi-valued perturbation on $\mathcal M_1^{\rm RGW}(L_j;\alpha)$ such that it induces a transversal
	multi-valued perturbation on $\mathcal M_1^{\rm RGW}(L_j;\alpha;p)$ for any $p \in L_0 \cap L_1$. 
        We extend these multi-valued perturbations to all moduli spaces $\mathcal M_{k+1}(L_j;\alpha)$ 
        in the $C^0$ sense such that they are compatible over the boundary in a similar sense as 
        in Theorem \ref{prop61111}. We use these multi-valued perturbations and 
        induction to define the required 
        multi-valued perturbations on the 
	moduli space $\mathcal M^{\rm RGW}_{k_1,k_0}(L_1,L_0;p,q;\beta)$.
	To be more detailed, we construct multi-valued perturbations on 
	$\mathcal M^{\rm RGW}_{0,0}(L_1,L_0;p,q;\beta)$ by induction on $\omega\cap \beta$.
	The multi-valued perturbation on the general moduli space 
	 $\mathcal M^{\rm RGW}_{k_1,k_0}(L_1,L_0;p,q;\beta)$ is given by pulling back 
	from $\mathcal M^{\rm RGW}_{0,0}(L_1,L_0;p,q;\beta)$.
	Here we use consistency of Kuranishi structures with the forgetful map.

	Suppose we have constructed required multi-valued perturbations 
	for $\beta$ with $\beta \cap \omega < \alpha \cap \omega$.
	We use the induction hypothesis and the constructed multi-valued perturbations for the moduli 
	spaces of discs to define a perturbation on 
	$\partial \mathcal M^{\rm RGW}_{0,0}(L_1,L_0;p,q;\beta)$ in the same way as in 
	Theorem \ref{prop61111}. We wish to analyze zeros of our induced 
	multi-section on the boundary of $\mathcal M^{\rm RGW}_{0,0}(L_1,L_0;p,q;\beta)$.
	In compare to Theorem \ref{prop61111}, the new types of zeros are 
	given by disc bubbles with Maslov index $2$. Such boundary components have the form
	\begin{equation}\label{case-I}
	  \mathcal M^{\rm RGW}_{1,0}(L_1,L_0;p,q;\beta_0) \,\hat\times_{L_1}\, \mathcal M_1^{\rm RGW}(L_1;\beta_1)
	\end{equation}
	or 
	\begin{equation}\label{case-II}
	  \mathcal M^{\rm RGW}_{0,1}(L_1,L_0;p,q;\beta_0) \,\hat\times_{L_0}\, \mathcal M_1^{\rm RGW}(L_0;\beta_2),
	\end{equation}
	where $\beta_0 + \beta_1 = \alpha$ and 
	the Maslov index of $\beta_1$ is $2$. We focus on the boundary components of the form in \eqref{case-I}.
	The other case is similar. There are two cases to consider.
	
	\par\smallskip
	\noindent {\bf (Case 1)} ($\beta_0\neq 0$):
	The virtual dimension of 
	$\mathcal M^{\rm RGW}_{1,0}(L_1,L_0;p,q;\beta_0)$ is
	\[
	  \dim(\mathcal M^{\rm RGW}_{0,0}(L_1,L_0;p,q;\alpha))-1.
	\]
	If the virtual dimension of $\mathcal M^{\rm RGW}_{0,0}(L_1,L_0;p,q;\alpha)$ is not greater than $0$,
	then the multi-section does not vanish on this component. To treat the case that the virtual dimension
	of $\mathcal M^{\rm RGW}_{0,0}(L_1,L_0;p,q;\alpha)$ is $1$, note that the multi-section of
	$\mathcal M^{\rm RGW}_{1,0}(L_1,L_0;p,q;\beta_0)$ is the pull-back of the multi-valued perturbation
	on $\mathcal M^{\rm RGW}_{0,0}(L_1,L_0;p,q;\beta_0)$. This latter moduli space has virtual dimension 
	$-1$ and hence the multi-section does not vanish on it. Therefore, the multi-section does not 
	have any zero on the moduli spaces $\mathcal M^{\rm RGW}_{1,0}(L_1,L_0;p,q;\beta_0)$
	and \eqref{case-I}.
	\par\smallskip
	\noindent {\bf (Case 2)} ($\beta_0=0$):
	In this case, $p=q$ and $\alpha= 0 \# \beta_1$ where $\beta_1$ is a homology class 
	in $\Pi_2(X;L_1)$ with Maslov index $2$.
	Therefore, the corresponding boundary component is
	identified with  $\mathcal M_1^{\rm RGW}(L_1;\alpha;p)$ where $p \in L_0\cap L_1$.
	We defined a multi-valued perturbation on this moduli space  
	such that its zero set is cut down transversely and consists of isolated points.
	Now we can proceed as in Theorem \ref{prop61111} 
	to complete the construction of multi-valued perturbations.
\end{proof}

\section{Proof of the Main Theorem}
\label{sec:mainthmproof}

\subsection{Floer Homology and Auxiliary Choices}
\label{subsec:welldef}
In this subsection, we show that Floer homology is independent of the auxiliary choices which we use in the definition. At the same time, we verify Item (2) of Theorem \ref{mainthm-part3}. To achieve these goals, we adapt standard arguments to the present set up.
\subsubsection{Hamiltonian Vector Fields and Floer Chain Complexes}
Let $L_0,L_1$ be Lagrangian submanifolds of $X \setminus \mathcal D$ and $\psi_H : X \to X$ be a Hamiltonian diffeomorphism generated by a compactly supported smooth map $H : (X \setminus \mathcal D) \times [0,1] \to \bbR$.
Our next goal is to compare Floer homology groups $HF(L_1,L_0;X \setminus \mathcal D)$
and  $HF(L_1,\psi_H(L_0);X \setminus \mathcal D)$.
\begin{remark}
	In Theorem \ref{mainthm-part3}, we state the isomorphism
	$$
	HF(L_1,L_0;X \setminus \mathcal D)\cong HF(\psi_{H'}(L_1),\psi_H(L_0);X \setminus \mathcal D).
	$$
	Namely, we move $L_1$ and $L_0$ by possibly different Hamiltonian diffeomorphisms. In fact, the isomorphism:
	\begin{equation}\label{formula410}
		HF(L_1,L_0;X \setminus \mathcal D)\cong HF(\psi(L_1),\psi(L_0);X \setminus \mathcal D),
	\end{equation}
	in the case that $L_0$, $L_1$ are moved by the same symplectic diffeomorphism $\psi$ is easy to prove
	because the moduli spaces used to define the left hand side of \eqref{formula410} with respect to 
	$J$ are isomorphic, as spaces with Kuranishi structures, to the moduli spaces used to define the right hand side 
	with respect to $\psi_*J$. Therefore, \eqref{formula410} follows from Proposition 
	\ref{prop42323} below about the independence of  Floer homology
	from the choice of almost complex structure. 
	As a result, to prove Item (2) of Theorem \ref{mainthm-part3}, it suffices to consider the case that
	we only move $L_0$ by a Hamiltonian diffeomorphism $\psi_H$.
\end{remark}

We firstly consider perturbations of holomorphic strips using Hamiltonian vector fields and time dependent almost complex 	structures. These perturbations will be also useful in Subsection \ref{thecaseofLandL}. Let $H : (X \setminus \mathcal D) \times [0,1] \to \bbR$ be as above. We put $H_t(x) = H(x,t)$ and let $X_{H_t}$ be the Hamiltonian vector field associated to $H_t$. Let $\mathcal J = \{ J_t \mid t\in [0,1]\}$ be a family of almost complex structures which is  parametrized by $[0,1]$. We require that, for each $t$, the almost complex structure $J_t$ has the form of the almost complex structures constructed in 
\cite[Subsection 3.2]{part1:top}.

We consider the following equation for $u : \bbR \times [0,1] \to X$:
\begin{equation}\label{Flequation}
	\frac{\partial u}{\partial \tau}+ J_t \left(\frac{\partial u}{\partial t} - X_{H_t}\right)= 0.
\end{equation}
In order to state the required asymptotic behavior of $u$, we need to replace intersection points $L_0 \cap L_1$ with the following space:
\begin{definition}
        Let $R(L_0,L_1;H)$ be the set of elements $\gamma 
        \in \Omega(L_0,L_1)$ such that
        \begin{equation}
        \frac{d\gamma}{dt}(t) = X_{H_t}(\gamma(t)).
        \end{equation}
        If $o \in \pi_0(\Omega(L_0,L_1))$, we define the subset 
        $$
        R(L_0,L_1;H;o) = R(L_0,L_1;H) \cap \Omega(L_0,L_1;o).
        $$
\end{definition}
Hereafter we assume $\psi_H(L_0)$ is transversal to $L_1$.
It implies that the set $R(L_0,L_1;H;o)$ is  finite.

For $\gamma_-, \gamma_+ \in R(L_0,L_1;H;o)$, we define $\Pi_2(X;L_1,L_0;\gamma_-,\gamma_+)$ to be the set of all homology classes of maps $u : \bbR \times [0,1] \to X$ such that $u(\tau,0) \in L_0$, $u(\tau,1) \in L_1$ and
\begin{equation}\label{form34342rev}
        \lim_{\tau \to -\infty} u(\tau,t) = \gamma_-(t),
        \qquad
        \lim_{\tau \to -\infty} u(\tau,t) = \gamma_+(t).
\end{equation}
The definition of homology classes is the same as \cite[Definition 2.2]{part1:top}.
\begin{definition}
        Let $\gamma_-, \gamma_+ \in R(L_0,L_1;H;o)$,
        $\beta \in \Pi_2(X;L_1,L_0;\gamma_-,\gamma_+)$.
        We define  
        $$
        \mathcal M_{k_1,k_0}^{\rm reg}(L_1,L_0;\gamma_-,\gamma_+;\beta;H,
        \mathcal J;{\bf m})
        $$
        to be the set of all objects 
        $((\Sigma,\vec z_0,\vec z_1,\vec w),u)$ satisfying the same condition as 
        \cite[Definition 3.80]{part1:top} except the following two points:
        \begin{enumerate}
                \item
                We require $u$ satisfies \eqref{Flequation} instead of the Cauchy-Riemann equation.
                \item
                We require \eqref{form34342rev} instead of \cite[(3.81)]{part1:top}.
        \end{enumerate}
\end{definition}

We can define its compactification $\mathcal M^{\rm RGW}_{k_1,k_0}(L_1,L_0;\gamma_-,\gamma_+;\beta;H,\mathcal J)$ in the same
way as $\mathcal M^{\rm RGW}_{k_1,k_0}(L_1,L_0;\gamma_-,\gamma_+;\beta)$.

By replacing $\mathcal M^{\rm RGW}_{k_1,k_0}(L_1,L_0;p,q;\beta)$
with $\mathcal M^{\rm RGW}_{k_1,k_0}(L_1,L_0;\gamma_-,\gamma_+;\beta;H,\mathcal J)$,
we obtain a modified version of Floer chain complex.
That is to say, we define:
\begin{equation}
CF(L_1,L_0;H,\mathcal J;o)
=
\bigoplus_{\gamma \in R(L_0,L_1;H;o)}\bbQ [\gamma]
\end{equation}
and
$$
\partial([\gamma_-])
= 
\sum_{\gamma_+,\beta}\#\mathcal M^{\rm RGW}_{0,0}(L_1,L_0;\gamma_-,\gamma_+;\beta) 
[\gamma_+],
$$
where the sum is taken over all $\gamma_+,\beta$ such that 
$\mathcal M^{\rm RGW}_{0,0}(L_1,L_0;\gamma_-,\gamma_+;\beta)$ has dimention 0.
(As before we take a system of multi-valued perturbation $\{\widehat{\frak s}^n\}$ and
the boundary operator depends on $\{\widehat{\frak s}^n\}$ and $n$.)

        Let $L_0,L_1$ and $H$ be given such that $\psi_H(L_0)$ is transversal to $L_1$.
        Let $\gamma : [0,1] \to X$ be an element of $\Omega(L_0,L_1)$.
        We define
        \begin{equation}\label{map41414}
        (\psi_H)_*(\gamma)(t)
        :=
        (\psi_H^t\circ (\psi_H)^{-1})(\gamma(t)).
        \end{equation}
        Here $\psi^t_{H}$ is defined by
        $\psi^0_H(x) = x$ and 
        $$
        \frac{d}{dt} \psi^t_H(x) = X_{H_t}(\psi^t_H(x)).
        $$
        This construction determines a map:
        $$
        (\psi_H)_* : \Omega(\psi_H(L_0),L_1) \to \Omega(L_0,L_1).
        $$
        The map
        $(\psi_H)_*$ induces an isomorphism:
        \begin{equation}\label{form4545}
        \psi_H(L_0)\cap_{(\psi_H)_*(o)} L_1  \cong R(L_0,L_1;H;o). 
        \end{equation}
        Thus it induce a homomorphism:
        \begin{equation}
	        (\psi_H)_* : CF(L_1,\psi_H(L_0);o) \cong CF(L_1,L_0;H,\mathcal J;(\psi_H)_*(o)).
        \end{equation}
        We define 
        $$
		(\psi_H)_* : \Pi_2(X;L_1,\psi_H(L_0);o)\to \Pi_2(L_1,L_0;H;(\psi_H)_*(o))
        $$
        in a similar way.
\begin{lemma}
	The map $(\psi_H)_*$ determines the following commutative diagrams.
        \[
	\begin{tikzcd}
		\pi_1(\Omega(\psi_H(L_0),L_1;o))   \ar[r,"\mu"]\ar[d,"(\psi_H)_*"]& \mathbb Z \ar[d, "="]\\
		\pi_1(\Omega(L_0,L_1;(\psi_H)_*(o))) \ar[r,"\mu"]& \mathbb Z
	\end{tikzcd}
	\hspace{1.2cm}
	\begin{tikzcd}
		\pi_1(\Omega(\psi_H(L_0),L_1;o))   \ar[r,"\omega(\cdot)"]\ar[d,"(\psi_H)_*"]& \mathbb R \ar[d, "="]\\
		\pi_1(\Omega(L_0,L_1;(\psi_H)_*(o))) \ar[r,"\omega(\cdot)"]& \mathbb R
	\end{tikzcd}
	\]
\end{lemma}
For the proof, see \cite[Proof of Lemma 5.10]{fooobook}.

\begin{cor}
	If $o \in \pi_0(\Omega(\psi_H(L_0),L_1))$ satisfies Condition \ref{cond420} {\rm (2)},
	then $(\psi_H)_*(o) \in \pi_0(\Omega(L_0,L_1))$ also satisfies the same condition.
\end{cor}
We next observe that
there exists an isomorphism:
	$$
	\aligned
	&\mathcal M_{k_1,k_0}^{\rm reg}(L_1,L_0;\gamma_-,\gamma_+;H,\mathcal J;(\psi_H)_*(\beta);{\bf m})\cong
	\mathcal M_{k_1,k_0}^{\rm reg}(L_1,\psi_H(L_0);p,q;\beta;{\bf m})
	\endaligned
	$$
	where $p$, $q$ are mapped to $\gamma_-$, $\gamma_+$ by the isomorphism 
	\eqref{form4545}.
	Here $\mathcal J = \{J_t \mid t\in [0,1]\}$ is defined by 
	$J_t = (\psi_H^t)_*J$.
Using the fact that $\psi_H^t$ is the identity map in a neighborhood of 
$\mathcal D$ we obtain an isomorphism
$$
\mathcal M^{\rm RGW}_{k_1,k_0}(L_1,L_0;\gamma_-,\gamma_+;(\psi_H)_*(\beta);\mathcal J,H)
\cong
\mathcal M^{\rm RGW}_{k_1,k_0}(L_1,\psi_H(L_0);p,q;\beta)
$$
of the spaces with Kuranishi structure.
Thus, if $o \in \pi_0(\Omega(L_0,\psi_H(L_0)))$ satisfies Condition \ref{cond420} (2), then we obtain an isomorphism of chain complexes:
\begin{equation}
(\psi_H)_* : 
(CF(L_1,\psi_H(L_0);o),\partial) \cong
(CF(L_1,L_0;H,\mathcal J;(\psi_H)_*(o))),\partial).
\end{equation}
Consequently, to prove independence of Floer homology from Hamiltonian 
isotopies of Lagrangian submanifolds, in the case  Condition \ref{cond420} (2) is 
satisfied, it suffices to prove independence from Hamiltonian perturbations.
To be more precise, let $H, H' : (X\setminus \mathcal D) \times [0,1] \to \bbR$ be Hamiltonian 
functions with compact support.
Let $\mathcal J = \{J_t \mid t \in [0,1] \}$ and $\mathcal J' = \{J'_t \mid t \in [0,1] \}$  
be 1-parameter families of almost complex structures given as in \cite[Subsection 3.2]{part1:top}.
Note that as we remarked in 
\cite[Remark 3.9]{part1:top} the set of such almost complex structures has trivial homotopy groups. 

\begin{prop}\label{prop42323}
        If $L_1,L_0$ and $o \in \pi_0(\Omega(L_0,L_1))$ satisfies Condition \ref{cond420}
        then we have an isomorphism
        $$
        HF(L_1,L_0;H,\mathcal J;o)
        \cong 
        HF(L_1,L_0;H',\mathcal J';o).
        $$
        The same holds if Condition \ref{cond420} {\rm (2)} is satisfied and 
        $\frak{PO}_{L_0}(1) = \frak{PO}_{L_1}(1)$.
\end{prop}

\subsubsection{Construction of Continuation Maps}
\label{subsub:conmap}

We take a 2-parameter family of compatible almost complex structures $\mathcal{JJ} = \{J_{\tau,t} \mid \tau \in \bbR, t \in [0,1]\}$.
Let $\mathcal H : \bbR \times (X\setminus \mathcal D) \times [0,1] \to \bbR$
be a smooth function.
We assume that the following conditions hold.
\begin{conds}
        \begin{enumerate}
                \item For each $(\tau,t)$, the almost complex structure $J_{\tau,t}$ has 
                the form of the almost complex structures constructed in \cite[Subsection 3.2]{part1:top}.
                \item
                $J_{\tau,t} = J_t$ if $\tau$ is sufficiently small, and 
                $J_{\tau,t} = J'_t$ if $\tau$ is sufficiently large.
                \item 
                There exists a neighborhood $U$ of $\mathcal D$ such that
                $\mathcal H(\tau,x,t) = 0$ for $x\in U$.
                \item
                $\mathcal H(\tau,x,t) = H(x,t)$ if $\tau$ is sufficiently small, and 
                $\mathcal H(\tau,x,t) = H'(x,t)$ if $\tau$ is sufficiently large.
        \end{enumerate}
        If these conditions are satisfied, then we say $(\mathcal{JJ},\mathcal H)$ is a {\it homotopy from $(\mathcal J,H)$ to 
        $(\mathcal J',H')$.}
\end{conds}
\begin{definition}
Let $\gamma_- \in R(L_0,L_1;H)$ 
and $\gamma_+ \in R(L_0,L_1;H')$.
We define the moduli space:
$$
\mathcal M_{k_1,k_0}^{\rm reg}(L_1,L_0;\gamma_-,\gamma_+;\beta;
\mathcal{JJ},\mathcal H;{\bf m})
$$
as the set of all objects 
$((\Sigma,\vec z_0,\vec z_1,\vec w),u)$ satisfying the same condition as 
\cite[Definition 3.80]{part1:top} with the following two differences:
\begin{enumerate}
\item
We require
$u$ satisfies 
\begin{equation}
\frac{\partial u}{\partial \tau}
+ J_{\tau,t} \left(
\frac{\partial u}{\partial t} - X_{H_{\tau,t}}
\right)
= 0.
\end{equation}
Here $H_{\tau,t}(x) = H(\tau,t,x)$.
\item
We require
\begin{equation}\label{form34342revrev}
\lim_{\tau \to -\infty} u(\tau,t) = \gamma_-(t),
\qquad
\lim_{\tau \to +\infty} u(\tau,t) = \gamma_+(t).
\end{equation}
\end{enumerate}
Here $\beta$ is a homology class of a map $u$ satisfying
(\ref{form34342revrev}) and $u(\tau,0) \in L_0$, $u(\tau,1) \in L_1$.
\end{definition}
We define:
$$
\mathcal M_{k_1,k_0}^{\rm reg}(L_1,L_0;\gamma_-,\gamma_+;\beta;
\mathcal{JJ},\mathcal H)
=
\mathcal M_{k_1,k_0}^{\rm reg}(L_1,L_0;\gamma_-,\gamma_+;\beta;
\mathcal{JJ},\mathcal H;\emptyset).
$$

        The space
        $\mathcal M_{k_1,k_0}^{\rm reg}(L_1,L_0;\gamma_-,\gamma_+;\beta;
        \mathcal{JJ},\mathcal H)$
        has a compactification, which we denote by 
        $\mathcal M^{\rm RGW}_{k_1,k_0}(L_1,L_0;\gamma_-,\gamma_+;\beta;\mathcal{JJ},\mathcal H)$.
        It carries an oriented Kuranishi structure with 
        boundary.

Using this moduli space
we define 
$$
\Phi_{\mathcal{JJ},\mathcal H}
: CF(L_1,L_0;H,\mathcal J;o) \to
CF(L_1,L_0;H',\mathcal J';o)
$$
by the formula:
$$
\Phi_{\mathcal{JJ},\mathcal H}([\gamma_-])
=
\sum
[(\mathcal M^{\rm RGW}_{0,0}(L_1,L_0;\gamma_-,\gamma_+;\beta;\mathcal H,\mathcal{JJ}),
\widehat{\frak s}^n)] [\gamma_+].
$$
Here the sum is taken over $\gamma_+$ and $\beta$ such that
$$
\dim \mathcal M^{\rm RGW}_{0,0}(L_1,L_0;\gamma_-,\gamma_+;\beta;\mathcal H, \mathcal{JJ})
= 0.
$$
Then the description of its boundary implies the next formula 
in the same way as in many other cases of 
Floer theory (see, for example, \cite{fooobook,fooonewbook,foooHam}).

\begin{equation}\label{form421421}
		\partial\circ \Phi_{\mathcal{JJ},\mathcal H} = \Phi_{\mathcal{JJ},\mathcal H}\circ \partial.
	\end{equation}

\subsubsection{Construction of Chain Homotopies}

Let $\mathcal J, \mathcal J', H, H'$ be as in the beginning of 
Subsubsection \ref{subsub:conmap}.
Suppose $(\mathcal{JJ},\mathcal H)$ and $(\mathcal{JJ}',\mathcal H')$
are both homotopies from $(\mathcal J,H)$ to $(\mathcal J',H')$.

A homotopy from $(\mathcal{JJ},\mathcal H)$ to $(\mathcal{JJ}',\mathcal H')$
is given by a pair 
\[\mathcal{JJJ} = \{J_{\tau,t,s} \mid \tau \in \bbR, t,s \in [0,1]\},\hspace{1cm}\mathcal{HH} : \bbR \times (X \setminus \mathcal D) \times [0,1]^2 \to \bbR,\]
such that $J_{\tau,t,0} = J_{\tau,t}$, $J_{\tau,t,1} = J'_{\tau,t}$,
$\mathcal{HH}(\tau,x,t,0) = \mathcal H(\tau,x,t)$, 
$\mathcal{HH}(\tau,x,t,1) = \mathcal H'(\tau,x,t)$
and, for each $s \in [0,1]$, the pair $(\mathcal {JJ}_s,\mathcal H_s)$, 
        defined as $\mathcal {JJ}_s = \{J_{\tau,t,s}\}$, $\mathcal H_s(\tau,x,t) = 
        \mathcal{HH}(\tau,x,t,s)$ is a homotopy from 
        $(\mathcal{J}, H)$ to $(\mathcal{J}', H')$.
        We define
        \begin{equation}
        \aligned
        &\mathcal M_{k_1,k_0}^{\rm RGW}(L_1,L_0;\gamma_-,\gamma_+;\beta;
        \mathcal{HH},\mathcal{JJJ};{\bf m}) \\
        &=
        \bigcup_{s \in [0,1]}
        \mathcal M_{k_1,k_0}^{\rm RGW}(L_1,L_0;\gamma_-,\gamma_+;\beta;
        \mathcal H_s,\mathcal{JJ}_s;{\bf m})
        \times \{s\}.
        \endaligned
        \end{equation}

This space carries a Kuranishi structure with corners.
We define 
$$
\frak H_{\mathcal{JJJ},\mathcal{HH}}
: CF(L_1,L_0;H,\mathcal J;o) \to
CF(L_1,L_0;H',\mathcal J';o)
$$
by the  formula:
$$
\frak H_{\mathcal{JJJ},\mathcal{HH}}
([\gamma_-])
=
\sum
[(\mathcal M^{\rm RGW}_{0,0}(L_1,L_0;\gamma_-,\gamma_+;\beta;\mathcal{HH},\mathcal{JJJ}),
\widehat{\frak s}^n)] [\gamma_+].
$$
Here the sum is taken over $\gamma_+$ and $\beta$ such that
the relevant moduli space is zero dimensional. By studying the boundary of those moduli spaces we have
\begin{equation}
\partial\circ \frak H_{\mathcal{JJJ},\mathcal{HH}}
+ \frak H_{\mathcal{JJJ},\mathcal{HH}}\circ \partial
= \Phi_{\mathcal{JJ}',\mathcal H'}-\Phi_{\mathcal{JJ},\mathcal H}.
\end{equation}

\subsubsection{Composition of Continuation Maps}

For $i=1,2,3$, let $\mathcal J^{(i)}, H^{(i)}$  be as in the beginning of 
Subsubsection \ref{subsub:conmap}.
Suppose $(\mathcal{JJ}^{ij},\mathcal H^{ij})$ is a homotopy from $(\mathcal J^{(i)},H^{(i)})$ to $(\mathcal J^{(j)},H^{(j)})$ for $(ij) = (12)$, $(23)$ and $(13)$.

\begin{lemma}\label{lem43333}
	We assume $L_1,L_0$ and $o \in \pi_0(\Omega(L_1,L_0))$ satisfy Condition \ref{cond420}.
	Then the composition $\Phi_{\mathcal{JJ}^{23},\mathcal H^{23}} \circ \Phi_{\mathcal{JJ}^{12},\mathcal H^{12}}$
	is chain homotopic to $\Phi_{\mathcal{JJ}^{13},\mathcal H^{13}}$.
	The same holds if we assume $\frak{PO}(L_1) = \frak{PO}(L_0)$ instead of Condition \ref{cond420} {\rm (1)}.
\end{lemma}
\begin{proof}
The proof uses the following family of Hamiltonians:
   	 \begin{equation}
                \mathcal{H}_\rho(\tau,x,t)
                = 
                \begin{cases}
                \mathcal{H}^{12}(\tau+T_0+\rho,x,t)  &\tau \le 0, \\
                \mathcal{H}^{23}(\tau-T_0-\rho,x,t)  &\tau \ge 0,
                \end{cases}
\end{equation}
Since it is similar to standard argument of similar Floer theory we omit it.
        \end{proof}
        \begin{proof}[Proof of Proposition \ref{prop42323}]
        Let $(H,\mathcal J)$ and $(H',\mathcal J')$ be as in Proposition \ref{prop42323}.
        We take a homotopy $(\mathcal H,\mathcal{JJ})$ from  $(H,\mathcal J)$ to $(H',\mathcal J')$
        and a homotopy $(\mathcal H',\mathcal{JJ}')$ from  $(H',\mathcal J')$ to $(H,\mathcal J)$.
        They induce chain maps
        $$
        \Phi_{\mathcal{JJ},\mathcal H}
        : CF(L_1,L_0;H,\mathcal J;o) \to
        CF(L_1,L_0;H',\mathcal J';o),
        $$
        and 
        $$
        \Phi_{\mathcal{JJ}',\mathcal H'}
        : CF(L_1,L_0;H',\mathcal J';o) \to
        CF(L_1,L_0;H,\mathcal J;o).
        $$
        It suffices to show that the compositions 
        $\Phi_{\mathcal{JJ}',\mathcal H'}\circ\Phi_{\mathcal{JJ},\mathcal H}$
        and 
        $\Phi_{\mathcal{JJ},\mathcal H}\circ\Phi_{\mathcal{JJ}',\mathcal H'}$
        are homotopic to identity.
        We take a trivial homotopy $(\mathcal H^0,\mathcal {JJ}^0)$ from $(H,\mathcal J)$ to $(H,\mathcal J)$,
        namely $\mathcal H^0_{\tau,t} \equiv H_t$ and $\mathcal {JJ}^0_{\tau,t} = J_{t}$.
        By Lemma \ref{lem43333}, it suffices to show that 
        this trivial homotopy induces the identity map.
        In the case 
        of the trivial homotopy, the space
        $\mathcal M^{\rm RGW}_{0,0}(L_1,L_0;\gamma_-,\gamma_+;\beta;\mathcal{JJ},\mathcal H)$
        has an $\bbR$ action induced by translations in the $\tau$ direction.
        This action is free unless $\beta=0$. 
        Therefore, by taking a multi-valued perturbation 
        invariant with respect to this $\bbR$ action, 
        we may assume that the number 
        $$
        [\mathcal M^{\rm RGW}_{0,0}(L_1,L_0;\gamma_-,\gamma_+;\beta;\mathcal H,\mathcal{JJ}),
        \widehat{\frak s}^n]
        $$
        is nonzero only if $\beta=0$.
        For $\beta=0$, the moduli space consists of one point.
        Thus the chain map induced by the trivial homotopy is an isomorphism. 
        This completes the proof of Proposition \ref{prop42323}.
\end{proof}
\subsection{Floer Homology with Coefficients in Novikov Rings.}
\label{subsec:novikv}

The purpose of this subsection is to remove Condition \ref{cond420} (2)
and construct Floer homology over Novikov ring.
Note that in Lemma \ref{lem460} we assumed Condition \ref{cond420} (2).
So when we remove this condition, Floer's boundary operator
as in Definition \ref{defn2.10} may not be defined.
This is the reason we need to use Novikov ring $\Lambda_0^{\bbQ}$.
Here $\Lambda_0^{\bbQ}$ consists of a formal sum $\sum a_i T^{\lambda_i}$
such that $a_i \in \bbQ$, $\lambda_i \ge 0$ and $\lim_{i\to \infty} \lambda_i = + \infty$.
($T$ is a formal variable.) We denote by $\Lambda^{\bbQ}$ its field of fractions which 
consists of similar formal sums but $\lambda_i$ can be negative.

We still assume $L_1$, $L_0$ are monotone in $X \setminus \mathcal D$. 
\begin{definition}
	Let $H : X \setminus \mathcal D \to \bbR$ be a compactly supported function.
        Let $\gamma_{\pm} \in R(L_0,L_1;H;o)$ and 
        $\beta \in \Pi_2(X;L_1,L_0;\gamma_-,\gamma_+)$.
        We define
        $$
        C^H(\gamma_{\pm}) = H(\gamma_{\pm}(1))= H(\gamma_{\pm}(0))
        $$
        and
        \begin{equation}\label{form428}
        \omega_H(\beta) =
        \int_{\bbR \times [0,1]} u^*\omega + C^H(\gamma_{+}) - C^H(\gamma_{-}).
        \end{equation}
\end{definition}
\begin{lemma}
	If $\mathcal M^{\rm RGW}_{k_1,k_0}(L_1,L_0;\gamma_-,\gamma_+;\beta;H,\mathcal J)$
	is nonempty, then $\omega_H(\beta) \ge 0$. It is zero only when $\gamma_-=\gamma_+$.
\end{lemma}
\begin{proof}
If $u : \bbR \times [0,1]$ satisfies (\ref{Flequation}) and (\ref{form34342rev}), then we can easily show\footnote{
Here we use the convention that $dH(\cdot) = \omega(X_H,\cdot)$ and $g_X(V,W) = \frac{\omega(V,JW)+\omega(W,JV)}{2}$.}:
$$
\int_{\bbR \times [0,1]} u^*\omega + C^H(\gamma_{+}) - C^H(\gamma_{-})
= 
\int_{\bbR \times [0,1]} \left\Vert\frac{\partial u}{\partial \tau}\right\Vert ^2
dtd\tau.\qedhere
$$ 
\end{proof}
We can use (a variant of) Gromov compactness theorem to show that for each 
$E$ there exists only a finite number of $\beta$ such that the moduli space
$\mathcal M^{\rm RGW}_{k_1,k_0}(L_1,L_0;\gamma_-,\gamma_+;\beta;H,\mathcal J)$ 
is nonempty and $\omega_H(\beta) \le E$.

\begin{definition}\label{defn439439}
        Let $L_1,L_0$ be monotone Lagrangian submanifolds of $X \setminus \mathcal D$
        satisfying $\frak{PO}^{L_1}(1) = \frak{PO}^{L_0}(1)$.
        Fix $E$ and choose a system of multi-valued perturbations
        $\{\frak s^n\}$ for 
        $\mathcal M^{\rm RGW}_{k_1,k_0}(L_1,L_0;\gamma_-,\gamma_+;\beta;H,\mathcal J)$ with 
        $\omega_H(\beta) \le E$. We define:
        $$
        \langle \partial_{\beta}[\gamma_-],[\gamma_+]\rangle
        = [\mathcal M^{\rm RGW}_{0,0}(L_1,L_0;\gamma_-,\gamma_+;\beta),\frak s^n] \in \bbQ.
        $$
        for a large enough value of $n$. Then we consider the sum:
        \begin{equation}\label{form429429}
        \partial = \sum T^{\omega_H(\beta)}\partial_{\beta},
        \end{equation}
        which is a $\Lambda_0^{\bbQ}$-linear map from 
        $CF(L_1,L_0;H,\mathcal J;o) \otimes_{\bbQ} \Lambda_0^{\bbQ}$ to itself of degree $1$.
\end{definition}
In the same way as  Theorems \ref{lem48} and \ref{lem413}
we can prove:
\begin{equation}
		\partial \circ \partial \equiv \widehat{\frak{PO}}_{L_1}(1) - \widehat{\frak{PO}}_{L_0}(1)
		\mod T^E.
\end{equation}
where
\begin{equation}\nonumber
\widehat{\frak{PO}}_L(1) = T^{\omega([\alpha])}[\mathcal M^{\rm RGW}(L;\alpha;p),\widehat{\frak s}^n]
\in \Lambda_0^{\bbQ}.
\end{equation}
(Compare with (\ref{eq:defPO}).)
\par
We thus obtain a chain complex over $\Lambda^{\bbQ}_0/T^E\Lambda^{\bbQ}_0$ under the assumption 
$\widehat{\frak{PO}}_{L_1}(1) = \widehat{\frak{PO}}_{L_2}(1)$.
We 
can show this chain complex is independent of the choices of Kuranishi 
structure and multi-valued perturbation up to chain homotopy equivalence in the same  
way as Subsection \ref{subsec:welldef}.
We can then use a `homotopy inductive limit argument'
similar to \cite{fooobook2,fooonewbook} to obtain a chain complex over the Novikov ring $\Lambda^{\bbQ}_0$.
Its cohomology is by definition 
$HF(L_1,L_0;o;\Lambda_0^{\bbQ};X \setminus \mathcal D)$.
Then in the same way as the proof of Proposition \ref{prop42323}
we can show
  $$
        HF(L_1,L_0;H;o;\Lambda_0^{\bbQ};X \setminus \mathcal D) \otimes_{\Lambda_0^{\bbQ}}\Lambda^{\bbQ}        
        \cong 
        HF(L_1,L_0;H';o;\Lambda_0^{\bbQ};X \setminus \mathcal D) \otimes_{\Lambda_0^{\bbQ}}\Lambda^{\bbQ}.
$$
The argument of this part is the same as \cite[Section 5.3]{fooobook2} and so is omitted.

\subsection{A Spectral Sequence for 
$HF(L,\psi_H(L);X\setminus \mathcal D)$}
\label{thecaseofLandL}

In this subsection we consider the case $L_0 = \psi_H(L_1)$
and prove part (4) of Theorem \ref{mainthm-part3}.
More precisely, we prove the following.
\begin{prop}\label{prop449}
        Suppose $L \subset X \setminus \mathcal D$ is a compact, connected, monotone and spin Lagrangian.
        Let $H : (X \setminus \mathcal D) \times [0,1] \to \bbR$ be a 
        compactly supported time dependent Hamiltonian which 
        generates a Hamiltonian diffeomorphism $\psi_H : X \to X$.
        We assume $L_1 = L$ and $L_0 = \psi_H(L)$ intersect transversely.
        Then there exists $o_0 \in \pi_0(\Omega(L_0,L_1))$ with 
        the following properties.
        \begin{enumerate}
        \item If $o \ne o_0$, then 
        the Floer homology group $HF(L_1,L_0;o;\Lambda^{\bbQ};X \setminus \mathcal D)$ is zero.
        \item
        $o_0$ satisfies Condition \ref{cond420} {\rm (2)} and so we can define the Floer homology group 
        $HF(L_1,L_0;o_0;{\bbQ};X \setminus \mathcal D)$.
        \item
        There exists a spectral sequence whose $E_2$ page is $H_*(L;\bbQ)$ and 
        which converges to $HF(L_1,L_0;o_0;{\bbQ};X \setminus \mathcal D)$. 
        \end{enumerate}
\end{prop}
\begin{proof}
	Because of the invariance of Floer homology with respect to Hamiltonian isotopies, it suffices to prove this proposition for 
	a specific choice of $H$. There is a symplectomorphism $\Phi$ from a neighborhood $\mathcal U$ of $L$ in $X$, disjoint from 
	a neighborhood of $\mathcal D$, to a neighborhood of the zero section in $T^*L$. 
	Using this symplectomorphism, we define $H$ to be the function
	\[
	  H(x,t):= f\circ \pi_L(x)\cdot \chi(|x|),
	\]
	where $f:L\to \bbR$ is a Morse-Smale function with respect to the metric induced by $J_0$ and $\omega$, $\pi_L:T^*L\to L$ is the projection map, $|x|$ is the fiber-wise norm of 
	an element of $T^*L$ and $\chi:\bbR_{\geq 0} \to [0,1]$ is a function which is equal to $1$ in neighborhood of $0$ and is equal 
	to $0$ for large enough values in $\bbR_{\geq 0}$. An appropriate choice of $\chi$ allows us to regard the above function as a
	function on $X$. We also assume that the $C^2$ norm of $f$ is very small. 
	
	The assumption on the $C^2$ norm of $f$ implies that $L_0\subset\mathcal U$ and $\Phi(L)$ can be identified with the graph of $df$ in $T^*L$. In particular, the 
	elements of $L_0\cap L_1$ can be identified with the critical points of $f$. We let $\mu_f$ denote the function on $L_0\cap L_1$ which associates to each intersection 
	point of $L_0$ and $L_1$ the Morse index of the corresponding critical point.
	Let $o_0$ be the element of $\pi_0(\Omega(L_0,L_1))$ that is mapped by $(\psi_H)_*$ to the connected component of the constant maps of $\Omega(L,L)$.
	Then it is clear that all the elements of $L_0\cap L_1$ represent $o_0$. In particular, $HF(L_1,L_0;o;\Lambda^{\bbQ};X \setminus \mathcal D)$ is trivial unless $o=o_0$.
	
	In order to prove (2), it suffices to show that there is a positive constant $c$ such that if $u : S^1 \times [0,1] \to X$ is a 
	map representing an element of $\pi_1(\Omega(L,L;o_0))$, then $\omega(u)=c\mu(u)$. We may assume that $u(0,t)$ is independent of $t$.
	(This is because $o_0$ is the component containing the constant paths.) 
	Therefore, $u$ is induced by a map $(D^2,\partial D^2) \to (X,L)$.
	Thus Condition \ref{cond420} (2) follows from the monotonicity of $L$.
	
	Suppose $p,q\in L_0\cap L_1$ and $u:\bbR\times [0,1] \to X$ is a map representing an element of $\mathcal M^{\rm RGW}_{0,0}(L_1,L_0;p,q;\beta;\mathcal J)$.
	If the image of this map is contained in $\mathcal U$, then it is a standard result that $\omega(\beta)=f(p)-f(q)$ and $\mu(\beta)=\mu_f(p)-\mu_f(q)$. 
	It is shown in \cite{Flo89} that $u(1,\cdot)$ defines a downward gradient flow line of the map $f$. Moreover, this gives a correspondence between 
	the moduli space $\mathcal M^{\rm RGW}_{0,0}(L_1,L_0;p,q;\beta;\mathcal J)$ and the moduli space of 
	unparametrized downward gradient flow lines from $p$ to $q$.
	(Here we need the assumption that $f$ is small in the $C^2$ norm.) 
	Next, let $u$ be a map whose image is not contained in $\mathcal U$. Using Gromov compactness theorem and the assumption on the size of $f$, 
	it is easy to show that 
	there is a constant $e$, independent of $u$, such that $\omega(\beta)\geq e$.
	Lemma \ref{index-energy-strip} implies that there is an even integer $m$ such that
	\[
	  \mu(\beta)=\mu_f(p)-\mu_f(q)+m,\hspace{2cm}\omega(\beta)=f(p)-f(q)+c\cdot m,
	\]
	where $c$ is the monotonicity constant of $L$. Since the $C^0$ norm of $f$ is small, the constant $m$ has to be positive. 
	
	The previous paragraph implies that the coefficient of $q$ in $\partial p$ is non-zero only if:
	\[
	  \mu_f(q)-\mu_f(p)\geq -1
	\]
	and if the equality holds, then $\langle \partial p,q\rangle$ is equal to the number of unparametrized downward gradient flow lines from $p$ to $q$.
	(To be more precise, we need to arrange for a multi-valued perturbation produced by Theorem \ref{prop61111} which is trivial in the case that 
	$\mu(\beta)=\mu_f(p)-\mu_f(q)=1$. Since $f$ is a Morse-Smale function, it is easy to see from the proof of Theorem \ref{prop61111} 
	that it is possible to arrange for such a multi-valued perturbation.)
	Therefore, we have
	\[
	  CF(L_1,L_0;\bbQ;o_0)=\bigoplus_d C_d,\hspace{1cm}\partial=\sum_{k=0}^\infty \partial_k,
	\] 
	where $C_d$ is generated by the elements of $L_0\cap L_1$ whose Morse index is equal to $d$, and the degree of $\partial_k$ with respect to this grading
	is equal to $-1+2k$.
	Now we can conclude Part (3) of the proposition from Lemma \ref{lem452}.
\end{proof}

\begin{lemma}\label{lem452}
        Let $C$ be a finite dimensional graded 
        vector space and 
        \begin{equation}\label{436436}
        \widehat \partial = \sum_{k=0}^{\infty} \partial_k
        \end{equation}
        be a liner map $C \to C$ such that $\partial_k$ has degree $-1+2k$ 
        and $\widehat \partial\circ \widehat \partial =0$.
        Then there exists a spectral sequence whose $E_2$ page 
        is $H(C,\partial_0)$ and which converges to $H(C;\widehat\partial)$.
\end{lemma}
We remark that the sum in (\ref{436436}) is actually a finite sum. We also remark that $\partial_0 \circ \partial_0 = 0$  follows 
from $\widehat \partial\circ \widehat \partial =0$.
\begin{proof}
         Let 
        $$
        (\frak F_{\ell}C)_d
        =
        \begin{cases}
        C_d     &\text{if $d >\ell$,} \\
        {\rm Im}\,\partial_0 \cap C_d     &\text{if $d =\ell$,} \\
        0 &\text{if $d <\ell$.}
        \end{cases}
        $$
        Here $C_d$ is the degree $d$ part of $C$. Since $\partial_k = 0$ for $k < 0$, $\frak F_{\ell}C$ defines a sub-complex of $C$.
        Let: 
	\[
 	  \left(\frac{\frak F_{\ell}C}{\frak F_{\ell+1}C},\overline\partial\right)
	\]
        be the graded complex of the filtration determined by the sub-complexes $\frak F_{\ell}C$. 
        In particular, there is a spectral sequence with the $E_2$ page:
        \begin{equation}\label{graded-subcx}
        		\bigoplus_{\ell} H\left(\frac{\frak F_{\ell}C}{\frak F_{\ell+1}C},\overline{\widehat\partial}\right)
	\end{equation}	       
        that converges to $H(C;\widehat\partial)$. It is easy to see that \eqref{graded-subcx} is isomorphic to $H(C,\partial_0)$
\end{proof}

\section{Possible Generalizations}
\label{secgeneralization}

In this section, we describe various possible generalizations of the results of this paper. We are planing to come back to some of these directions elsewhere.
In particular, we believe that modifications of the method of the present series of papers can be applied to prove many of the conjectures stated here.

\subsection{$A_{\infty}$-Structures}

Suppose $(X,\omega)$ and $\mathcal D$ are as in the beginning of the introduction. Let $L \subset X \setminus \mathcal D$ be a relatively spin compact Lagrangian submanifold of $X \setminus \mathcal D$.

\begin{conjecture}\label{conj71}
        There exists a curved filtered $A_{\infty}$-algebra, denoted by $(H^*(L;\Lambda^{\bbQ}_0),\{\frak m_k \mid k=0,1,2\dots \})$, on 
        the cohomology group $H^*(L;\Lambda^{\bbQ}_0)$ with Novikov ring coefficients.
        This $A_{\infty}$-algebra is independent of various choices such as almost complex structures up to homotopy equivalence.
        If $L$ is monotone, then by putting $T=1$, the element $
        \frak m_0(1) \in H^{0}(L;\Lambda^{\bbQ}_0)$ becomes $\frak{PO}_L(1) 1,
        $
        where $\frak{PO}_L$ is as in Definition \ref{defn41414} and $1 \in H^0(L;\Lambda^{\bbQ}_0)$ is the unit.
\end{conjecture}
See, for example, \cite[Section 2]{fu2017} for the notion of curved filtered $A_{\infty}$-categories and $A_{\infty}$-algebras.

Let $\Lambda^{\bbQ}_+$ be the subring of $\Lambda^{\bbQ}_0$ 
consisting of the sums $\sum a_i T^{\lambda_i}$ with $\lambda_i > 0$.
We call $b \in H^{\rm even}(L;\Lambda^{\bbQ}_+)$ a {\it weak bounding
cochain} if 
\begin{equation}
\sum_{k=0}^{\infty} \frak m_k(b,\dots,b) \in  H^{0}(L;\Lambda^{\bbQ}_0)
\end{equation}
and define $\frak{PO}(L;b) \in \Lambda_+^{\bbQ}$ by 
$$
\sum_{k=0}^{\infty} \frak m_k(b,\dots,b) =  \frak{PO}(L;b) 1.
$$

Let $\frak L = \{ L_i \mid i=1,\dots, N\}$ be a finite set of relatively spin compact Lagrangian submanifolds of $X \setminus \mathcal D$ such that $L_i$ is transversal to $L_j$ for $i\ne j$.

\begin{conjecture}\label{conj72}
	For each $c \in \Lambda_+^{\bbQ}$, there exists a filtered $A_{\infty}$-category $\frak{Fuks}(X\setminus \mathcal D;c;\frak L)$
	with the following properties:
        \begin{enumerate}
		\item An object of $\frak{Fuks}(X\setminus \mathcal D;c;\frak L)$ is a pair $(L_i,b)$ where $L_i \in \frak L$ and $b$ is a 
			weak bounding cochain of $L_i$ with $\frak{PO}(L_i;b) = c$.
		\item The set of morphisms $CF((L_i,b),(L_i,b))$ is an $A_{\infty}$-algebra which is not curved. The $A_\infty$ operations of this 
		$A_\infty$-algebra are given by
			\[
			  \frak m'_k(x_1,\dots,x_k)= \sum_{\ell_0,\dots,\ell_k = 0}^{\infty}
			  \frak m_{k+\sum\ell_i} (b^{\ell_0},x_1,b^{\ell_1},\dots,b^{\ell_{k-1}},x_k,b^{\ell_k}).
			\]
       			 where $\frak m_*$ is given in Conjecture \ref{conj71}.
        \item
        If $L_i$ is monotone and $c = \frak{PO}(L_i)$ for $i=1,2$, then $(L_i,0)$
        is an object of $\frak{Fuks}(X\setminus \mathcal D;c;\frak L)$ for $i=1,2$.
        The homology group $H_*(CF((L_1,0),(L_2,0));\frak m_1)$
        agrees with the Floer cohomology of Theorem \ref{mainthm-part3}.
        \end{enumerate}
\end{conjecture}

Conjecture \ref{conj71} can be proved using (the system of) Kuranishi structures on the moduli spaces $\mathcal M_{k+1}^{\rm RGW}(L;\beta)$ produced in \cite{part2:kura}. However, in this situation, we need to study moduli spaces whose virtual dimension is higher than one. If $\Lambda^{\bbQ}_0$ is the coefficient ring, then we may use singular homology as a chain model for the cohomology of $L$ analogous to \cite[Chapter 7]{fooobook2}. If $\Lambda^{\bbR}_0$ is the coefficient ring, then we may use de Rham cohomology as a chain model for the cohomology of $L$ analogous to \cite[Chapters 21 and 22]{fooonewbook}. To use de Rham cohomology, we also need smoothness of the coordinate change maps of Kuranishi structures. The moduli spaces of pseudo-holomorphic polygons and their RGW-compactifications would be also the main ingredient to prove Conjecture \ref{conj72}. See, for example, \cite[Section 3]{fu2017}.

\subsection{Normal Crossing Divisor}
\label{subsec;NCdivors}

Let $(X,\omega)$ be a compact symplectic manifold. Let
\begin{equation}\label{formula72}
	\mathcal D = \bigcup_{i=1}^m \mathcal D_i.
\end{equation}
For $I = \{i_1,\dots,i_{\vert I\vert}\} \subseteq \{1,\dots,m\}$, we define
\begin{equation}\label{DefD_I}
\mathcal D_I = \bigcap_{j=1}^{\vert I\vert} \mathcal D_{i_j}.
\end{equation}
Then $\mathcal D$ is a normal crossing divisor if it satisfies the following properties (see \cite{FMZ:NC}):
\begin{enumerate}
	\item[(mc.1)] Each $\mathcal D_i$ is a codimension 2 smooth symplectic submanifold of $(X,\omega)$. 
	\item[(mc.2)]  The intersection (\ref{DefD_I}) is transversal.
	\item[(mc.3)] The restriction of $\omega$ to $\mathcal D_I$ defines a symplectic structure on it.
\end{enumerate}

For a compact and relatively spin Lagrangian submanifold $L$ in $X \setminus \mathcal D$, we can define the monotonicity of $L$ in $X \setminus \mathcal D$ and minimal Maslov number of $L \subset X \setminus \mathcal D$ in the same way as in the case that $\mathcal D$ is a smooth divisor.

We consider the homology classes $\beta \in H_2(X,L;\bbZ)$ such that $\beta \cap [\mathcal D_i] = 0$ for all $i$. Then it is reasonable to expect that there is an RGW-type compactification $\mathcal M^{\rm reg}_{k+1}(L;\beta)$ of the moduli space of pseudo-holomorphic disks $u : (D^2,\partial D^2) \to (X,L)$
of homology class $\beta$. This compactification would be useful to address the following conjecture.

\begin{conjecture}\label{conj73}
	Theorem \ref{mainthm-part3},  Conjectures \ref{conj71} and \ref{conj72} can be generalized to this setup.
\end{conjecture}


For each $I$, there exists a local $\bbC_*^{\vert I\vert}$ action on  a neighborhood of $\mathcal D_I$.
Under the assumption (mc.1)-(mc.3), we can modify the construction of compatible almost complex structure 
$J$ in \cite[Subsection 3.2]{part1:top} so that 
on a neighborhood $U(I)$ of $\mathcal D_I$, the almost complex structure 
is invariant under the local $\bbC_*^{\vert I\vert}$ action.

The projectivization of the normal bundle $\mathcal N_{\mathcal D_{i_j}}(X)$ defines a ${\bf P}^1$-bundle on $\mathcal D_I$ for each $j$. Therefore, we can form a $({\bf P}^1)^{\vert I\vert}$-bundle on $\mathcal D_I$. (The fiber is a direct product of $\vert I\vert$ copies of ${\bf P}^1$.) Using this $({\bf P}^1)^{\vert I\vert}$-bundle for various $I$, instead of the ${\bf P}^1$ bundle ${\bf P}(\mathcal N_{\mathcal D}(X) \oplus \bbC)$, we can generalize the constructions in \cite{part1:top} to define RGW-compactification of $\mathcal M^{\rm reg}_{k+1}(L;\beta)$. The main difference is that here we need to use $m$ level functions, one level function for each of the irreducible components of $\mathcal D$. We also need to use $m$-multiplicity functions. Then one should be able to generalize the construction of \cite{part2:kura} to obtain a system of Kuranishi structures on the RGW-compactifications of $\mathcal M^{\rm reg}_{k+1}(L;\beta)$ and generalize the argument of Section \ref{sec:mainthmproof} to prove an analogue of Theorem \ref{mainthm-part3} in this setup. A corresponding compactification in the context of relative Gromov-Witten theory is developed in \cite{T}. 

\begin{remark}
	In general, a normal crossing divisor $\mathcal D$ may not have a decomposition as in \eqref{formula72} as some irreducible components may have self intersection.
	When we stratify $\mathcal D$ in a similar way, the bundle which corresponds to the above $({\bf P}^1)^{\vert I\vert}$-bundle 
	still exists. However, there may exist a nontrivial monodromy which exchange the 
	factors. In this case, several level functions may have various symmetries.
	Therefore, one needs a more careful argument to handle this case. However, we expect 
	that the same conclusion holds for such normal crossing divisors, too.
\end{remark}

\begin{remark}
	A conjecture similar to Conjecture \ref{conj73} was proposed by M. Gross in a talk in 2016.
\end{remark}

\subsection{Group Actions on $X$}

Let $(X,\omega)$ be a compact symplectic manifold on which a compact Lie group $G$ acts, preserving the symplectic structure. Let $\mathcal D \subset X$ be a submanifold as in the beginning of the introduction. We assume $\mathcal D$ is $G$-invariant and the almost complex structure $J$ is also $G$-invariant.

\begin{conjecture}\label{conj74}
	Let $L_0,L_1$ be compact, monotone and relatively spin Lagrangian submanifolds of $X \setminus \mathcal D$ which are $G$-invariant.
	We assume {\rm (a)} or {\rm (b)} of Theorem \ref{mainthm-part3} holds.
	Then there is a $G$-equivariant Lagrangian Floer cohomology group $HF^G(L_1,L_0;X\setminus \mathcal D)$,
	which is a $H^*(BG;\Lambda^{\bbR})$-module and has the following properties.
	\begin{enumerate}
		\item If $L_1$ and $L_0$ have clean intersections, then there exists a spectral sequence converging to 
			$HF^G(L_1,L_0;X\setminus \mathcal D)$ 
			such that the $E_2$ page of this spectral sequence is the $G$-equivariant cohomology group:
			\[
			  H^G(L_1 \cap L_0;\Theta \otimes \Lambda^{\bbR}).
			\]
			Here $\Theta$ is an appropriate $\bbZ_2$ local system.\footnote{Such local systems appear in the Morse-Bott case of 
			Lagrangian Floer theory. See \cite[Subsection 8.8]{fooobook2}.}
		\item It is invariant under $G$-equivariant Hamiltonian isotopies applied to $L_0$ and $L_1$.
\end{enumerate}
	The same claim holds if $\mathcal D$ is a normal crossing divisor, namely, the case which is discussed in Subsection \ref{subsec;NCdivors}.
	There is also a filtered $A_{\infty}$-category for a finite set of  $G$-invariant, compact and 
	relatively spin Lagrangian submanifolds in $X \setminus \mathcal D$, in a similar way as in Conjectures \ref{conj71} and \ref{conj72}.
\end{conjecture}

As we mentioned in the introduction and in \cite{DF}, we can use Conjecture \ref{conj74} in the case that $X$ is a symplectic cut of the {\it extended moduli space}\footnote{See \cite{mw}.} to formulate the symplectic side of the Atiyah-Floer conjecture.

\subsection{Non-compact Lagrangians}

In this series of papers, we studied compact Lagrangians in the divisor complement. It is natural to generalize this story to the following situation.
Let $\overline L$ be a submanifold with boundary
$\partial L$ such that
$\partial L = \overline L \cap \mathcal D$.
Put $L = \overline L \setminus \partial L$.
We assume $L$ is a Lagrangian submanifold of $X \setminus \mathcal D$, and $\partial L$ is a
Lagrangian submanifold of $\mathcal D$.
We also assume that $L$ is invariant under the local $\bbR _+$ action
induced by the local $\bbC_*$ action in a neighborhood of $\mathcal D$.
Let $\overline L_i$ ($i=0,1$) be a pair of such Lagrangians.
We assume $\partial L_0$, $\partial L_1$ (resp. $L_0$, $L_1$) intersect
transversally  in $\mathcal D$ (resp. $X \setminus \mathcal D$).
We want to define a Floer homology
$HF(L_0,L_1)$ between such pairs (under appropriate assumptions).

The first observation is that the free abelian group generated by the intersection points $L_0 \cap L_1$ is not
an appropriate underlying chain group for this version of Lagrangian Floer homology, unless we assume a certain positivity condition on the divisor.
To see this, fix $p,q \in L_0 \cap L_1$ and consider a moduli space $\mathcal M(p,q;\beta)$ of pseudo-holomorphic strips with boundary on $L_0,L_1$ and asymptotic to $p,q$ that has (virtual) dimension one. One would expect that studying such moduli spaces can be used to define an appropriate version of Floer's boundary operator, which is a differential. In addition to the boundary configuration appearing in Theorem \ref{theorem30}, there is another type of codimension one boundary stratum. 

This new configuration can be described as the product of three types of moduli spaces 
\begin{equation}\label{dble-brkn}
  \mathcal M(p,x;\beta_1)\times\mathcal M(x,y;\beta_2) \times \mathcal M(y,q;\beta_3),
\end{equation}
where $x,y \in \partial L_0 \cap \partial L_1$,
$\mathcal M(p,x;\beta_1)$ (resp. $\mathcal M(y,q;\beta_3)$) is the moduli space of pseudo-holomorphic strips bounding $L_0$ and $L_1$
and asymptotic to $p$ and $x$ (resp. $y$ and $q$), and $\mathcal M(x,y;\beta_2)$ is the moduli spaces of pseudo-holomorphic strips in $\mathcal D$, bounding $\partial L_0$ and $\partial L_1$ and asymptotic to $x$ and $y$  (see Figure \ref{lastfigure}). We also make the remark that for any element $u$ of $\mathcal M(p,x;\beta_1)$, we can consider the paths $\gamma_\tau:=u(\tau,\cdot)$, and as $\tau$ goes to infinity this gives rise to a path from the ray $T_xL_0/T_x\partial L_0$ to the ray $T_xL_1/T_x\partial L_0$ inside the fiber of the normal bundle of $\mathcal D$ at the point $x$. In fact, this path goes in the {\it clockwise} direction and the set of homotopy classes of all such paths can be parametrized by non-negative integers. This can be regarded as the relative version of the fact that for our choices of almost complex structure, any intersection point of a pseudo-holomorphic with $\mathcal D$ has positive multiplicity. A similar comment applies to the elements of the moduli space $\mathcal M(y,q;\beta_3)$.

Although pseudo-holomorphic curves as in \eqref{dble-brkn} with two boundary nodes appear in the union of the codimension 2 strata of the stable map compactification, by a similar argument as in \cite{part2:kura} in the case of pseudo-holomorphic spheres, we can see that such the configuration in \eqref{dble-brkn} form codimension 1 strata in the natural adaptation of the RGW compactification to the present setup. This suggests that  we need to include the intersection points $\partial L_0 \cap \partial L_1$ in the definition of the chain complex.
More precisely, we expect that the chain complex defining $HF(L_0,L_1)$ has the form of the abelian group
\begin{equation}\label{fl-cx}
\bigoplus_{p \in L_0\cap L_1} \bbQ \oplus
\bigoplus_{x \in \partial L_0\cap \partial L_1} \bbQ[t].
\end{equation}
Here $t$ is a formal parameter, and $\bbQ[t]$ is the polynomial ring with rational coefficients. The powers of $t$ are used to keep track of the winding numbers of pseudo-holomorphic curves around points of $\partial L_0 \cap \partial L_1$, which we know that are always in the clockwise direction by the discussion of the previous paragraph.

\begin{figure}[h]
	\centering
	\includegraphics[scale=0.4]{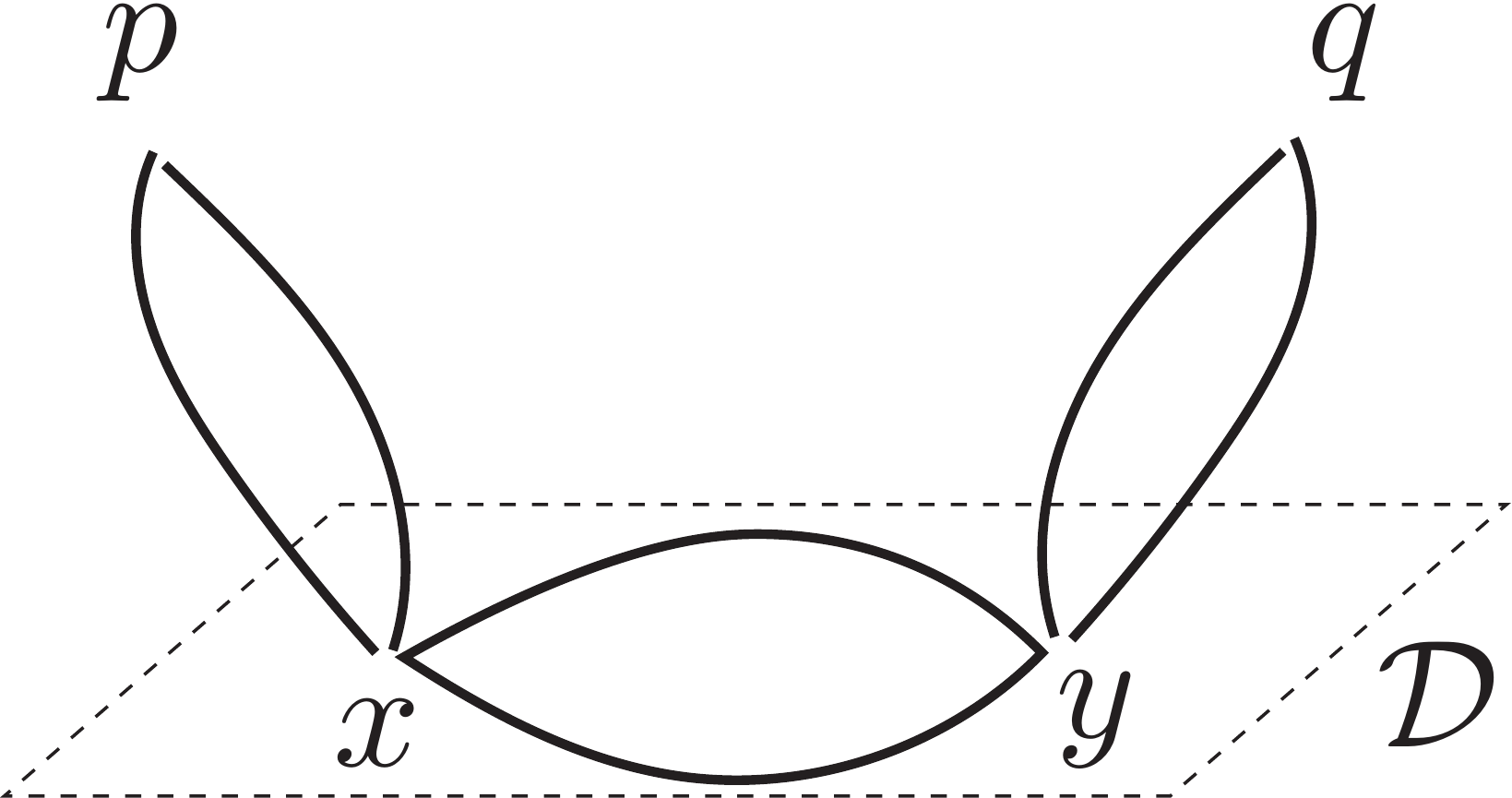}
	\caption{A new type of boundary component}
	\label{lastfigure}
\end{figure}

The structure of the above Floer complex is reminiscent of $S^1$-equivariant complexes, and one should regard each summand $\bbQ[t]$ as the cohomology of the classifying space $BS^1$ of $S^1$. In fact, one can justify the appearance of Floer complexes as in \eqref{fl-cx} using Kronheimer and Mrowka's framework for $S^1$-equivariant Floer homology in \cite{km:monopole} (where they use their setup in the context of monopole Floer homology). Lagrangian Floer homology can be regarded as the Morse homology of the {\it action functional} on the path space $\Omega(L_0,L_1)$ of Lagrangians $L_0$ and $L_1$. The partial  $\bbC_*$-action on a neighborhood $\mathfrak U$ of $\mathcal D$ induces an action of $\bbC_*$ on the subspace of $\Omega(L_0,L_1)$ given by paths from $L_0$ to $L_1$ that are contained in $\mathfrak U$. In particular, paths into $\mathcal D$ are fixed by this action, and we can take the real blow up of the space of all such paths in $\Omega(L_0,L_1)$ following \cite{km:monopole}. After lifting the action functional to this blown up space, the Floer complex of this functional has \eqref{fl-cx} as its underlying chain group. In fact, one expects that there is another flavor of the above complex $HF'(L_0,L_1)$ such that $HF(L_0,L_1)$, $HF'(L_0,L_1)$ and some twisted version of Lagrangian Floer homology of $\partial L_0$ and $\partial L_1$ form an exact triangle similar to the exact triangle that the three flavors $\widehat{HM}$, $\widecheck{HM}$ and $\overline{HM}$ of monopole Floer homology fit into \cite{km:monopole}.

A further generalization of the above direction is obtained by considering non-compact Lagrangians for normal crossing divisors as in Subsection \ref{subsec;NCdivors}. In this setup, we may consider $\overline L \subset X$, which is
a submanifold with boundary and corners, where its codimension $k$ corner consists of the union of $\overline L \cap \mathcal D_I$
with $\vert I\vert =k$.

\bibliography{references}

\begin{bibdiv}
\begin{biblist}

\bib{DF}{incollection}{
      author={Daemi, Aliakbar},
      author={Fukaya, Kenji},
       title={Atiyah-{F}loer conjecture: a formulation, a strategy of proof and
  generalizations},
        date={2018},
   booktitle={Modern geometry: a celebration of the work of {S}imon
  {D}onaldson},
      series={Proc. Sympos. Pure Math.},
      volume={99},
   publisher={Amer. Math. Soc., Providence, RI},
       pages={23\ndash 57},
      review={\MR{3838878}},
}

\bib{part1:top}{article}{
      author={Daemi, Aliakbar},
      author={Fukaya, Kenji},
       title={Monotone {L}agrangian {F}loer theory in smooth divisor
  complements: {I}},
        date={2018},
         url={https://arxiv.org/abs/1809.03409},
}

\bib{part2:kura}{misc}{
      author={Daemi, Aliakbar},
      author={Fukaya, Kenji},
       title={Monotone {L}agrangian {F}loer theory in smooth divisor
  complements: {II}},
   publisher={arXiv},
        date={2018},
         url={https://arxiv.org/abs/1809.03409},
}

\bib{Flo89}{article}{
      author={Floer, Andreas},
       title={Witten's complex and infinite-dimensional {M}orse theory},
        date={1989},
        ISSN={0022-040X},
     journal={J. Differential Geom.},
      volume={30},
      number={1},
       pages={207\ndash 221},
         url={http://projecteuclid.org/euclid.jdg/1214443291},
      review={\MR{1001276}},
}

\bib{fooobook}{book}{
      author={Fukaya, Kenji},
      author={Oh, Yong-Geun},
      author={Ohta, Hiroshi},
      author={Ono, Kaoru},
       title={{L}agrangian intersection {F}loer theory: anomaly and
  obstruction. {P}art {I}},
      series={AMS/IP Studies in Advanced Mathematics},
   publisher={American Mathematical Society, Providence, RI; International
  Press, Somerville, MA},
        date={2009},
      volume={46},
        ISBN={978-0-8218-4836-4},
         url={https://doi.org/10.1090/crmp/049/07},
      review={\MR{2553465}},
}

\bib{fooobook2}{book}{
      author={Fukaya, Kenji},
      author={Oh, Yong-Geun},
      author={Ohta, Hiroshi},
      author={Ono, Kaoru},
       title={{L}agrangian intersection {F}loer theory: anomaly and
  obstruction. {P}art {II}},
      series={AMS/IP Studies in Advanced Mathematics},
   publisher={American Mathematical Society, Providence, RI; International
  Press, Somerville, MA},
        date={2009},
      volume={46},
        ISBN={978-0-8218-4837-1},
         url={https://doi.org/10.1090/crmp/049/07},
      review={\MR{2548482}},
}

\bib{fooo:tech2-1}{misc}{
      author={Fukaya, Kenji},
      author={Oh, Yong-Geun},
      author={Ohta, Hiroshi},
      author={Ono, Kaoru},
       title={{K}uranishi structure, pseudo-holomorphic curve, and virtual
  fundamental chain: Part 1},
   publisher={arXiv},
        date={2015},
         url={https://arxiv.org/abs/1503.07631},
}

\bib{foooexp}{misc}{
      author={Fukaya, Kenji},
      author={Oh, Yong-Geun},
      author={Ohta, Hiroshi},
      author={Ono, Kaoru},
       title={Exponential decay estimates and smoothness of the moduli space of
  pseudoholomorphic curves},
   publisher={to appear in Memoirs of the AMS, arXiv},
        date={2016},
         url={https://arxiv.org/abs/1603.07026},
}

\bib{foooast}{article}{
      author={Fukaya, Kenji},
      author={Oh, Yong-Geun},
      author={Ohta, Hiroshi},
      author={Ono, Kaoru},
       title={{L}agrangian {F}loer theory and mirror symmetry on compact toric
  manifolds},
        date={2016},
        ISSN={0303-1179},
     journal={Ast\'{e}risque},
      number={376},
       pages={vi+340},
      review={\MR{3460884}},
}

\bib{fooo:const1}{incollection}{
      author={Fukaya, Kenji},
      author={Oh, Yong-Geun},
      author={Ohta, Hiroshi},
      author={Ono, Kaoru},
       title={Construction of {K}uranishi structures on the moduli spaces of
  pseudo holomorphic disks: {I}},
        date={2018},
   booktitle={Surveys in differential geometry 2017. {C}elebrating the 50th
  anniversary of the {J}ournal of {D}ifferential {G}eometry},
      series={Surv. Differ. Geom.},
      volume={22},
   publisher={Int. Press, Somerville, MA},
       pages={133\ndash 190},
      review={\MR{3838117}},
}

\bib{fooo:const2}{misc}{
      author={Fukaya, Kenji},
      author={Oh, Yong-Geun},
      author={Ohta, Hiroshi},
      author={Ono, Kaoru},
       title={Construction of kuranishi structures on the moduli spaces of
  pseudo holomorphic disks: {II}},
   publisher={arXiv},
        date={2018},
         url={https://arxiv.org/abs/1808.06106},
}

\bib{fooonewbook}{book}{
      author={Fukaya, Kenji},
      author={Oh, Yong-Geun},
      author={Ohta, Hiroshi},
      author={Ono, Kaoru},
       title={{K}uranishi structures and virtual fundamental chains},
      series={Springer Monographs in Mathematics},
   publisher={Springer, Singapore},
        date={2020},
        ISBN={978-981-15-5562-6; 978-981-15-5561-9},
         url={https://doi.org/10.1007/978-981-15-5562-6},
      review={\MR{4179586}},
}

\bib{foooHam}{article}{
      author={Fukaya, Kenji},
      author={Oh, Yong-Geun},
      author={Ohta, Hiroshi},
      author={Ono, Kaoru},
       title={Construction of a linear {K}-system in {H}amiltonian {F}loer
  theory},
        date={2022},
        ISSN={1661-7738},
     journal={J. Fixed Point Theory Appl.},
      volume={24},
      number={2},
       pages={Paper No. 39, 110},
         url={https://doi.org/10.1007/s11784-022-00960-x},
      review={\MR{4409857}},
}

\bib{fu2017}{misc}{
      author={Fukaya, Kenji},
       title={Unobstructed immersed {L}agrangian correspondence and filtered
  {$A_\infty$} functor},
   publisher={arXiv},
        date={2017},
         url={https://arxiv.org/abs/1706.02131},
}

\bib{km:monopole}{book}{
      author={Kronheimer, Peter},
      author={Mrowka, Tomasz},
       title={Monopoles and three-manifolds},
      series={New Mathematical Monographs},
   publisher={Cambridge University Press, Cambridge},
        date={2007},
      volume={10},
        ISBN={978-0-521-88022-0},
         url={https://doi.org/10.1017/CBO9780511543111},
      review={\MR{2388043}},
}

\bib{mw}{incollection}{
      author={Manolescu, Ciprian},
      author={Woodward, Christopher},
       title={{F}loer homology on the extended moduli space},
        date={2012},
   booktitle={Perspectives in analysis, geometry, and topology},
      series={Progr. Math.},
      volume={296},
   publisher={Birkh\"{a}user/Springer, New York},
       pages={283\ndash 329},
         url={https://doi.org/10.1007/978-0-8176-8277-4_13},
      review={\MR{2884041}},
}

\bib{ohbook}{book}{
      author={Oh, Yong-Geun},
       title={Symplectic topology and {F}loer homology. {V}ol. 2},
      series={New Mathematical Monographs},
   publisher={Cambridge University Press, Cambridge},
        date={2015},
      volume={29},
        ISBN={978-1-107-10967-4},
        note={Floer homology and its applications},
      review={\MR{3524783}},
}

\bib{oh}{article}{
      author={Oh, Yong-Geun},
       title={{F}loer cohomology of {L}agrangian intersections and
  pseudo-holomorphic disks. {I}},
        date={1993},
        ISSN={0010-3640},
     journal={Comm. Pure Appl. Math.},
      volume={46},
      number={7},
       pages={949\ndash 993},
         url={https://doi.org/10.1002/cpa.3160460702},
      review={\MR{1223659}},
}

\bib{T2}{article}{
      author={Tehrani, Mohammad~F.},
       title={Open {G}romov-{W}itten theory on symplectic manifolds and
  symplectic cutting},
        date={2013},
        ISSN={0001-8708},
     journal={Adv. Math.},
      volume={232},
       pages={238\ndash 270},
         url={https://doi.org/10.1016/j.aim.2012.09.015},
      review={\MR{2989982}},
}

\bib{T}{article}{
      author={Tehrani, Mohammad~F.},
       title={Pseudoholomorphic curves relative to a normal crossings
  symplectic divisor: compactification},
        date={2022aug},
     journal={Geometry {\&} Topology},
      volume={26},
      number={3},
       pages={989\ndash 1075},
}

\bib{FMZ:NC}{article}{
      author={Tehrani, Mohammad~F.},
      author={McLean, Mark},
      author={Zinger, Aleksey},
       title={Normal crossings singularities for symplectic topology},
        date={2018},
        ISSN={0001-8708},
     journal={Adv. Math.},
      volume={339},
       pages={672\ndash 748},
         url={https://doi.org/10.1016/j.aim.2018.09.035},
      review={\MR{3866910}},
}

\end{biblist}
\end{bibdiv}
\bibliographystyle{alpha.bst}
\end{document}